\documentclass[10pt, reqno]{amsart}
\usepackage{amssymb}
\usepackage{hyperref}
\usepackage{cite}

\numberwithin{equation}{section}

\newtheorem{theorem}{Theorem}[section]

\newtheorem{lemma}[theorem]{Lemma}

\newtheorem{proposition}[theorem]{Proposition}

\theoremstyle{definition}
\newtheorem{definition}[theorem]{Definition}
\newtheorem{remark}[theorem]{Remark}
\newtheorem*{remarksonT}{Remarks}

\theoremstyle{remark}

\newcommand{\R}{\mathbb{R}}
\newcommand{\C}{\mathbb{C}}

\newcommand{\N}{\mathcal{N}}

\newcommand{\B}{\mathcal{B}}

\let\Re=\undefined\DeclareMathOperator*{\Re}{Re}
\let\Im=\undefined\DeclareMathOperator*{\Im}{Im}

\newcommand{\eps}{\varepsilon}
\newcommand{\wh}{\widehat}
\newcommand{\whb}[1]{\widehat{\bar{#1}}}

\DeclareMathOperator{\supp}{supp}

\newcommand{\op}[1]{{\left\vert\kern-0.25ex\left\vert\kern-0.25ex\left\vert #1
    \right\vert\kern-0.25ex\right\vert\kern-0.25ex\right\vert}}

\newcommand{\opt}[1]{\|#1\|_{L_\xi^\infty \dot H_{\xi_2}^1}^{\frac12}\|#1\|_{L_\xi^\infty \dot H_{\xi_2}^2}^{\frac12}}
\newcommand{\opo}[1]{\|#1\|_{L_\xi^\infty \dot H_{\xi_1}^1}^{\frac12}\|#1\|_{L_\xi^\infty \dot H_{\xi_1}^2}^{\frac12}}

\newcommand{\jb}{\langle\nabla\rangle}
\newcommand{\jbb}{\langle\widetilde{\nabla}\rangle}

\newcommand{\qtq}[1]{\quad\text{#1}\quad}

\newcommand{\tunit}[1]{\tfrac{#1}{|#1|}}

\newcounter{smalllist}

\begin{document}


\title[Cubic-quintic NLS]{The initial-value problem for the cubic-quintic NLS with non-vanishing boundary conditions}

\author[R. Killip]{Rowan Killip}
\address{Department of Mathematics, UCLA, Los Angeles, USA}
\email{killip@math.ucla.edu}

\author[J. Murphy]{Jason Murphy}
\address{Department of Mathematics and Statistics, Missouri University of Science and Technology}
\email{jason.murphy@mst.edu}

\author[M. Visan]{Monica Visan}
\address{Department of Mathematics, UCLA, Los Angeles, USA}
\email{visan@math.ucla.edu}

\begin{abstract}
We consider the initial-value problem for the cubic-quintic NLS
\[
(i\partial_t+\Delta)\psi=\alpha_1 \psi-\alpha_{3}\vert \psi\vert^2 \psi+\alpha_5\vert \psi\vert^4 \psi
\]
in three spatial dimensions in the class of solutions with $|\psi(x)|\to c >0$ as $|x|\to\infty$.  Here $\alpha_1$, $\alpha_3$, $\alpha_5$ and $c$ are such that $\psi(x)\equiv c$ is an energetically stable equilibrium solution to this equation.  Normalizing the boundary condition to $\psi(x)\to 1$ as $|x|\to\infty$, we study the associated initial-value problem for $u=\psi-1$ and prove a scattering result for small initial data in a weighted Sobolev space.
\end{abstract}

\maketitle

\section{Introduction}\label{S:intro}

We study the initial-value problem for the cubic-quintic nonlinear Schr\"odinger equation (NLS) with non-vanishing boundary conditions in three space dimensions:
\begin{equation}\label{E:cq}
\begin{cases}
(i\partial_t + \Delta) \psi = \alpha_1\psi - \alpha_3|\psi|^2\psi + \alpha_5 |\psi|^4\psi, & (t,x)\in\R\times\R^3, \\
\psi(0)=\psi_0.
\end{cases}
\end{equation}
We restrict attention to parameters $\alpha_1,\alpha_3,\alpha_5\in\R$ for which the polynomial
\[
p(x)=\alpha_1 - \alpha_3 x + \alpha_5 x^2 
\]
has a positive root $c^2$ with $p'(c^2)>0$.  This guarantees the energetic stability of the constant solution $\psi\equiv c$ and we are interested in the dynamics of perturbations to this equilibrium.  Correspondingly, we impose the boundary condition
\begin{equation}\label{E:cq-boundary}
\lim_{|x|\to\infty} |\psi(t,x)| = c
\end{equation}
on our solutions to \eqref{E:cq}.

Rescaling both space-time and the values of the solution, we may reduce \eqref{E:cq} and \eqref{E:cq-boundary} to
\begin{equation}\label{E:cq2}
\begin{cases}
(i\partial_t+\Delta)\psi = (|\psi|^2-1)(\beta(|\psi|^2-1)+1)\psi, \\
\psi(0) = \psi_0, \\
\lim_{|x|\to\infty} |\psi(t,x)| = 1,
\end{cases}
\end{equation}
for some $\beta\in\R$.  This is the Hamiltonian evolution associated to the (conserved) energy
\begin{equation}\label{E:psi energy}
\int_{\R^3} \tfrac12 |\nabla\psi|^2 + \tfrac14\bigl(|\psi|^2 -1\bigr)^2 + \tfrac\beta6 (|\psi|^2 -1\bigr)^3\,dx.
\end{equation}
In particular, we see that $\psi\equiv 1$ is always a local minimum of the energy.  When $\beta<0$, the energy is unbounded below which makes the system susceptible to wave collapse for large initial perturbations; however, we shall focus here on small perturbations.  When $\beta\leq 1$, the potential energy has a local maximum at $\psi\equiv 0$, while for $\beta>1$, this energy has a second local minimum at $\psi\equiv 0$.

The model \eqref{E:cq2} describes the behavior of a localized disturbance in an infinite expanse of quantum fluid that is otherwise quiescent.  The particular case $\beta=0$ is known as the Gross--Pitaevskii equation and has received much attention.  In particular, the asymptotic stability of the equilibrium solution to the Gross--Pitaevskii equation was proved in works of Gustafson, Nakanishi, and Tsai \cite{GNT:dd, GNT:2d, GNT:3d}.  We will extend this result to the more general model \eqref{E:cq2}, while also permitting  a wider class of initial data.

The inclusion of an additional parameter $\beta$ in \eqref{E:cq2} has allowed researchers to better fit the behavior of several real physical systems and \eqref{E:cq} has been used as a model in superfluidity \cite{Ginsburg1958, Ginsburg1976}, descriptions of bosons \cite{Barashenkov} and  of defectons \cite{Pushakarov1978}, the theory of ferromagnetic and molecular chains \cite{Pushakarov1984, Pushakarov1986}, and in nuclear hydrodynamics \cite{Kartavenko}.  For comparison, the initial-value problem \eqref{E:cq} with data decaying at infinity describes the dynamics of a finite body of fluid and it was studied in \cite{KOPV:35}.

The boundary condition \eqref{E:cq-boundary} may be further simplified to the following:
\begin{equation}\label{E:cq-boundary2}
\lim_{|x|\to\infty}\psi(t,x) = 1.
\end{equation}
Indeed, finite energy functions obeying \eqref{E:cq-boundary} have a limiting phase as $|x|\to\infty$, which we can normalize to zero; furthermore, the dynamics of \eqref{E:cq} preserve the value of this phase, so that the boundary condition is independent of time.  See \cite{Gerard} for these observations in the case of the Gross--Pitaevskii equation. The fact that the boundary condition is independent of time breaks the gauge invariance of \eqref{E:cq}; in particular, we cannot use a time-dependent phase factor to remove the linear term in this equation.  The linear term ultimately leads to weaker wave-like dispersion at low frequencies, which presents a key challenge in understanding the long-time behavior of solutions.

As we study perturbations of the constant solution $\psi\equiv 1$, it is natural to introduce the function $u=u_1+iu_2$ defined via $\psi = 1 + u$. Using \eqref{E:cq2}, we have the following equation for $u$:
\begin{equation}\label{E:cqu}
(i\partial_t+\Delta) u = 2 u_1 + N(u),
\end{equation}
where the nonlinearity is given by
\begin{align*}
N(u) & = (3+4\beta)u_1^2 + u_2^2 + 2iu_1 u_2 + |u|^2 u + 4\beta[|u|^2 u_1+uu_1^2] \\
& \quad + \beta[|u|^4 + 4|u|^2 uu_1]+\beta |u|^4 u. 
\end{align*}
We may also write $N(u) = \sum_{k=2}^5 N_k(u)$, where $N_k(u)$ represents the polynomial of degree $k$ in $u_1$ and $u_2$ appearing in $N(u)$.

To put \eqref{E:cqu} into the standard framework of dispersive equations, it is convenient to diagonalize the equation.  To do this, we employ operators $U$ and $H$ defined by
\[
U=\tfrac{|\nabla|}{\jb}\qtq{and} H = |\nabla|\jb,\qtq{where}\jb:=\sqrt{2-\Delta}\qtq{and}|\nabla|=(-\Delta)^{1/2}.
\]
The function $v=Vu = u_1+iUu_2$ then satisfies the following equation
\begin{equation}\label{E:cqv}
(i\partial_t - H)v = N_v(u) = U\Re N(u) + i\Im N(u).
\end{equation}

In our previous work \cite{KMV} we considered the final-state problem for \eqref{E:cqu} (with $\beta>1$) and constructed solutions scattering to prescribed asymptotic states.  Here we consider the initial-value problem for small localized data (and $\beta\in\R$).

\begin{theorem}\label{T:main} Let $v_0\in H_x^1(\R^3)$ with $xv_0\in L_x^2$ and $(x\times\nabla)v_0\in L_x^2$.  Suppose
\begin{equation}\label{E:X0}
\| v_0 \|_{X_0} := \|\jb v_0\|_{L_x^2} +  \|(x\times \nabla) v_0\|_{L_x^2} + \|xv_0\|_{L_x^2}
\end{equation}
is sufficiently small.  Then there exists a unique global solution $v$ to \eqref{E:cqv} with $v(0)=v_0$ that scatters in both time directions, that is, there exist unique $v_\pm$ so that
\begin{equation}\label{E:main}
\|e^{itH}v(t) - v_{\pm} \|_{X_0} \longrightarrow 0 \quad\text{as $t\to\pm\infty$}.
\end{equation}
Furthermore, the convergence in \eqref{E:main} holds at a rate of $|t|^{-\eps}$ for some $\eps>0$.
\end{theorem}

We will outline the proof of this result in Section~\ref{SS:strategy}.  Before turning to that subject, we will first make a series of remarks on the nature of this theorem and its relation to prior work.  In particular, we will discuss the meaning of our theorem in the variables $u=V^{-1}u = \Re v + i U^{-1} \Im v$.  Correspondingly, we define
\[
u^{\text{lin}}_\pm(t) = V^{-1} e^{-itH} Vu_\pm \qtq{which solves} \left\{\begin{aligned} & (i\partial_t + \Delta)u^{\text{lin}}_\pm - 2 \Re u^{\text{lin}}_\pm = 0  \\
        &u^{\text{lin}}_\pm(0) = u_\pm = V^{-1} v_\pm.\end{aligned}\right.
\]

\begin{remarksonT}
\textbf{1.} The norm $X_0$ appearing here occurs as the specialization to initial data of the more general $X$-norm, defined in \eqref{energy norm}.

\textbf{2.} It was already known that our assumptions on $v_0$ imply global existence and uniqueness of solutions to \eqref{E:cqu}.  This follows from the well-posedness theory for \eqref{E:cqu} that we developed in our previous work \cite{KMV}.  Specifically, our arguments in \cite{KMV} yield large-data global well-posedness for \eqref{E:cqu} in the energy space
\[
\mathcal{E}=\{u\in \dot H_x^1: |u|^2 + 2\Re u \in L_x^2\}
\]
when $0\leq \beta\leq \frac32$ and small-data global well-posedness for the remaining values of $\beta$.
Note that $\mathcal E$ is not a vector space, but can naturally be interpreted as a complete metric space; see \cite{Gerard,KMV}.  By using \eqref{UL2} below, it is easy to see that under the hypotheses of Theorem~\ref{T:main}, the initial data $u_0=\Re v_0 + i U^{-1} \Im v_0 \in H_x^1\subset\mathcal{E}$ and so the previous analysis applies.  This reasoning also shows that \eqref{E:main} implies that
$$
\text{dist}(u(t),u^{\text{lin}}_\pm(t))\to 0 \quad\text{as $t\to\pm\infty$ in the natural metric on $\mathcal E$.}
$$

\textbf{3.} To properly recast the hypotheses on $v_0$ in terms of $u_0$, it is natural to separate the high and low frequencies using Littewood--Paley operators; see Section~\ref{S:notation} for the precise definitions.  Self evidently, we have
$$
\| P_{>1} v_0\|_{X_0} \sim \| P_{>1} u_0\|_{X_0} \sim \|\nabla  P_{>1} u_0\|_{L_x^2} +  \|(x\times \nabla) P_{>1} u_0\|_{L_x^2} + \|x P_{>1} u_0\|_{L_x^2},
$$
while setting $t=0$ in the proof of Lemma~\ref{lem:str} allows one to deduce
$$
\| P_{\leq 1} v_0 \|_{X_0} \sim \|\langle x\rangle P_{\leq 1}  \Re u_0\|_{L_x^2} + \|\langle x\rangle P_{\leq 1} \nabla \Im u_0\|_{L_x^2}.
$$
Analogously, the scattering statement can be recast as follows:
\begin{align*}
&\|u(t)-u^{\text{lin}}_\pm(t)\|_{H^1_x} +  \|(x\times \nabla) [u(t)-u^{\text{lin}}_\pm(t)]\|_{L_x^2}  \\
 &\quad{}+ \|J(t) \Re [u(t)-u^{\text{lin}}_\pm(t)] \|_{L_x^2} + \|J(t) U \Im  [u(t)-u^{\text{lin}}_\pm(t)] \|_{L^2_x} \to 0
\end{align*}
as $t\to\pm\infty$.  Here $J(t):=e^{-itH}xe^{itH}$.

As our hypotheses do not guarantee that $x \Im[V^{-1}e^{itH}Vu(t)]$ nor $x \Im u_{\pm}$ are square-integrable, one should not expect that one can remove the operator $U$ from the last term above.  On the other hand, due to the quantitative rate of convergence we obtain in Theorem~\ref{T:main}, one can show that
\begin{align*}
\lim_{t\to\pm\infty} \|x P_{>|t|^{-\delta}}\, \Im[V^{-1}e^{itH}Vu(t)-u_{\pm}]\|_{L_x^2}  = 0,
\end{align*}
for some $\delta>0$.

\textbf{4.} As the previous remark showed, our hypotheses on the real and imaginary parts of the physical field $u$ are different.  While this may seem peculiar in a Schr\"odinger-like setting, it is perfectly normal in studies on the wave equation, where one invariably adopts differing norms for the displacement and velocity components.  As the dispersion relation of the linear equation underlying \eqref{E:cq} is wave-like at low frequencies (it has a conical singularity at the origin), it is natural that our hypotheses on $\Im u$, which is analogous to displacement, and those on $\Re u$, which is analogous to velocity, differ by exactly one derivative at low frequencies.

\textbf{5.} When $\beta=0$, \eqref{E:cq2} becomes the Gross--Pitaevskii equation:
\begin{equation}\label{E:GP}
(i\partial_t+\Delta)\psi = (|\psi|^2-1)\psi, \qtq{with} \lim_{x\to\infty} \psi(t,x) = 1.
\end{equation}
The question of the long-time behavior of solutions to this equation was the central topic of a recent series of papers \cite{GNT:dd,GNT:2d,GNT:3d}.  The last of these shows scattering in the three-dimensional setting (as considered here) and much of what we do follows closely in the footsteps of that paper.  In actuality, the paper \cite{GNT:3d} has had a profound impact far beyond its original scope by codifying and popularizing (contemporaneously with \cite{GMS}) what has become known as the space-time resonance method.

More concretely, the paper \cite{GNT:3d} proves scattering for solutions to \eqref{E:GP} under the hypothesis that $\langle x\rangle v$ and $\langle x \rangle \nabla v$ belong to $L^2(\R^3)$.  These are stronger hypotheses than those employed in this paper in the sense that we have no weighted hypotheses on the derivatives beyond the angular regularity requirement $(x\times \nabla) v \in L^2$.  The nature of this improvement is most dramatic in the case of spherically symmetric data, for which $(x\times\nabla)v\equiv 0$. 

\textbf{6.} In the special case of radial data, however, the contribution of this paper is eclipsed by the more recent work \cite{GHN} on \eqref{E:GP}, which appeared while this paper was being finalized.  Nevertheless, for general data, the hypotheses of the two works are incomparable.  Specifically, \cite{GHN} has weaker hypotheses on the low-frequency portion of the solution, since they do not require weighted decay in $L^2$; on the other hand, they do place more stringent hypotheses on the high frequencies, requiring that $(x\times\nabla)P_{\geq 1} v$ belongs to $H^1(\R^3)$ and not merely $L^2(\R^3)$.

\end{remarksonT}

\subsection{Strategy of the proof}\label{SS:strategy}
The proof of Theorem~\ref{T:main} relies on a bootstrap argument.  For this, we introduce a few specific norms.  Letting $I$ denote any time interval, we first consider the Strichartz-type norm
\begin{equation}\label{def:S}
\|v\|_{S(I)}:= \| \jb v\|_{ L_t^\infty L_x^2\cap L_t^2 L_x^6}+ \|(x\times\nabla)v\|_{L_t^\infty L_x^2 \cap L_t^2 L_x^6},
\end{equation}
where the space-time norms are taken over $I\times\R^3$.

Next, we introduce the vector field $J(t)=e^{-itH}xe^{itH}$ and the `energy norm'
\begin{equation}\label{energy norm}
\|v(t)\|_X := \|\jb v(t)\|_{L_x^2} + \|J(t)v(t)\|_{L_x^2} + \|(x\times\nabla)v(t)\|_{L_x^2}.
\end{equation}
We refer to the $Jv$ component of this norm as the `weighted norm'.

For convenience, we also introduce the notation
\[
\|v\|_{Z(I)} := \|v\|_{S(I)} + \|v\|_{L_t^\infty(I;X)}.
\]

We will prove a bootstrap estimate of the form
\[
\|v\|_{Z([t_0,t_1])} \lesssim \|v(t_0)\|_{X} + \langle t_0\rangle^{-\eps}\sum_{k=2}^6 \|v\|_{Z([t_0,t_1])}^k,
\]
which, for small initial data, closes to give control in the $Z$-norm for all time.  Using global $Z$-norm bounds, we can then prove scattering by standard arguments (see Section~\ref{S:proof}).

To explain our choice of spaces, we compare \eqref{E:cqv} with related scale-invariant nonlinear Schr\"odinger equations of the form
\begin{equation}\label{nls}
(i\partial_t + \Delta) u = F(u) = |u|^p u, \quad p>0.
\end{equation}
The scaling symmetry of \eqref{nls} identifies a scaling-critical space of initial data, namely $\dot H^{s_c}(\R^3)$, where $s_c=\frac32-\frac2p$.  The regularity associated to a quintic nonlinearity is $s_c=1$, while the regularity associated to a quadratic nonlinearity is $s_c=-\frac12$.  As \eqref{E:cqv} has a quintic nonlinear term, it is natural to seek control in a Strichartz space at $H^1_x$ regularity, namely, the $S$-norm.

For Schr\"odinger equations with negative critical regularity, one seeks control of the solution in weighted spaces rather than Sobolev spaces of negative order, in which NLS is known to be illposed.  From the point of view of scaling, prescribing $|x|^s u_0\in L^2$ is like prescribing $u_0\in \dot H^{-s}$.  However, it is not expected that solutions will remain bounded in a weighted $L^2$-norm.  Indeed, this is not even case for the linear equation.  Instead, in the case of \eqref{nls}, one endeavors to prove $L^2$-bounds for the Galilean quantity $(x+2it\nabla)u$, which in turn implies decay for the solution via Klainerman--Sobolev-type inequalities.  As one can check that $x+2it\nabla = e^{it\Delta} x e^{-it\Delta}$, we see that the operator $J(t)$ defined above is the analogue of $x+2it\nabla$ for the problem \eqref{E:cqv}.

For the case of \eqref{nls}, the operator $x+2it\nabla$ essentially obeys a chain rule when applied to the nonlinearity, which greatly aids the analysis. This fact relies crucially on the gauge-invariance of the nonlinearity, i.e. the symmetry $F(e^{i\theta}u)=e^{i\theta}F(u)$.  In the case of \eqref{E:cqv}, the nonlinearity is not gauge-invariant.  Because of this, estimating $J(t)v(t)$ in $L^2$ becomes significantly more challenging.  In fact, establishing a suitable estimate for this term is the principal difficulty in this paper.

Broadly speaking, we follow an approach known as the method of space-time resonances (cf. \cite{GMS,GNT:3d}), which we now briefly describe in our setting.  Introducing the interaction variable $f(t)=e^{itH}v(t)$, we have
\[
\|J(t)v(t)\|_{L_x^2} \sim \| \nabla_{\xi} \wh{f}(t)\|_{L_\xi^2}.
\]
To estimate the latter term, we first use the Duhamel formula for $v$ (cf. \eqref{duhamel1} below) to write an integral formula for $\wh{f}$.  Because it is the quadratic terms that will ultimately be the most difficult to estimate, let us describe the technique for a single quadratic nonlinear term, say $v^2$.  For such a term, we would be led to estimate in $L_\xi^2$ the following term:
\[
\nabla_{\xi} \iint e^{isH(\xi)}\wh{v}(s,\xi-\eta)\wh{v}(s,\eta)\,d\eta\,ds = \nabla_{\xi} \iint e^{is\Phi}\wh{f}(s,\xi-\eta)\wh{f}(s,\eta)\,d\eta\,ds,
\]
where the phase $\Phi = H(\xi) - H(\xi-\eta) - H(\eta)$.  If the derivative lands on $\wh{f}(s,\xi-\eta)$, then after applying Plancherel, we recover a copy of $Jv$ and have a term that is amenable to Strichartz estimates.  We will see that such terms are not too difficult to estimate (cf. Section~\ref{S:weighted1}).

However, if the derivative lands on $e^{is\Phi}$, then we are faced with estimating the following `phase derivative' term:
\begin{equation}\label{ipd}
\iint e^{is\Phi} [s\nabla_{\xi}\Phi]\wh{f}(s,\xi-\eta)\wh{f}(s,\eta)\,d\eta\,ds.
\end{equation}
Such terms are significantly more difficult to estimate.  In particular, we need to exhibit additional decay to overcome the factor $s\nabla_\xi\Phi$.  The idea is to exploit oscillation in the phase, which as usual is achieved via integration by parts.

On `time non-resonant' regions (i.e. regions on which $\Phi \neq 0$), one can use the identity
\[
e^{is\Phi} = \tfrac{1}{i\Phi}\partial_s e^{is\Phi}
\]
to integrate by parts with respect to $s$, which either cancels the $s$ in \eqref{ipd} or leads to additional copies of the solution (which decays, due to the bootstrap assumption) via \eqref{E:cqv}.  On `space non-resonant' regions (i.e. regions on which $\nabla_\eta\Phi\neq 0$), one can use the identity
\[
e^{is\Phi} = \tfrac{1}{is|\nabla_\eta\Phi|^2}\nabla_\eta\Phi \cdot \nabla_\eta e^{is\Phi}
\]
to integrate by parts with respect to $\eta$, which cancels the $s$ in \eqref{ipd}; if the derivative then lands on a copy of $\wh{f}$, we recover another copy of $Jv$.

The strategy is then to decompose frequency space into regions of non-resonance and estimate the contribution of each region separately.  Note that each type of quadratic term (namely, $\bar v^2$, $v^2$, and $|v|^2$) will lead to a different phase in \eqref{ipd}, and hence requires its own decomposition.  For us, the most difficult case will come from the $|v|^2$ nonlinearity.  In fact, in this case we will need to rely on one other notion of non-resonance, which we have called `angular non-resonance'; this  refers to the non-vanishing of
\[
\bigl(\xi-\xi\cdot\tfrac{\eta}{|\eta|}\tfrac{\eta}{|\eta|}\bigr)\cdot\nabla_{\eta}\Phi
\]
and uses a different integration by parts identity (see Section~\ref{S:anr}).  It is only in treating the angular non-resonant region that we will need control over the angular momentum type quantity $\|(x\times\nabla) v\|_{L^2_x}$.

After integrating by parts on a non-resonant region, we can again write the resulting expressions back in terms of $v$ and $\bar{v}$.  An example of a resulting term for the $v^2$ nonlinearity (in the time non-resonant case, say) is
\[
\int e^{isH}b[v(s),v(s)]\,ds,
\]
where $b[\cdot,\cdot]$ is the bilinear multiplier with symbol $\tfrac{\nabla_\xi\Phi}{\Phi}\chi,$ with $\chi$ denoting a cutoff to the time non-resonant region.  In general, the terms we need to estimate will now involve a bilinear multiplier applied to $v$, $Jv$, $(x\times\nabla)v$, or $N_v(u)$; the particular multipliers that appear depend on whether the region is space, time, or angular non-resonant (see Sections~\ref{S:tnr}--\ref{S:anr}).  The multipliers that appear will not typically be amenable to standard bilinear estimates such as the Coifman--Meyer theorem.  Indeed, even many of the cutoff functions used to partition frequency space into appropriate non-resonant regions already suffer from this problem.  As in \cite{GNT:3d}, it is essential to curtail the associated losses by exploiting the fact that the final result is to be estimated in $L^2$.  In this paper, this key bilinear estimate appears as Proposition~\ref{prop:bilinear}, which extends an analogous result from \cite{GNT:3d}.

There is one additional twist to what we described above, namely, the fact that a direct implementation of the space-time resonance approach with the quadratic nonlinearity
\[
U[(3+4\beta)u_1^2 + u_2^2] + 2i u_1 u_2
\]
in \eqref{E:cqv} does not work.  Problems arise in quadratic frequency interactions with a small output frequency.  For the first two terms above, the factor of $U=\tfrac{|\nabla|}{\langle \nabla \rangle}$ provides cancellation at zero output frequency; however, the third term is missing this factor.  For this reason, we employ a normal form transformation to remove the quadratic term $2i u_1 u_2$ and carry out the space-time resonance method with the normal form of the equation.  See Section~\ref{S:NF} for further discussion.

Having completed the outline of the proof, it now seems pertinent to discuss more fully its relation to the works \cite{GHN,GNT:3d} on \eqref{E:GP}.  
First and foremost, our improvement on \cite{GNT:3d} is based on the discovery and subsequent exploitation of additional non-resonance phenomena in the case of high-high frequency interactions.  Most important of these is the angular non-resonance discussed in Section~\ref{SS:9.4}; however, the use of time non-resonance in Section~\ref{SS:9.3} is also novel and essential.  On top of this, we believe that our paper introduces a number of individually minor, but collectively significant, improvements to the analysis in \cite{GNT:3d}.  Two examples of this are the introduction and systematic use of the bound
\[
\| u(t) \|_{L^6_x} \lesssim \langle t\rangle ^{-\frac{7}{9}} \| v \|_Z
\]
(see Lemma~\ref{lem:decay}) and the extension of their bilinear estimate to include the end-point case (see Proposition~\ref{prop:bilinear}).  One end-result of these improvements is that we are able to present our result in a paper of comparable length while including a much more thorough exposition of the details.  In truth, the brevity of \cite{GNT:3d} rendered us unable to reconstruct their arguments in several places.

There is less of a connection between what we do here and the paper \cite{GHN}, since the key point of that paper is to exploit additional Strichartz estimates that become available for radial data or data with higher angular regularity.  The failure of such estimates in general manifests, for example, in the fact that while the traditional NLS is ill-posed in spaces of negative regularity, well-posedness can be restored by passing to radial data, or data with additional angular regularity.
 
\subsection{Organization of the paper}  In Section~\ref{S:notation}, we set up notation and collect some useful lemmas, including the bilinear estimate Proposition~\ref{prop:bilinear}.  In Section~\ref{S:decay}, we collect some consequences of the boundedness of the energy norm, including some decay estimates and control over the Strichartz norm.  In Section~\ref{S:NF}, we discuss a normal form transformation used to ameliorate the effect of the quadratic terms in the nonlinearity.  In Section~\ref{S:weighted1}, we begin to estimate the weighted norm, dealing with the cubic and higher terms, along with the quadratic terms without phase derivatives.  Sections~\ref{S:W2}--\ref{S:est} are dedicated to estimating the quadratic terms containing the phase derivatives (see Proposition~\ref{prop:bs-quad}).  These sections comprise the heart of the paper.  Finally, in Section~\ref{S:proof}, we collect all of the estimates and complete the proof of Theorem~\ref{T:main}.


\subsection*{Acknowledgements} We are grateful to an anonymous referee for their suggestion of the parameterization of the nonlinearity appearing in \eqref{E:cq2}; the original version of this paper employed a different parameterization, which resulted in us treating only what are now the cases $\beta >1$.

This work was partially supported by a grant from the Simons Foundation (\#342360 to Rowan Killip). R. K. was further supported by NSF grants DMS-1265868 and DMS-1600942.  J.~M. was supported by the NSF Postdoctoral Fellowship DMS-1400706 at the University of California, Berkeley.  M. V. was supported by NSF grant DMS-1500707. Part of the work on this project was supported by the NSF grant DMS-1440140, while the authors were in residence at the Mathematical Sciences Research Institute in Berkeley, California, during the Fall 2015 semester.   

\section{Notation and useful lemmas}\label{S:notation} For non-negative quantities $A$ and $B$, we write $A\lesssim B$ to denote $A\leq CB$ for some $C>0$.  We write $A\ll B$ to denote $A\leq cB$ for some small $c\in (0,1)$.  We write $A\sim B$ if $A\lesssim B$ and $B\lesssim A$.  We write $A\wedge B = \min\{A,B\}$ and $A\vee B=\max\{A,B\}$.  We write $a\pm$ to denote $a\pm\eps$ for some small $\eps>0$.

We write a complex-valued function $v$ as $v=v_1+iv_2$.  We write $\tilde v$ to indicate that either $v$ or $\bar v$ may appear.  When $X$ is a monomial, we write $\text{\O}(X)$ to denote a finite linear combination of the factors of $X$, where Mikhlin multipliers (for example, Littlewood--Paley projections or the operator $U$ defined below) and/or complex conjugation may be additionally applied in each factor.  We extend $\text{\O}$ to polynomials via $\text{\O}(X+Y)=\text{\O}(X)+\text{\O}(Y)$.

In what follows, $\nabla_\xi$ denotes derivatives with respect to $\xi$ with $\xi_2$ fixed and $\xi_1=\xi-\xi_2$.  On the other hand, $\nabla_{\xi_2}$ indicates derivatives with respect to $\xi_2$ with $\xi$ fixed and $\xi_1=\xi-\xi_2$.  Similarly, $\nabla_{\xi_1}$ indicates derivatives with respect to $\xi_1$ with $\xi$ fixed and $\xi_2=\xi-\xi_1$.

For a time interval $I$, we write $L_t^q L_x^r(I\times\R^3)$ for the Banach space of functions $u:I\times\R^3\to\C$ equipped with the norm
\[
\|u\|_{L_t^q L_x^r(I\times\R^3)}=\biggl(\int_I \|u(t)\|_{L_x^r(\R^3)}^q\,dt\biggr)^{1/q},
\]
with the usual adjustments when $q=\infty$. If $q=r$, we write $L_t^q L_x^q = L_{t,x}^q$.  We will often abbreviate
\[
\|u\|_{L_t^q L_x^r(I\times\R^3)}=\|u\|_{L_t^q L_x^r}\qtq{and} \|u\|_{L_x^r(\R^3)} = \|u\|_{L_x^r}.
\]
We write $r'\in[1,\infty]$ for the H\"older dual of $r\in[1,\infty]$, i.e. the solution to $\tfrac1r+\tfrac1{r'}=1$.

At times, we will make use of the Lorentz spaces $L_x^{q,\alpha}$ defined via the quasi-norm
\[
\| v\|_{L_x^{q,\alpha}}:= \bigl\| \lambda\,\bigl| \{x:|v(x)|>\lambda\}\bigr|^{\frac1q}\bigr\|_{L^\alpha((0,\infty),\frac{d\lambda}{\lambda})},
\]
where $1\leq q<\infty$ and $1\leq \alpha\leq\infty$.  We have $L_x^{q,q}=L_x^q$ and  $L_x^{q,\alpha}\hookrightarrow L_x^{q,\beta}$ for $\alpha<\beta$. We also have the following generalized H\"older and Hardy--Littlewood--Sobolev inequalities:

\begin{lemma}[H\"older and Hardy--Littlewood--Sobolev in Lorentz spaces, \cite{Hunt, One}]\label{lem:hls}\leavevmode\!\\The following estimates hold:
\begin{itemize}
\item[(i)]
\[
\|fg\|_{L_x^{q,\alpha}} \lesssim \|f\|_{L_x^{q_1,\alpha_1}}\|g\|_{L_x^{q_2,\alpha_2}},
\]
whenever $1\leq q,q_1,q_2<\infty$ and $1\leq\alpha,\alpha_1,\alpha_2\leq\infty$ satisfy $\tfrac{1}{q}=\tfrac{1}{q_1}+\tfrac{1}{q_2}$ and $\tfrac{1}{\alpha}=\tfrac{1}{\alpha_1}+\tfrac{1}{\alpha_2}$.
\item[(ii)]
\[
\|f\ast g\|_{L_x^{q,\alpha}} \lesssim \|f\|_{L_x^{q_1,\alpha_1}}\|g\|_{L_x^{q_2,\alpha_2}},
\]
whenever $1\leq q,q_1,q_2<\infty$ and $1\leq \alpha,\alpha_1,\alpha_2\leq\infty$ satisfy $\tfrac{1}{q}+1=\tfrac{1}{q_1}+\tfrac{1}{q_2}$ and $\tfrac{1}{\alpha}=\tfrac{1}{\alpha_1}+\tfrac{1}{\alpha_2}$.
\end{itemize}
\end{lemma}
Lemma~\ref{lem:hls}(ii) implies the following Sobolev embedding on $\R^3$:
\begin{equation}\label{sobolev}
\| |\nabla|^{-1} v\|_{L_x^2} \sim \|\,\tfrac{1}{|x|^2}\ast v\|_{L_x^2} \lesssim \| \tfrac{1}{|x|^2}\|_{L_x^{\frac32,\infty}}\| v\|_{L_x^{\frac65,2}} \lesssim \|v\|_{L_x^{\frac65,2}}.
\end{equation}

We define the Fourier transform on $\R^3$ via
\[
\wh{f}(\xi)=(2\pi)^{-3/2}\int_{\R^3} e^{-ix\xi}f(x)\,dx,\qtq{so that} f(x)=(2\pi)^{-3/2}\int_{\R^3} e^{ix\xi}\wh{f}(\xi)\,d\xi.
\]

We define $|\nabla|^s$ via $\widehat{|\nabla|^s f}(\xi)=|\xi|^s \wh{f}(\xi)$.  We make use of the following Fourier multiplier operators:
\[
\jb := \sqrt{2-\Delta},\quad  U = |\nabla|\jb^{-1},\quad H = |\nabla|\jb.
\]
We also use the notation $\langle \xi\rangle := \sqrt{2+|\xi|^2}$.  

We employ the standard Littlewood--Paley theory.  Let $\varphi$ be a radial bump function supported in $\{|\xi|\leq \tfrac{11}{10}\}$ and equal to one on the unit ball.  Let $\psi(\xi) = \varphi(\xi)-\varphi(2\xi)$.  For $N\in 2^{\mathbb{Z}}$, we define the Littlewood--Paley projections
\[
\widehat{P_{\leq N}u}(\xi) = \varphi(\xi/N)\wh{u}(\xi),\quad \widehat{P_N u}(\xi)=\psi(\xi/N)\wh{u}(\xi), \quad P_{>N} = Id-P_{\leq N}.
\]
We also write $u_{\leq N} := P_{\leq N} u$, and similarly for the other operators.   These operators commute with all other Fourier multiplier operators.  They are self-adjoint and bounded on every $L_x^p$ and $H_x^s$ space for $1\leq p\leq \infty$ and $s\geq 0$.  They obey the following standard estimates.
\begin{lemma}[Bernstein] Let $1\leq r\leq q\leq\infty$ and $s\geq 0$.  Then
\begin{align*}
\| |\nabla|^s P_{\leq N}u\|_{L_x^r(\R^3)}& \lesssim N^s \|P_{\leq N}u\|_{L_x^r(\R^3)}, \\
\|P_{>N}u\|_{L_x^r(\R^3)}&\lesssim N^{-s}\||\nabla|^s P_{>N}u\|_{L_x^r(\R^3)}, \\
\| P_{\leq N} u\|_{L_x^q(\R^3)} & \lesssim N^{\frac3r-\frac3q}\|P_{\leq N}u\|_{L_x^r(\R^3)}.
\end{align*}
\end{lemma}

The operators $P_N$ are not true projections, in the sense that $P_N\neq P_N^2$.  As a substitute, we introduce the `fattened' Littlewood--Paley multiplier $\tilde \psi(\xi) = \psi(2\xi)+\psi(\xi)+\psi(\xi/2)$ and define the operators $\tilde P_N$ analogously to the above. We then have $P_N = \tilde P_N P_N$.  

\subsection{Linear estimates} In this section, we record some estimates for the linear propagator $e^{\pm itH}$.

\begin{proposition}[Dispersive estimates]\label{prop:dispersive} For $0<N\leq 1$ and $t\neq 0$,
\begin{equation}\label{E:FLDE}
\|e^{\pm it H}f_N\|_{L_x^\infty(\R^3)} \lesssim \min\{N^3, N^2|t|^{-1}, N^{\frac12}|t|^{-\frac32}\}\|f_N\|_{L_x^1(\R^3)}.
\end{equation}

For $2\leq r<\infty$ and $t\neq 0$,
\[
\|e^{\pm it H}f\|_{L_x^{r,2}(\R^3)} \lesssim |t|^{-(\frac32-\frac3r)}\|U^{\frac12-\frac1r}f\|_{L_x^{r',2}(\R^3)}.
\]
\end{proposition}

The proof of this proposition is a standard application of stationary phase methods.  The key information is the eigenvalues of the Hessian of $H(\xi)$; these are
\[
\tfrac{2|\xi|(3+|\xi|^2)}{(2+|\xi|^2)^{3/2}}\sim \tfrac{|\xi|}{\langle \xi\rangle} \qtq{and} \tfrac{2(1+|\xi|^2)}{|\xi|(2+|\xi|^2)^{1/2}}\sim\tfrac{\langle \xi\rangle}{|\xi|}
\]
and the latter has multiplicity $2$.

From these dispersive estimates, one can deduce the following Strichartz estimates (see \cite{KeelTao}, for example).  In the following, we call $(q,r)$ an \emph{admissible pair} if $2\leq q\leq\infty$ and $\tfrac2q+\tfrac3r=\tfrac32$.  We call $(\alpha,\beta)$ a \emph{dual admissible pair} if $(\alpha',\beta')$ is an admissible pair.

\begin{proposition}[Strichartz estimates]\label{prop:strichartz} Let $I$ be a time interval and $t_0\in I$. Let $(q,r)$ be an admissible pair and let $(\alpha,\beta)$ be a dual admissible pair.  Then
\[
\biggl\| e^{-itH}\varphi+ \int_{t_0}^t e^{-i(t-s)H}F(s)\,ds\biggr\|_{L_t^q L_x^r(I\times\R^3)} \lesssim \|\varphi\|_{L_x^2}+ \|F\|_{L_t^{\alpha} L_x^{\beta}(I\times\R^3)}.
\]
\end{proposition}

\subsection{Bilinear estimates} We will frequently encounter bilinear Fourier multiplier operators of the following type.

\begin{definition}[Bilinear operators]\label{D:bilinear} Given a function $B:\R^3\times\R^3\to\C$, we define
\[
B[f,g](x)= (2\pi)^{-3}\iint e^{ix\xi} B(\xi-\eta,\eta) \wh{f}(\xi-\eta)\wh{g}(\eta)\,d\eta \,d\xi.
\]
For functions $A,B:\R^3\times\R^3\to\C$ we write $AB[\cdot,\cdot]$ for the multiplier with symbol given by the pointwise product $AB$.
\end{definition}

\begin{remark} The symbol $B=1$ corresponds to the pointwise product.  The condition $\overline{B(\xi_1, \xi_2)}=B(-\xi_1, -\xi_2)$ characterizes those multipliers that map real-valued functions to real-valued functions.
\end{remark}

We will often work with bilinear multipliers on the Fourier side.  Note that
\[
\widehat{B[f,g]}(\xi) = (2\pi)^{-3/2}\int B(\xi-\eta,\eta)\wh{f}(\xi-\eta)\wh{g}(\eta)\,d\eta.
\]

We will consistently use the notation
\[
\xi_1=\xi-\eta,\quad \xi_2=\eta,\quad \xi=\xi_1+\xi_2
\]
and write
\[
\widehat{B[f,g]}(\xi)=(2\pi)^{-3/2}\int B(\xi_1,\xi_2)\wh{f}(\xi_1)\wh{g}(\xi_2)\,d\xi_2.
\]

We will need several estimates for bilinear operators.  We begin with some standard results concerning sufficiently smooth multipliers.

\begin{definition}[Coifman--Meyer--Mikhlin] We call a Fourier multiplier $m:\R^3\to\C$ a \emph{Mikhlin multiplier} if
\[
|\partial_\xi^\alpha m(\xi)|\lesssim_\alpha |\xi|^{-|\alpha|}
\]
for all multiindices up to sufficiently high order.

We call a bilinear multiplier $B:\R^3\times\R^3\to\C$ a \emph{Coifman--Meyer multiplier} if
\[
|\partial_{\xi_1}^\alpha \partial_{\xi_2}^\beta B(\xi_1,\xi_2)| \lesssim_{\alpha,\beta} \bigl(|\xi_1|+|\xi_2|\bigr)^{-(|\alpha|+|\beta|)}
\]
for all multiindices up to sufficiently high order.

We call a bilinear multiplier $B:\R^3\times\R^3\to\C$ a \emph{Coifman--Meyer--Mikhlin} multiplier if it may be written in the form
\begin{equation}\label{cmm}
B(\xi_1,\xi_2) = m(\xi)m_1(\xi_1) m_2(\xi_2)B_0(\xi_1,\xi_2) ,
\end{equation}
where $m,m_1,m_2$ are Mikhlin and $B_0$ is Coifman--Meyer.
\end{definition}

Combining the Coifman--Meyer theorem \cite{CoiMey, CoiMey2} together with the standard Mikhlin multiplier theorem, we obtain the following result:

\begin{proposition}[Coifman--Meyer--Mikhlin estimate \cite{CoiMey, CoiMey2}]\label{prop:CM} If $B:\R^3\times\R^3\to\C$ is Coifman--Meyer--Mikhlin, then
\[
\| B\|_{L_x^{r_1}\otimes L_x^{r_2}\to L_x^r} \lesssim 1
\]
for all $1<r<\infty$ and $1<r_1,r_2<\infty$ satisfying $\tfrac{1}{r}=\tfrac{1}{r_1}+\tfrac{1}{r_2}$.
\end{proposition}

The next bilinear estimates allow for more singular multipliers.  These estimates are a slight strengthening of similar estimates appearing in \cite{GNT:3d} (cf. \cite[Lemma~9.1, Corollary~9.3]{GNT:3d}).  Pointwise symbol estimates are replaced by conditions involving $L^2$-based Sobolev regularity.  We need both a time-independent bilinear estimate and a Strichartz-type bilinear estimate.  Given a symbol $b:\R^3\times\R^3\to\C$, we will consider the following `norm':
\begin{equation}\label{bnorm}
\op{b} := \|b\|_{L_\xi^\infty \dot H_{\xi_1}^1}^{\frac12} \|b\|_{L_\xi^\infty \dot H_{\xi_1}^2}^{\frac12} \wedge \|b\|_{L_\xi^\infty \dot H_{\xi_2}^1}^{\frac12} \|b\|_{L_\xi^\infty \dot H_{\xi_2}^2}^{\frac12},
\end{equation}
where for fixed $\xi$ we consider $b$ as a function of $\xi_1$ via $b=b(\xi_1,\xi-\xi_1)$ or as a function of $\xi_2$ via $b=b(\xi-\xi_2,\xi_2)$.

\begin{proposition}[Bilinear estimates]\label{prop:bilinear}  Let $a,b:\R^3\times\R^3\to\C$ be the symbols of bilinear operators $a,b$.
\begin{itemize}
\item[(i)] For any $1\leq r_1,r_2\leq\infty$,
\[
\| ab\|_{L_x^{r_1}\otimes L_x^{r_2}\to L_x^2} \lesssim \| a\|_{L_x^{r_1}\otimes L_x^{r_2} \to L_x^2} \op{b}.
\]
\item[(ii)]  For any dual admissible pair $(\alpha,\beta)$ and any $q_1,r_1,q_2,r_2\in[1,\infty]$ satisfying $\tfrac{1}{q_1}+\tfrac{1}{q_2}=\tfrac{1}{\alpha}$ and $\tfrac{1}{r_1}+\tfrac{1}{r_2}=\tfrac{1}{\beta}$,
\[
\biggl\| \int e^{itH}ab[f_1(t),f_2(t)]\,dt\biggr\|_{L_x^2} \lesssim \|a\|_{L_x^{r_1}\otimes L_x^{r_2} \to L_x^\beta} \op{b}\, \|f_1\|_{L_t^{q_1} L_x^{r_1}} \|f_2\|_{L_t^{q_2}L_x^{r_2}}.
\]
\end{itemize}
\end{proposition}

\begin{remark} No assumptions are needed for the operator $a$ beyond boundedness.   In Lemma~\ref{AOK}(ii) below we will see an example of a bilinear multiplier that is readily seen to be bounded (by relatively elementary means) but which is not of Coifman--Meyer--Mikhlin type.  A larger class of symbols obeying $\textit{\L}^\infty \dot B^{3/2}_{2,1}$ bounds is admitted in \cite{GNT:3d}; this can be immediately recovered from the above by dyadic splitting and the triangle inequality.  We prefer to formulate our result in terms of the norm \eqref{bnorm} since in practical instances, this seems easiest to bound in an effective manner.
\end{remark}

\begin{proof} By symmetry, it suffices to prove the result with the first factor on the right-hand side of \eqref{bnorm}.  Define $w:\R^3\times\R^3\to\C$ so that
\[
b(\xi_1,\xi-\xi_1) = \int \overline{w(x,\xi)} e^{ix\xi_1}\,dx.
\]
As
\[
\| b(\xi_1,\xi-\xi_1) \|_{\dot H_{\xi_1}^s} \sim \|\,|x|^s w(x,\xi)\|_{L_x^2}
\]
uniformly for $\xi\in\R^3$ (by Plancherel), we may choose $R>0$ so that
\begin{equation}\label{choice-of-R}
R\int_{|x|\leq R} |x|^2|w(x,\xi)|^2\,dx + R^{-1}\int_{|x|>R}  |x|^4|w(x,\xi)|^4\,dx \lesssim \op{b}^2,
\end{equation}
uniformly for $\xi\in\R^3$.

(i) Fix $f_j\in L_x^{r_j}$ and $h\in L_x^2$. By Plancherel,
\[
\langle h, ab[f_1,f_2]\rangle = (2\pi)^{-3/2}\iint \overline{w(x,\xi)\wh{h}(\xi)} a(\xi_1,\xi-\xi_1)e^{ix\xi_1}\wh{f_1}(\xi_1)\wh{f_2}(\xi-\xi_1)\,d\xi_1\,d\xi\,dx.
\]
Using Plancherel again and the fact that translation is an isometry on $L_x^{r_1}$, we deduce
\[
\bigl|\langle h, ab[f_1,f_2]\rangle\bigr| \lesssim \|a\|_{L_x^{r_1}\otimes L_x^{r_2}\to L_x^2} \|f_1\|_{L_x^{r_1}}\|f_2\|_{L_x^{r_2}} \int \|w(x,\xi)\wh{h}(\xi)\|_{L_\xi^2}\,dx.
\]
By Cauchy--Schwarz and \eqref{choice-of-R},
\begin{align*}
\biggl(\int \|w(x,\xi)\wh{h}(\xi)\|_{L_\xi^2}\,dx\biggr)^2 & \lesssim  \int_{|x|\leq R} |x|^{-2}\,dx\cdot\int_{|x|\leq R}\int|x|^2 |w(x,\xi)|^2|\wh{h}(\xi)|^2\,d\xi\,dx \\
&\quad + \int_{|x|>R}|x|^{-4}\,dx\cdot\int_{|x|>R}|x|^4 |w(x,\xi)|^4 |\wh{h}(\xi)|^2\,d\xi\,dx \\
& \lesssim \op{b}^2\|h\|_{L_x^2}^2.
\end{align*}
The estimate in (i) follows.

(ii) In this case, we write
\begin{align*}
(2\pi&)^{3/2}\bigl\langle h, \int e^{itH} ab[f_1(t),f_2(t)]\,dt\bigr\rangle \\
& = \iiiint \overline{w(x,\xi)\wh{h}(\xi)} e^{itH(\xi)} a(\xi_1,\xi-\xi_1) e^{ix\xi_1}f_1(t,\xi_1)f_2(t,\xi-\xi_1)\,d\xi_1\,dt\,d\xi\,dx.
\end{align*}
Thus, applying Strichartz (Proposition~\ref{prop:strichartz}) and estimating as above,
\begin{align*}
\bigl|\bigl\langle& h,\int e^{itH}ab[f_1(t),f_2(t)]\,dt\bigr\rangle\bigr| \\
& \lesssim \|a\|_{L_x^{r_1}\otimes L_x^{r_2}\to L_x^\beta} \bigl\| \|f_1(t)\|_{L_x^{r_1}} \|f_2(t)\|_{L_x^{r_2}} \bigr\|_{L_t^\alpha} \int \|w(x,\xi)\wh{h}(\xi)\|_{L_\xi^2}\,dx \\
&\lesssim \|a\|_{L_x^{r_1}\otimes L_x^{r_2}\to L_x^\beta}\op{b} \,\|f_1\|_{L_t^{q_1}L_x^{r_1}} \|f_2\|_{L_t^{q_2}L_x^{r_2}} \|h\|_{L_x^2}.
\end{align*}
The estimate in (ii) follows.
\end{proof}

\section{Decay via the energy norm}\label{S:decay}  In this section we prove that control over the energy norm
\[
\|v(t)\|_X = \|\jb v(t)\|_{L_x^2} + \|J(t)v(t)\|_{L_x^2}+\|(x\times\nabla)v(t)\|_{L_x^2},
\]
where $J(t)=e^{-itH}xe^{itH}$ implies decay for the solution $v$, as well as for $u=v_1+iU^{-1}v_2$.  In fact, one does not need to use the $(x\times\nabla)v$ term to prove decay.  The estimates we prove are in the spirit of Klainerman--Sobolev inequalities.  As a consequence of the decay estimates, we can then prove an \emph{a priori} estimate for the Strichartz-norm of $v$.  At the end of this section, we also record some bounds for norms with weights that will play a role in controlling the $L_x^2$-norm of $J(t)v(t)$.

We begin with an elementary lemma.
\begin{lemma}[Control of low frequencies]\label{lem:UL2}
\begin{equation}\label{UL2}
\| U^{-2} v(t)\|_{L_x^6} + \|U^{-1}v(t)\|_{L_x^2} \lesssim \|v(t)\|_X.
\end{equation}
\end{lemma}
\begin{proof} As the $X$-norm controls the $H_x^1$-norm, it suffices to consider the low frequencies.  For this, we use the Sobolev embedding \eqref{sobolev} and H\"older's inequality (cf. Lemma~\ref{lem:hls}) to estimate
\begin{align*}
\| U^{-2} v_{\leq 1} \|_{L_x^6} + \| U^{-1} v_{\leq 1}\|_{L_x^2}  \lesssim \||\nabla|^{-1} e^{itH}v\|_{L_x^2} &\lesssim \| e^{itH} v\|_{L_x^{\frac65,2}} \\ & \lesssim \|\,|x|^{-1}\|_{L_x^{3,\infty}} \|Jv\|_{L_x^2}.
\end{align*}
The result follows.\end{proof}

Next, we prove decay.

\begin{lemma}[Decay]\label{lem:decay} Fix $t\geq 0$.
\begin{itemize}
\item[(i)] $\|v(t)\|_{L_x^6}\lesssim\langle t\rangle^{-1} \|v(t)\|_X.$
\item[(ii)] For $0<N\leq 1$,
\begin{equation}\label{FL6}
\|U^{-1}v_N(t)\|_{L_x^{6,2}} \lesssim  \min\{N, N^{\frac13} t^{-\frac23}, N^{-\frac23}t^{-1}\}\|Jv(t)\|_{L_x^2}.
\end{equation}
Consequently,
\[
\|U^{-1}v(t)\|_{L_x^6} \lesssim \langle t\rangle ^{-\frac79}\|v(t)\|_{X}.
\]
\item[(iii)] For any $2\leq r\leq 6$,
\[
\|v(t)\|_{L_x^r}\lesssim \langle t\rangle^{-(\frac32-\frac3r)}\|v(t)\|_X \qtq{and} \|U^{-1} v(t) \|_{L_x^r} \lesssim \langle t\rangle^{-\frac{7}{9}(\frac32-\frac3r)}\|v(t)\|_X.
\]
\end{itemize}
\end{lemma}

\begin{proof}  First note that by Sobolev embedding,
\[
\|v(t)\|_{L_x^6} \lesssim \|U^{-1}v(t)\|_{L_x^6} \lesssim \|\jb v(t)\|_{L_x^2}.
\]
Thus, to prove (i) and (ii), it suffices to show that control over $Jv$ implies decay for times $t\geq1$.

To this end, we note that the dispersive estimate (Proposition~\ref{prop:dispersive}) and H\"older's inequality imply
\[
\|v(t)\|_{L_x^6} \lesssim t^{-1}\|e^{itH}v(t)\|_{L_x^{\frac65,2}} \lesssim t^{-1}\|Jv(t)\|_{L_x^2},
\]
giving (i).

Using instead the frequency-localized dispersive estimate \eqref{E:FLDE}, we easily obtain \eqref{FL6}. Thus, for $t\geq 1$,
\[
\|U^{-1}v_{\leq 1}(t)\|_{L_x^6}\lesssim \biggl(\sum_{N\leq t^{-\frac13}} N^{\frac13}t^{-\frac23}+\sum_{t^{-\frac13}<N\leq 1} N^{-\frac23}t^{-1}\biggr)\|Jv(t)\|_{L_x^2} \lesssim t^{-\frac79}\|Jv(t)\|_{L_x^2},
\]
which together with (i) settles (ii).

The estimates in  (iii) follow by interpolating between (i) and (ii).\end{proof}

We will next prove control over the Strichartz norm of $v$.  As the nonlinearity is described in terms of $u=v_1+iU^{-1}v_2$, it is useful to record the following lemma.

\begin{lemma}[Control of $u$]\label{lem:str} Let $v=u_1+iUu_2$ and let $I$ be a time interval.  Then with all space-time norms over $I\times\R^3$,
\[
\|\jb u\|_{L_t^\infty L_x^2\cap L_t^2 L_x^6} + \|(x\times\nabla)u\|_{L_t^\infty L_x^2} \lesssim \|v\|_{Z(I)}.
\]
\end{lemma}

\begin{proof} As $P_{>1}U^{-1}$ is bounded on $L_x^2$ and $L_x^6$ and $x\times\nabla$ commutes with any Fourier multiplier operator whose symbol is radial, it is clear that
\[
\|\jb u_{>1}\|_{L_t^\infty L_x^2\cap L_t^2 L_x^6}+\|(x\times\nabla)u_{>1}\|_{L_t^\infty L_x^2} \lesssim \|v\|_{Z(I)}.
\]

For the low frequencies, we use Bernstein and \eqref{UL2} to see that
\begin{equation}\label{low-u-L2}
\|\jb u_{\leq 1}\|_{L_t^\infty L_x^2(I\times\R^3)} \lesssim \|v\|_{L_t^\infty (I;X)}.
\end{equation}
Next, we use Lemma~\ref{lem:decay} to estimate
\[
\|\jb u_{\leq 1}\|_{L_t^2 L_x^6} \lesssim \|U^{-1} v\|_{L_t^2 L_x^6} \lesssim \|\langle t\rangle^{-\frac79}\|_{L_t^2(I)} \|v\|_{L_t^\infty(I;X)}.
\]
Finally, using $x\times\nabla = -\nabla\times x$, unitarity of $e^{itH}$, and boundedness of Riesz potentials,
\begin{align*}
\|(x\times\nabla)u_{\leq 1}(t)\|_{L_x^2}\lesssim\sum_{j\neq k}\|P_{\leq 1}U^{-1}\partial_j x_k e^{itH}v(t)\|_{L_x^2} \lesssim \| J(t)v(t)\|_{L_x^2} \lesssim \|v(t)\|_X
\end{align*}
The result follows. \end{proof}


We can now prove control of the Strichartz norm (cf. \eqref{def:S}).

\begin{proposition}\label{prop:str} Let $0\leq t_0=\inf I$ and suppose $v:I\times\R^3\to\C$ solves \eqref{E:cqv}. Write $v=u_1+iUu_2$.  There exists $\eps>0$ such that
\begin{align}
\|e^{-i(t-t_0)H}v(t_0)\|_{S(I)} & \lesssim \| v(t_0)\|_{X}, \label{E:str1} \\
\biggl\| \int_{t_0}^t e^{-i(t-s)H} N_v(u(s)) \,ds \biggr\|_{S(I)} & \lesssim \langle t_0\rangle^{-\eps} \sum_{k=2}^5 \|v\|_{Z(I)}^k.\label{E:str2}
\end{align}
\end{proposition}

\begin{proof} Recalling that $e^{itH}$ commutes with $x\times\nabla$, the estimate \eqref{E:str1} follows immediately from Proposition~\ref{prop:strichartz}.

We turn to \eqref{E:str2}.  An application of Proposition~\ref{prop:strichartz} shows that to prove \eqref{E:str2}, we need to estimate
\[
\|D\text{\O}(u^2)\|_{L_t^{\frac43}L_x^{\frac32}} + \sum_{k=3}^5\|D\text{\O}(u^k)\|_{L_t^2 L_x^{\frac65}},\quad D\in\{\jb,x\times\nabla\}.
\]
To do this, we will rely on the decay for $u$ provided by Lemma~\ref{lem:decay}, as well as Lemma~\ref{lem:str}.  We also note that $x\times\nabla$ obeys Leibniz and chain rules.

For the quadratic and cubic terms, we may estimate
\begin{align*}
\|D\text{\O}(u^2)\|_{L_t^{\frac43}L_x^{\frac32}} & \lesssim \|u\|_{L_t^{\frac43} L_x^6} \|D u\|_{L_t^\infty L_x^2} \lesssim \langle t_0\rangle^{-\frac{1}{36}}\|v\|_{Z(I)}^2, \\
\|D\text{\O}(u^3)\|_{L_t^2 L_x^{\frac65}} & \lesssim \|u\|_{L_t^4 L_x^6}^2 \|D u\|_{L_t^\infty L_x^2} \lesssim \langle t_0\rangle^{-\frac{19}{18}}\|v\|_{Z(I)}^3
\end{align*}
for $D\in\{\jb,x\times\nabla\}$.

We turn to the quartic and quintic terms.  Writing $u=u_{\leq 1}+u_{>1}$ and distributing derivatives, we see that to estimate the terms involving $x\times\nabla$, it suffices to consider the contribution of the following terms:
\[
\text{\O}(u^k (x\times\nabla)u_{\leq 1}) + \text{\O}(u^k(x\times\nabla)u_{>1}),\quad k\in\{3,4\}.
\]

Recalling boundedness of $U^{-1}P_{>1}$ (and the fact that this operator commutes with $x\times\nabla$), we can first bound
\begin{align*}
\|\jb\text{\O}&(u^4)\|_{L_t^2 L_x^{\frac65}}+\|u^3(x\times\nabla)u_{>1}\|_{L_t^2 L_x^{\frac65}} \\
& \lesssim \|u\|_{L_t^\infty L_x^{\frac92}}^3 \bigl(\|\jb u\|_{L_t^2 L_x^6}+\|(x\times\nabla)v\|_{L_t^2 L_x^6}\bigr) \lesssim \langle t_0\rangle^{-\frac{35}{18}}\|v\|_{Z(I)}^4, \\
\|\jb\text{\O}&(u^5)\|_{L_t^2 L_x^{\frac65}} + \|u^4(x\times\nabla)u_{>1}\|_{L_t^2 L_x^{\frac65}} \\
&\lesssim \|u\|_{L_t^\infty L_x^6}^4 \bigl(\|\jb u\|_{L_t^2 L_x^6}+ \|(x\times\nabla)v\|_{L_t^2 L_x^6}\bigr)\lesssim \langle t_0\rangle^{-\frac{28}{9}}\|v\|_{Z(I)}^5.
\end{align*}

Finally, we note that by Bernstein and Lemma~\ref{lem:str} (and the fact that $x\times\nabla$ commutes with frequency projections), we have
\[
\|(x\times\nabla)u_{\leq 1}\|_{L_t^\infty L_x^r}\lesssim \|v\|_{Z(I)} \qtq{for any}r\geq 2.
\]
Thus, estimating with the same spaces as above, we have
\begin{align*}
\|u^3(x\times\nabla)u_{\leq1}\|_{L_t^2 L_x^{\frac65}}& \lesssim \|u\|_{L_t^\infty L_x^{\frac92}}^2 \|(x\times\nabla)u_{\leq 1}\|_{L_t^\infty L_x^{\frac92}} \|u\|_{L_t^2 L_x^6} \lesssim \langle t_0\rangle^{-\frac{35}{27}}\|v\|_{Z(I)}^4, \\
\|u^4(x\times\nabla)u_{\leq 1}\|_{L_t^2 L_x^{\frac65}}&\lesssim\|u\|_{L_t^\infty L_x^6}^3 \|(x\times\nabla)u_{\leq 1}\|_{L_t^\infty L_x^6}\|u\|_{L_t^2 L_x^6}\lesssim \langle t_0\rangle^{-\frac73}\|v\|_{Z(I)}^5.
\end{align*}
The result follows. \end{proof}

\begin{remark} Note that the proof above gives
\begin{equation}\label{dual-est}
\| \jb\text{\O}(u^k) \|_{L_t^2 L_x^{\frac65}} \lesssim \langle t_0\rangle^{-\frac{19}{18}}\|v\|_{Z(I)}^k \qtq{for} k\in\{3,4,5\},
\end{equation}
which we will use in Section~\ref{S:est}.
\end{remark}


The bulk of the remainder of the paper is devoted to controlling the weighted norm.  Taking the Fourier transform, we compute
\begin{equation}\label{expandJ}
J(t) = e^{-itH}xe^{itH} = x+\tfrac{2it(1-\Delta)}{\jb}\tfrac{\nabla}{|\nabla|} =: x+it\jbb.
\end{equation}
To estimate the contribution of the cubic and higher order terms in the nonlinearity to the weighted norm, we will write $J(t)$ as in \eqref{expandJ} and estimate each resulting piece separately.  This requires bounds on $xu$, for which we rely on the following lemma.

\begin{lemma}[Bounds for weights]\label{lem:weighted} Let $v=u_1+iUu_2$ and $t\geq 0$. Then
\begin{align}
\|xu_{>1}(t)\|_{L_x^2} &\lesssim \langle t\rangle \|v(t)\|_X, \label{E:xu2} \\
\|xu_{\leq 1}(t)\|_{L_x^6} & \lesssim \langle t\rangle^{\frac29}\|v(t)\|_X. \label{E:xu6}
\end{align}
\end{lemma}

\begin{proof} As the commutator $[x,P]$ is bounded on $L_x^2$ and $L_x^6$ for $P\in\{P_{\leq 1},P_{>1}\}$ and $\|u\|_{L_x^2\cap L_x^6}\lesssim \|v\|_X$, it suffices to prove the estimates \eqref{E:xu2} and \eqref{E:xu6} for $P(xu)$.

To this end, we write
\begin{align*}
& xu_1 = xv_1 = \Re(Jv) +t\jbb v_2, \\
& xu_2 = xU^{-1}v_2 = [x,U^{-1}]v_2 + U^{-1}\bigl[ \Im(Jv) -t\jbb v_1\bigr].
\end{align*}
As the commutator $[x,U^{-1}]=\tfrac{-2}{\Delta\jb}\tfrac{\nabla}{|\nabla|}$, we obtain
\begin{align*}
&\|P_{>1}(xu_1)\|_{L_x^2} \lesssim \|Jv\|_{L_x^2} + t\|\jb v\|_{L_x^2} \lesssim \langle t\rangle\|v\|_X, \\
&\|P_{>1}(xu_2)\|_{L_x^2} \lesssim \|v\|_{L_x^2} + \|Jv\|_{L_x^2} +t\|\jb v\|_{L_x^2} \lesssim \langle t\rangle\|v\|_X.
\end{align*}
Next, using Lemmas~\ref{lem:UL2} and \ref{lem:decay} together with Sobolev embedding, we estimate
\begin{align*}
& \|P_{\leq 1}(xu_1)\|_{L_x^6} \lesssim \|Jv\|_{L_x^2} + t\|v\|_{L_x^6} \lesssim \|v\|_X, \\
& \|P_{\leq 1}(xu_2)\|_{L_x^6} \lesssim \|\Delta^{-1} v\|_{L_x^6} + \||\nabla|^{-1} Jv\|_{L_x^6} + t\|U^{-1}v\|_{L_x^6} \lesssim \langle t\rangle^{\frac29}\|v\|_X.
\end{align*}
The result follows.
\end{proof}

\section{Normal form transformation}\label{S:NF}

In this section, we discuss normal form transformations for \eqref{E:cqv}.  The use of normal forms in PDE originated in the work of Shatah \cite{Sha} and has since become a widely-used technique in the setting of nonlinear dispersive equations.  Briefly, the idea is to look for a change of variables with the effect of removing the most problematic terms from the nonlinearity.

In our setting it is the quadratic terms, namely,
\[
U[(3+4\beta)u_1^2 +  u_2^2] + 2i u_1 u_2,
\]
that are the most difficult to estimate. Recall also that $u_1 = v_1$ and $u_2 = U^{-1} v_2$.  In our previous work concerning the final-state problem \cite{KMV}, we used a normal form transformation to eliminate the $u_2^2$ term from the nonlinearity.  In \cite{KMV}, we worked with energy-space solutions (i.e. no weighted assumption) and the poor spatial decay of $u_2$  was the biggest difficulty in handling the quadratic terms; in particular, we had only $u_2\in L_x^6$, compared to $u_1\in L_x^3\cap L_x^6$.  Due to the nature of our arguments (which entailed testing against a dense subclass of test functions, namely, those supported away from the origin in frequency space), the interaction output frequency did not play an important role.  In particular, the presence of $U$ in the real part of the nonlinearity offered no real advantage.

In this work, to treat the quadratic terms will require a careful decomposition of frequency space into `non-resonant' regions, which requires an accounting of all of the various quadratic frequency interactions; see Sections~\ref{S:W2}--\ref{dec3}.  In this analysis, the operator $U$ will be crucial in our treatment of interactions with small output frequency.  Accordingly, the absence of $U$ in the imaginary part of the nonlinearity is problematic.  Note also that the weighted assumption gives better control over $u_2=U^{-1}v_2$ (cf. Section~\ref{S:decay}), which already makes the $u_2^2$ term less intimidating in this setting.  Thus, in this work, we will employ a normal form transformation to remove the term $2i u_1u_2$.  The authors of \cite{GNT:3d} also elected to remove this term in order to treat the initial-value problem for Gross--Pitaevskii, for the same reasons.   As we will see, the normal form transformation that removes the $u_1 u_2$ term also cancels the two copies of $U^{-1}$ appearing in the $v_2^2$ term, which actually puts the $v_1^2$ and $v_2^2$ terms on essentially equal footing.  It is worth noting that the symmetry conditions required of the normal form transformation make it impossible to remove \emph{all} of the quadratic terms.

The standard approach to normal form transformations is to introduce a new variable of the form
\[
w = v + \B[v,v],
\]
where $\B$ is a bilinear operator that is chosen so that
\[
(i\partial_t - H)\B[v,v]
\]
cancels the problematic term(s) in the nonlinearity.  Note that this will add some new terms to the nonlinearity, as well.  One then proves estimates for $w$, which (hopefully) solves a better equation than $v$, and then inverts the normal form to deduce information about $v$.

We take the following closely-related approach.  First, we rewrite \eqref{E:cqv} in integral form:
\begin{equation}\label{duhamel1}
e^{itH}v(t) = e^{it_0H}v(t_0) - i \int_{t_0}^t e^{isH} N_v(u) \,ds.
\end{equation}
Given a bilinear operator $\B$ as above, we can write
\begin{align*}
\int_{t_0}^t e^{isH} N_v(u) \,ds & = \int_{t_0}^t e^{isH}\bigl(N_v(u) + (i\partial_s-H)\B[v,v]\bigr)\,ds - \int_{t_0}^t i\partial_s \bigr(e^{isH}\B[v,v]\bigr)\,ds \\
& =\int_{t_0}^t e^{isH}\bigl(N_v(u) + (i\partial_s-H)\B[v,v]\bigr)\,ds - \bigl(i e^{isH}\B[v,v]\bigr)\bigr|_{s=t_0}^t.
\end{align*}
As before, we wish to choose $\B[v,v]$ to cancel the problematic quadratic terms.  In place of inverting the transformation, we will instead need to estimate the quadratic boundary term above.

We turn now to the details.  We take
\[
\B[v,v]=B_1[v_1,v_1]+B_2[v_2,v_2],
\]
where $B_1,B_2$ are symmetric, bilinear Fourier multiplier operators to be determined below.  Using \eqref{E:cqv}, we have
\begin{align*}
(i\partial_t-H)\B[v,v] & = 2i\bigl(B_1[v_1,Hv_2]-B_2[Hv_1,v_2]\bigr) \\
& \quad - H\bigl(B_1[v_1,v_1]+B_2[v_2,v_2]\bigr) \\
& \quad + 2i\bigl(B_1[v_1,\Im N(u)]-B_2[v_2,U\Re N(u)]\bigr).
\end{align*}

In order to cancel the term $2i u_1 u_2$, we need to choose $B_1$ and $B_2$ so that
\[
v_1 U^{-1}v_2 + B_1[v_1,Hv_2]-B_2[Hv_1,v_2] = 0.
\]
Imposing symmetry on $B_1$ and $B_2$ then leads to the unique choice
\[
B_1[f,g] = B[f,g] \qtq{and} B_2[f,g] =-B[U^{-1}f,U^{-1}g],
\]
where the symbol $B(\xi_1,\xi_2)$ of $B$ is given by
\begin{equation}\label{dnf}
B(\xi_1,\xi_2):= -(2+ |\xi_1|^2+|\xi_2|^2)^{-1}.
\end{equation}

With this choice, we write
\begin{equation}\label{nf-duh}
\begin{aligned}
\int_{t_0}^t e^{isH} N_v(u)\,ds & = \sum_{k=2}^6 \int_{t_0}^t e^{isH} \N_k \,ds  \\
&\quad -i \bigl[e^{isH}(B[v_1,v_1]-B[U^{-1}v_2,U^{-1}v_2])\bigr]_{s=t_0}^t,
\end{aligned}
\end{equation}
where the $\N_k$ are defined as follows:
\begin{equation}\label{Nk}
\begin{aligned}
\N_2 & =  (3+4\beta)U(v_1^2) - H B(v_1,v_1) +  U[U^{-1}v_2]^2 + H B(U^{-1}v_2,U^{-1}v_2), \\
\N_k & = U\Re N_k(u) + i\Im N_k(u) \\
&\quad + 2i[B(v_1,\Im N_{k-1}(u)) + B(U^{-1}v_2, \Re N_{k-1}(u))]\quad\text{for}\quad  k\in\{3,4,5\},\\
\N_6 &= 2i[B(v_1, \Im N_5(u)) + B(U^{-1}v_2, \Re N_5(u))].
\end{aligned}
\end{equation}

The structure of $\N_k$ for $k\in\{3,4,5,6\}$ will not be too important.  The quadratic terms, however, have some nice structure that will be important in the analysis below:
\[
\N_2 = U [ A_1(v_1,v_1) + A_2(v_2,v_2)],
\]
where
\begin{equation}\label{Aj}
\begin{aligned}
&A_1(\xi_1,\xi_2) = (4+4\beta)+\frac{ 2 \xi_1\cdot \xi_2}{2+|\xi_1|^2+|\xi_2|^2}, \\
& A_2(\xi_1,\xi_2) = -\frac{2 \langle \xi_1\rangle\langle\xi_2\rangle}{2+|\xi_1|^2+|\xi_2|^2}\frac{\xi_1}{|\xi_1|}\cdot\frac{\xi_2}{|\xi_2|}.
\end{aligned}
\end{equation}

One can check that
\begin{equation}\label{Ajbds}
|\partial_{\xi_j}^\alpha A_1| \lesssim \bigl(\tfrac{1}{|\xi_1|\vee|\xi_2|}\bigr)^{|\alpha|} \qtq{and} |\partial_{\xi_j}^\alpha A_2| \lesssim \bigl(\tfrac{1}{|\xi_1|\wedge|\xi_2|}\bigr)^{|\alpha|}.
\end{equation}
In particular, $A_1$ is Coifman--Meyer, but $A_2$ is not.  In fact, $A_2$ is not even Coifman--Meyer--Mikhlin; nonetheless, it is bounded, as the following lemma shows.

\begin{lemma}\label{AOK}\text{ }
\begin{itemize}
\item[(i)] The operator $A_1$ is Coifman--Meyer.
\item[(ii)] For any Coifman--Meyer--Mikhlin multiplier $b$, the operator with symbol
\[
b(\xi_1,\xi_2)\frac{\langle \xi_1\rangle\langle\xi_2\rangle}{2+|\xi_1|^2+|\xi_2|^2}
\]
maps $L^{r_1}\otimes L^{r_2}\to L^r$ boundedly whenever $1<r,r_1,r_2<\infty$ and $\tfrac{1}{r_1}+\tfrac{1}{r_2}=\frac{1}{r}.$ In particular, $A_2b:L^{r_1}\otimes L^{r_2}\to L^r$ boundedly.
\end{itemize}
\end{lemma}

\begin{proof} Item (i) follows from \eqref{Ajbds}. We focus on (ii) and argue by interpolation: Writing
\[
m_z(\xi_1,\xi_2) = b(\xi_1,\xi_2)\frac{\langle \xi_1\rangle^{2z}\langle \xi_2\rangle^{2(1-z)}}{2+|\xi_1|^2+|\xi_2|^2}
\]
for $z\in [0,1]+i\R$, it suffices to show that the operator corresponding to $m_{1/2}$ maps $L^{r_1}\otimes L^{r_2}\to L^r$ boundedly.  To this end, we note that $m_{iy}$ and $m_{1+iy}$ are Coifman--Meyer--Mikhlin with norms bounded by $\langle y\rangle^{N}$ for some $N\in\mathbb{N}$.  Indeed,
\[
m_{iy}(\xi_1,\xi_2) = (2+|\xi_1|^2)^{iy} (2+|\xi_2|^2)^{-iy} \tfrac{2 + |\xi_2|^2}{2+|\xi_1|^2+|\xi_2|^2}b(\xi_1,\xi_2),
\]
while for $m_{1+iy}$ one exchanges $\xi_1$ and $\xi_2$ in the first three factors.  Boundedness of $m_{1/2}$ can then be deduced by the three lines theorem (see, for example, \cite[Lemma~4.2]{SteinWeiss}). \end{proof}

\section{The weighted norm, part I}\label{S:weighted1}

In this section, we begin estimating the $L_x^2$-norm of $Jv$, where $v$ is a solution to \eqref{E:cqv}.  To do this, we write the Duhamel formula for $v(t)$ using the normal form \eqref{nf-duh}.

We first consider the contribution of everything except the quadratic terms appearing in the nonlinearity.

\begin{proposition}\label{prop:weighted1} Let $t_0\geq 0$ and let $I=[t_0,1]$ if $t_0<1$ and $I=[t_0,2t_0]$ otherwise. Let $v:I\times\R^3\to\C$ be a solution to \eqref{E:cqv}.  There exists $\eps>0$ such that for all $t\in I$,
\begin{align}
\| J(t) e^{-i(t-t_0)H}v(t_0)\|_{L_x^2} & \lesssim \|J(t_0)v(t_0)\|_{L_x^2}, \label{weighted1-1}\\
\| J(t) e^{-i(t-s)H} B[U^{-1}\tilde v(s),U^{-1}\tilde v(s)] \|_{L_x^2} & \lesssim \langle s\rangle^{-\eps}\|v\|_{Z(I)}^2, \quad s\in\{t_0,t\},\label{weighted1-2} \\
\biggl\| \int_{t_0}^t J(t) e^{-i(t-s)H} \N_k(u(s))\,ds\biggr\|_{L_x^2}& \lesssim\langle t_0\rangle^{-\eps} \|v\|_{Z(I)}^k, \quad k\in\{3,4,5,6\}. \label{weighted1-3}
\end{align}
Here $B$ is as in \eqref{dnf} and the $\N_k$ are as in \eqref{Nk}.
\end{proposition}

\begin{proof} To begin, recall that $J(t)=e^{-itH}xe^{itH}$, so that
\begin{equation}\label{Jcommute}
J(t)e^{-i(t-s)H} = e^{-itH} x e^{isH} = e^{-i(t-s)H}J(s).
\end{equation}
This, together with the fact that $e^{it H}$ is unitary on $L_x^2$, immediately implies \eqref{weighted1-1}.

We turn to \eqref{weighted1-2} and fix $s\in\{t_0,t\}$.  By Plancherel, it is equivalent to estimate
\[
\biggl\|\nabla_\xi \int e^{is\Phi} B(\xi_1,\xi_2)U^{-1}(\xi_1)U^{-1}(\xi_2) \wh{\tilde f}(\xi_1)\wh{\tilde f}(\xi_2)]\,d\xi_2 \biggr\|_{L_\xi^2},
\]
where $f(t)=e^{itH}v(t)$ and $\Phi = H(\xi)\pm H(\xi_1)\pm H(\xi_2)$.

Distributing the derivative, we are led to estimate the following terms:
\begin{equation}\label{three terms}
 s \| b_1[U^{-1}\tilde v, U^{-1}\tilde v] \|_{L_x^2}, \quad \|b_2[U^{-2}\tilde v, U^{-1} \tilde v]\|_{L_x^2}, \quad \|B[U^{-1}\widetilde{Jv}, U^{-1}\tilde v]\|_{L_x^2},
\end{equation}
where the multipliers $b_1$ and $b_2$ have symbols
\[
 b_1(\xi_1,\xi_2) = [\nabla_\xi \Phi]B(\xi_1,\xi_2) \qtq{and} b_2(\xi_1,\xi_2) = U^2(\xi_1)\nabla_\xi[U^{-1}(\xi_1)B(\xi_1,\xi_2)].
\]

Note that while the operator with symbol $b_1(\xi_1, \xi_2)$ is not Coifman--Meyer, it is Coifman--Meyer--Mikhlin.  Indeed, this is apparent from
\begin{align*}
b_1(\xi_1,\xi_2) = -2 \tfrac{\xi}{|\xi|}\tfrac{1}{\langle \xi\rangle}\tfrac{1+|\xi|^2}{2 +|\xi_1|^2+|\xi_2|^2} \mp 2 \tfrac{\xi_1}{|\xi_1|}\tfrac{1}{\langle \xi_1\rangle}\tfrac{1+|\xi_1|^2}{2 +|\xi_1|^2+|\xi_2|^2}.
\end{align*}
Thus, using Proposition~\ref{prop:CM} and Lemma~\ref{lem:decay}, we estimate
\begin{align*}
s\|b_1[U^{-1}v,U^{-1}v]\|_{L_x^2} & \lesssim s\|U^{-1}v\|_{L_x^6} \|U^{-1}v\|_{L_x^3} \lesssim \langle s\rangle^{-\frac16} \|v\|_{Z(I)}^2.
\end{align*}

Similarly,
\[
b_2(\xi_1, \xi_2)= 2\tfrac{1}{\langle \xi_1\rangle}\tfrac{\xi_1}{|\xi_1|}\tfrac{1}{(2+|\xi_1|^2+|\xi_2|^2)^2} - 2 \tfrac{1}{\langle \xi_1\rangle^3}\tfrac{\xi_1}{|\xi_1|}\tfrac{1}{2+|\xi_1|^2+|\xi_2|^2}
\]
is Coifman--Meyer--Mikhlin.  Thus, we can use Lemma~\ref{lem:UL2} and Lemma~\ref{lem:decay} to estimate
\[
\|b_2[U^{-2}v, U^{-1}v]\|_{L_x^2} \lesssim \|U^{-2}v\|_{L_x^6} \|U^{-1}v\|_{L_x^3} \lesssim \langle s\rangle^{-\frac{7}{18}} \|v\|_{Z(I)}^2.
\]

For the third term in \eqref{three terms}, we split
\[
U^{-1} Jv = U^{-1}P_{\leq 1}(Jv) + U^{-1} P_{>1}(Jv).
\]
We estimate the contribution of the first summand by
\[
\| B[U^{-1} P_{\leq 1}(Jv), U^{-1}v]\|_{L_x^2} \lesssim \| U^{-1}P_{\leq 1}(Jv)\|_{L_x^6} \|U^{-1} v\|_{L_x^3} \lesssim \langle s\rangle^{-\frac7{18}} \|v\|_{Z(I)}^2.
\]
For the second summand, we note that $\xi B(\xi_1,\xi_2)$ is still a Coifman--Meyer multiplier. Thus
\begin{align*}
\| B[U^{-1}P_{>1}(Jv), U^{-1}v]\|_{L_x^2}& \lesssim \|\nabla B[U^{-1}P_{>1}(Jv), U^{-1}v]\|_{L_x^{\frac65}} \\
& \lesssim \|Jv\|_{L_x^2} \|U^{-1}v\|_{L_x^3} \lesssim \langle s\rangle^{-\frac7{18}}\|v\|_{Z(I)}^2.
\end{align*}

Finally, we consider \eqref{weighted1-3}. For this, we again use \eqref{Jcommute}.  The nonlinearity consists of two types of terms, namely, the original terms coming from $N_v(u)$ and the terms coming from the normal form.

For the original terms, we recall from \eqref{expandJ} that we may write
\begin{equation}\label{weighted-splitJ}
J(s)=x+is\jbb
\end{equation}
and we note that the commutator is given by $[J,U] = \tfrac{2}{\jb^3}\tfrac{\nabla}{|\nabla|}.$  Applying Proposition~\ref{prop:strichartz}, we get
\begin{equation}\label{weighted-higher1}
\biggl\|\int_{t_0}^t e^{isH} J(s)\text{\O}[U(u^k) + u^k](s) \biggr\|_{L_x^2} \lesssim \|xu^k\|_{N(I)} + \|u^k\|_{N(I)} + \|t\langle\tilde \nabla\rangle(u^k)\|_{N(I)},
\end{equation}
where $k\in \{3,4,5\}$ and $N(I)$ denotes any dual Strichartz norm.

For the terms coming from the normal form, it suffices to estimate
\[
\biggl\| \int_{t_0}^t e^{isH} J(s) B[u, u^k](s)\,ds\biggr\|_{L_x^2} \qtq{with} k\in\{2,3,4,5\}.
\]
Once again, we use \eqref{weighted-splitJ} and write
\begin{equation}\label{weighted-higher2}
JB[u,u^k] = B[xu,u^k] + \tilde B[u,u^k] + it\jbb B[u,u^k],
\end{equation}
where $\tilde B$ has the symbol $\tilde B(\xi_1,\xi_2)=\nabla_\xi B(\xi_1,\xi_2)$.  Note that $\tilde B$ and $\jbb B$ are Coifman--Meyer--Mikhlin operators.

Applying Strichartz, it remains to estimate the terms on the right-hand sides of \eqref{weighted-higher1} and \eqref{weighted-higher2} in any dual Strichartz space.  Note that by \eqref{dual-est}, the contribution of $\|u^k\|_{N(I)}$ with $k\in\{3,4,5\}$ is acceptable.

We first consider the cubic contributions in the remaining terms.  By Lemma~\ref{lem:weighted} and Lemma~\ref{lem:decay}, we may estimate as follows:
\begin{align*}
\|xu^3\|_{L_t^2 L_x^{\frac65}} {+} \| B[xu,u^2]\|_{L_t^2 L_x^{\frac65}}
& \lesssim \bigl\| \|xu_{\leq 1}\|_{L_x^6} \|u\|_{L_x^2} \|u\|_{L_x^6} \bigr\|_{L_t^2} + \bigl\| \|xu_{>1}\|_{L_x^2} \|u\|_{L_x^6}^2 \bigr\|_{L_t^2}  \\
& \lesssim \|\langle t\rangle^{-\frac59}\|_{L_t^2} \|v\|_{Z(I)}^3 \lesssim  \langle t_0\rangle^{-\frac{1}{18}}\|v\|_{Z(I)}^3.
\end{align*}
Next, by Proposition~\ref{prop:CM} and Lemma~\ref{lem:str},
\begin{align*}
\|t\jbb(u^3)\|_{L_t^1 L_x^2}& + \|\tilde B[u,u^2]\|_{L_t^1L_x^2} + \|t\jbb B[u,u^2]\|_{L_t^1 L_x^2} \\
& \lesssim \bigl\|\langle t\rangle \|u\|_{L_x^6}^2\bigr\|_{L_t^2} \|\jb u\|_{L_t^2 L_x^6}  \lesssim \|\langle t\rangle^{-\frac59}\|_{L_t^2} \|v\|_{Z(I)}^3 \lesssim \langle t_0\rangle^{-\frac{1}{18}}\|v\|_{Z(I)}^3.
\end{align*}

Let us consider now the quartic contributions.  Arguing as above,
\begin{align*}
\|xu^4\|_{L_t^2 L_x^{\frac65}}&+\| B[xu,u^3]\|_{L_t^2 L_x^{\frac65}}\\
& \lesssim \bigl[\bigl\| \|xu_{\leq 1}\|_{L_x^6}\|u\|_{L_x^2}\|u\|_{L_x^6}^{\frac12} \bigr\|_{L_t^\infty} + \bigl\| \|xu_{>1}\|_{L_x^2}\|u\|_{L_x^6}^{\frac32}\bigr\|_{L_t^\infty}\bigr]\|u\|_{L_t^3 L_x^{18}}^{\frac32} \\
& \lesssim \langle t_0\rangle^{-\frac16}\|v\|_{Z(I)}^{\frac52}\| |\nabla|^{\frac23} u\|_{L_t^3 L_x^{\frac{18}{5}}}^{\frac32} \lesssim \langle t_0\rangle^{-\frac16}\|v\|_{Z(I)}^4
\end{align*}
and
\begin{align*}
\|t\jbb(u^4)&\|_{L_t^2 L_x^{\frac65}}   + \|\tilde B[u,u^3]\|_{L_t^2 L_x^{\frac65}} + \|t\jbb B[u,u^3]\|_{L_t^2 L_x^{\frac65}}  \\
& \lesssim \|\langle t\rangle\|u\|_{L_x^6}^{\frac97}\|_{L_t^\infty}^{\phantom{\frac11}} \|u\|_{L_t^\infty L_x^{\frac{72}{19}}}^{\frac{12}{7}}\|\jb u\|_{L_t^2 L_x^6} \lesssim \langle t_0\rangle^{-\frac{17}{18}} \|v\|_{Z(I)}^4.
\end{align*}

We now turn to the quintic contributions.  Arguing as above,
\begin{align*}
\|xu^5\|_{L_t^2 L_x^{\frac65}}&+ \|B[xu, u^4]|_{L_t^2 L_x^{\frac65}} \\
& \lesssim \bigl(\| xu_{\leq 1}\|_{L_x^6} \|u\|_{L_x^2} \|u\|_{L_x^6}^{\frac12}\bigr\|_{L_t^\infty} + \bigl\| \| xu_{>1}\|_{L_x^2} \|u\|_{L_x^6}^{\frac32}\bigr\|_{L_t^\infty} \bigr)\|u\|_{L_t^5 L_x^{30} }^{\frac52} \\
& \lesssim \langle t_0\rangle^{-\frac16}\|v\|_{Z(I)}^{\frac52} \|\nabla u\|_{L_t^5 L_x^{\frac{30}{11}}}^{\frac52} \lesssim
\langle t_0\rangle^{-\frac16}\|v\|_{Z(I)}^5.
\end{align*}
Next,
\begin{align*}
\|t\jbb(u^5)&\|_{L_t^2 L_x^{\frac65}}   + \|\tilde B[u,u^4]\|_{L_t^2 L_x^{\frac65}} + \|t\jbb B[u,u^4]\|_{L_t^2 L_x^{\frac65}}  \\
& \lesssim \|\langle t\rangle\|u\|_{L_x^6}^{\frac97}\|_{L_t^\infty} \|u\|_{L_t^\infty L_x^{6}}^{\frac{19}{7}}\|\jb u\|_{L_t^2 L_x^6} \lesssim \langle t_0\rangle^{-\frac{19}{9}} \|v\|_{Z(I)}^5.
\end{align*}

Finally, we consider the sextic contribution, which comes only from the normal form.  We rely on Sobolev embedding and the fact that $\xi B(\xi_1,\xi_2)$ is Coifman--Meyer.  We get
\begin{align*}
\|B[xu,u^5] \|_{L_t^1 L_x^2}& \lesssim \|\nabla B[xu,u^5]\|_{L_t^1 L_x^{\frac65}}  \\
& \lesssim \bigl( \bigl\| \|xu_{\leq 1}\|_{L_x^6} \|u\|_{L_x^2} \|u\|_{L_x^6}^{\frac12}\bigr\|_{L_t^\infty} + \bigl\| \|xu_{>1}\|_{L_x^2} \|u\|_{L_x^6}^{\frac32}\bigr\|_{L_t^\infty}\bigr)\|u\|_{L_t^{\frac72} L_x^{42}}^{\frac72} \\
& \lesssim \langle t_0\rangle^{-\frac16}\|v\|_{Z(I)}^{\frac52} \||\nabla|^{\frac67} u\|_{L_t^{\frac72} L_x^{\frac{42}{13}}}^{\frac72} \lesssim \langle t_0\rangle^{-\frac16} \|v\|_{Z(I)}^6.
\end{align*}
Using that $\xi \tilde B(\xi_1,\xi_2)$ and $\xi\widetilde{\langle \xi\rangle} B$ are Coifman--Meyer--Mikhlin, we estimate
\begin{align*}
\|\tilde B[u,u^5]\|_{L_t^1 L_x^2} + \|t\jbb B[u,u^5]\|_{L_t^1 L_x^2} & \lesssim \|\nabla \tilde B[u,u^5]\|_{L_t^1 L_x^{\frac65}} + \|t\nabla \jbb B[u,u^5]\|_{L_t^1 L_x^{\frac65}}  \\
& \lesssim \| \langle t\rangle \|u\|_{L_x^6}^{\frac32}\|_{L_t^\infty} \|u\|_{L_t^\infty L_x^2} \|u\|_{L_t^{\frac72}L_x^{42}}^{\frac72} \\
& \lesssim \langle t_0\rangle^{-\frac16}\|v\|_{Z(I)}^6.
\end{align*}
This completes the proof of Proposition~\ref{prop:weighted1}.
\end{proof}


We turn to estimating the contribution of the quadratic terms arising from the nonlinearity in \eqref{nf-duh}, that is,
\[
\biggl\| x\int_{t_0}^t e^{isH} \N_2\,ds\biggr\|_{L_x^2},
\]
where we recall
\[
\N_2 = U\bigl(A_1[v_1,v_1] +  A_2[v_2,v_2]\bigr),\quad  v_1 = \tfrac12(v+\bar v), \qtq{and} v_2 = \tfrac1{2i}(v-\bar v).
\]

Thus, by Plancherel we are led to estimate terms of the form
\[
\biggl\| \nabla_\xi\int_{t_0}^t e^{is\Phi}U(\xi) A(\xi_1,\xi_2) \wh{\tilde f}(\xi_1)\wh{\tilde f}(\xi_2)\,d\xi_2\,ds\biggr\|_{L_\xi^2},
\]
where $f(t) = e^{itH}v(t)$, $\Phi=H(\xi)\pm H(\xi_1)\pm H(\xi_2)$, and $A\in\{A_1,A_2\}$ (cf. \eqref{Aj}).  In the remainder of this section, we consider the easier terms in which the derivative misses the phase.

\begin{proposition}\label{prop:weighted2} Let $t_0>0$ and let $I=[t_0,1]$ if $t_0<1$ and $I=[t_0,2t_0]$ otherwise.  Let $v:I\times\R^3\to\C$ be a solution to \eqref{E:cqv} and $f(t)=e^{itH}v(t)$. For all $t\in I$,
\[
\biggl\| \int_{t_0}^t \int_{\R^3}e^{is\Phi} \nabla_\xi[U(\xi)A(\xi_1,\xi_2)\wh{\tilde f}(\xi_1)\wh{\tilde f}(\xi_2)]\,d\xi_2\,ds\biggr\|_{L_\xi^2} \lesssim \langle t_0\rangle^{-\frac14} \|v\|_{Z(I)}^2,
\]
where $A\in\{A_1,A_2\}$ as in \eqref{Aj}.
\end{proposition}

\begin{proof}  If the derivative lands on $\wh{\tilde f}(\xi_1)$, then we apply Proposition~\ref{prop:strichartz}, Lemma~\ref{AOK}, and Lemma~\ref{lem:decay} to estimate
\[
\| U A[Jv, v]\|_{L_t^{\frac43} L_x^{\frac32}} \lesssim \| v\|_{L_t^{\frac43} L_x^6} \|Jv\|_{L_t^\infty L_x^2}
\lesssim \langle t_0\rangle^{-\frac14} \|v\|_{Z(I)}^2.
\]

We next consider the operators
\[
\nabla_\xi[U(\xi)A_j(\xi_1,\xi_2)],\quad j\in\{1,2\}.
\]
Distributing the derivative, we first consider the term in which $\nabla_\xi$ hits the factor $\frac{\xi_1}{|\xi_1|}$ in the multiplier $A_2$.  One can verify that this term is of the form
\[
\tilde A(\xi_1,\xi_2) = \tfrac{1}{|\xi_1|} \tfrac{\langle \xi_1\rangle\langle\xi_2\rangle}{2+|\xi_1|^2+|\xi_2|^2} b(\xi_1,\xi_2),
\]
where $b$ is Coifman--Meyer--Mikhlin.  Thus, by Lemmas~\ref{AOK}, \ref{UL2}, and \ref{lem:decay},
\[
\| \tilde A[v,v]\|_{L_t^{\frac43} L_x^{\frac32}} \lesssim \|v\|_{L_t^{\frac43} L_x^6} \|U^{-1}v\|_{L_t^\infty L_x^2} \lesssim \langle t_0\rangle^{-\frac14} \|v\|_{Z(I)}^2.
\]
All remaining terms are either Coifman--Meyer--Mikhlin symbols or products of Coifman--Meyer--Mikhlin symbols with the multiplier
\[
\tfrac{\langle \xi_1\rangle\langle\xi_2\rangle}{2+|\xi_1|^2+|\xi_2|^2}.
\]
Thus, by Proposition~\ref{prop:CM}, Lemma~\ref{AOK}, and Lemma~\ref{lem:decay}, the contribution of these terms is bounded by
\[
\|v\|_{L_t^{\frac43} L_x^6} \|v\|_{L_t^\infty L_x^2} \lesssim \langle t_0\rangle^{-\frac14} \|v\|_{Z(I)}^2.
\]
This completes the proof of the proposition.\end{proof}

\section{The weighted norm, part II}\label{S:W2}

In this section we begin to estimate the contribution of the quadratic nonlinearity when the derivative hits the phase.  The goal of the next several sections is to prove the following:

\begin{proposition}\label{prop:bs-quad} Let $t_0>0$ and let $I=[t_0,1]$ if $t_0<1$ and $I=[t_0,2t_0]$ otherwise.  There exists $\eps>0$ such that for all $t\in I$,
\begin{equation}\label{bs-quad}
\biggl\| \int_{t_0}^t \int e^{is\Phi}[s\nabla_\xi\Phi] U(\xi)A(\xi_1,\xi_2)\wh{\tilde{f}}(\xi_1)\wh{\tilde{f}}(\xi_2)\,d\xi_2\,ds \biggr\|_{L_\xi^2} \lesssim  \langle t_0\rangle^{-\eps}\sum_{k=2}^6 \|v\|_{Z(I)}^k,
\end{equation}
where $A\in\{A_1,A_2\}$ as in \eqref{Aj} and $\Phi=H(\xi)\pm H(\xi_1)\pm H(\xi_2)$.
\end{proposition}

We briefly recall the general strategy.  We consider separately the phases arising from the $\bar v^2$, $v^2$, and $|v|^2$ terms.  For each phase, we decompose $\R_{\xi_1}^3\times\R_{\xi_2}^3$ into regions on which we have `non-resonance', which refers to the non-vanishing of $\Phi$ or some particular derivative of $\Phi$.  We carry out these decompositions in Section~\ref{dec1}, Section~\ref{dec2}, and Section~\ref{dec3}.  On each such region, we perform an integration by parts in the term
\[
\int_{t_0}^t \int e^{is\Phi}[s\nabla_\xi\Phi] U(\xi) A(\xi_1,\xi_2)\wh{\tilde{f}}(\xi_1)\wh{\tilde{f}}(\xi_2)\,d\xi_2\,ds
\]
using an identity that capitalizes on the particular from of non-resonance; see \eqref{id-1}, \eqref{id-2}, and \eqref{id-3} below.  The terms arising from integration by parts using a given identity are all of a similar form; we estimate these in Sections~\ref{est-tnr}--\ref{est-anr}, relying heavily on Proposition~\ref{prop:bilinear} and Lemma~\ref{lem:decay}.

We will always perform our decompositions and integrations by parts at fixed frequencies $|\xi_j|\sim N_j$.  As $P_N$ are not true projections, we will also use the fattened Littlewood--Paley projections $\tilde P_N$, which ensure that $P_N=\tilde P_N P_N$.

As noted in the introduction, many of the bilinear multipliers used to partition frequency space will not not obey uniform bounds; indeed in some cases they constitute the worst term.

\subsection{Time non-resonance}\label{S:tnr}  Time non-resonance refers to the non-vanishing of $\Phi$.  Suppose $\chi^T$ is a cutoff to a region on which $\Phi\neq 0$.  We wish to estimate the contribution of
\[
\sum_{N_1,N_2}\Bigl\|\int_{t_0}^t \!\int\! e^{is\Phi}\chi^T(\xi_1,\xi_2)\tilde\psi(\tfrac{\xi_1}{N_1}) \tilde\psi(\tfrac{\xi_2}{N_2})[s\nabla_\xi\Phi]
    U(\xi) A(\xi_1,\xi_2)\wh{\tilde{f}_{N_1}}(\xi_1)\wh{\tilde{f}_{N_2}}(\xi_2) d\xi_2 ds\Bigr\|_{L^2_\xi}
\]
to \eqref{bs-quad}. As $\Phi\neq 0$ on the support of $\chi^T$, we may use the identity
\begin{equation}\label{id-1}
e^{is\Phi} = \partial_s \frac{e^{is\Phi}}{i\Phi}
\end{equation}
to integrate by parts in the term above.  Defining the multiplier
\begin{equation}\label{bT}
b^T_{N_1,N_2}(\xi_1,\xi_2) =\chi^T(\xi_1,\xi_2) \tilde\psi(\tfrac{\xi_1}{N_1}) \tilde\psi(\tfrac{\xi_2}{N_2})A(\xi_1,\xi_2)U(\xi)\frac{\nabla_\xi\Phi}{\Phi},
\end{equation}
we see that it suffices to estimate the following terms:
\begin{align}
&\left.\begin{aligned}
&\sum_{N_1,N_2}\bigl\| s e^{isH} b^T_{N_1,N_2}[\tilde v_{N_1}(s), \tilde v_{N_2}(s)]\bigr\|_{L_x^2}, \quad s\in\{t_0,t\}, \\
&\sum_{N_1,N_2} \int_{t_0}^t \bigl\| e^{isH} b^T_{N_1,N_2}[\tilde v_{N_1}(s),\tilde v_{N_2}(s)]\bigr\|_{L_x^2}\,ds,
\end{aligned}\qquad\right\}\label{tnr} \\
&\sum_{N_1\!,N_2}\biggl\| \int_{t_0}^t\! s e^{isH}\!\bigl\{b^T_{N_1,N_2}[\tilde v_{N_1}(s), P_{N_2} N_v(u(s))]+b^T_{N_1,N_2}[ P_{N_1}N_v(u(s)), \tilde v_{N_2}(s)]\bigr\} ds\biggr\|_{L_x^2}\label{tnr2}
\end{align}

Note that because we are working on dyadic time intervals, we can estimate the two terms in \eqref{tnr} in the same way.  Note also that to arrive at \eqref{tnr2}, we have used the fact that
\[
i\partial_t [e^{itH} v] = e^{itH}{(i\partial_t - H)v}.
\]

\subsection{Space non-resonance}\label{S:snr} Space non-resonance refers to the non-vanishing of $\nabla_{\xi_2}\Phi$.  Suppose $\chi^X$ is a cutoff to a region on which $\nabla_{\xi_2}\Phi\neq 0$. We wish to estimate the contribution of
\[
\sum_{N_1\!,N_2} \!\Bigr\|\int_{t_0}^t \int\! e^{is\Phi}\chi^X(\xi_1,\xi_2)\tilde\psi(\tfrac{\xi_1}{N_1}) \tilde\psi(\tfrac{\xi_2}{N_2})[s\nabla_\xi\Phi]
    U(\xi) A(\xi_1,\xi_2)\wh{\tilde f_{N_1}}(\xi_1)\wh{\tilde f_{N_2}}(\xi_2) d\xi_2 ds\Bigl\|_{L_\xi^2}
\]
to \eqref{bs-quad}. As $\nabla_{\xi_2}\Phi\neq 0$ on the support of $\chi^X$, we may integrate by parts using the identity
\begin{equation}\label{id-2}
e^{is\Phi} = \frac{\nabla_{\xi_2}\Phi}{is|\nabla_{\xi_2}\Phi|^2}\cdot\nabla_{\xi_2} e^{is\Phi}.
\end{equation}
Defining the multipliers
\begin{align}
& b^X_{N_1,N_2}(\xi_1,\xi_2) = \chi^X(\xi_1,\xi_2) \tilde\psi(\tfrac{\xi_1}{N_1}) \tilde\psi(\tfrac{\xi_2}{N_2})A(\xi_1,\xi_2)U(\xi)\frac{\nabla_\xi\Phi\nabla_{\xi_2}\Phi}{ |\nabla_{\xi_2}\Phi|^{2}},\label{bX} \\
& \tilde b^X_{N_1,N_2}(\xi_1,\xi_2) = \nabla_{\xi_2}\cdot b^X_{N_1,N_2}(\xi_1,\xi_2), \label{tbX}
\end{align}
we see that it suffices to estimate the following terms:
\begin{align}
& \sum_{N_1,N_2}\biggl\|\int_{t_0}^t e^{isH} \bigl\{ b^X_{N_1,N_2}[P_{N_1}\widetilde{Jv}(s),\tilde v_{N_2}(s)] + b^X_{N_1,N_2}[\tilde v_{N_1}(s),P_{N_2}\widetilde{Jv}(s)]\bigr\}\,ds\biggr\|_{L_x^2}, \label{snr1} \\
& \sum_{N_1,N_2}\biggl\|\int_{t_0}^t e^{isH} \tilde b^X_{N_1,N_2}[\tilde v_{N_1}(s), \tilde v_{N_2}(s)]\,ds\biggr\|_{L_x^2}. \label{snr2}
\end{align}

\subsection{Angular non-resonance}\label{S:anr}  For the phase corresponding to the $|v|^2$ nonlinearity, namely,
\[
\Phi = H(\xi)+H(\xi_2)-H(\xi_1),
\]
we will use a different identity to integrate by parts on one region of frequency space.  We adopt the following notation for the projections of $\xi$ into the directions orthogonal to $\xi_1$ and $\xi_2$:
\[
\xi^{\perp_j} = \xi-(\xi\cdot\tunit{\xi_j})\tunit{\xi_j}.
\]

We will refer to the non-vanishing of $\xi^{\perp_2}\cdot\nabla_{\xi_2}\Phi$ as angular non-resonance.  Supposing $\chi^\angle$ is a cutoff to a region on which $\xi^{\perp_2}\cdot\nabla_{\xi_2}\Phi\neq 0$, we wish to estimate the contribution of
\[
\sum_{N_1,N_2}\biggl\|\int_{t_0}^t\!\int\! e^{is\Phi} \chi^\angle(\xi_1,\xi_2)\tilde\psi(\tfrac{\xi_1}{N_1})\tilde\psi(\tfrac{\xi_2}{N_2})
    [s\nabla_\xi\Phi]U(\xi)A(\xi_1,\xi_2)\wh{f_{N_1}}(\xi_1)\wh{\bar f_{N_2}}(\xi_2) d\xi_2 ds\Big\|_{L^2_\xi}
\]
to \eqref{bs-quad}.  As $\nabla_{\xi_2}\Phi = \nabla H(\xi_1)+\nabla H(\xi_2)$, one has
\[
\xi^{\perp_2}\cdot \nabla_{\xi_2}\Phi = \xi^{\perp_2}\cdot \nabla H(\xi_1).
\]
Thus, as $\xi^{\perp_2}\cdot\nabla_{\xi_2}\Phi\neq 0$ on the support of $\chi^\angle$, we may integrate by parts in the term above using the identity
\begin{equation}\label{id-3}
e^{is\Phi} = \frac{1}{is(\xi^{\perp_2}\cdot \nabla H(\xi_1))}\xi^{\perp_2}\cdot\nabla_{\xi_2}e^{is\Phi}.
\end{equation}
Defining
\begin{align}
&m_{N_1,N_2}(\xi_1,\xi_2) =\chi^\angle(\xi_1,\xi_2)\tilde\psi(\tfrac{\xi_1}{N_1})\tilde\psi(\tfrac{\xi_2}{N_2}) A(\xi_1,\xi_2) U(\xi)\tfrac{\nabla_\xi\Phi}{\xi^{\perp_2}\cdot \nabla H(\xi_1)},\label{mA} \\
&b_{N_1,N_2}^{\angle}(\xi_1,\xi_2)=\nabla_{\xi_2}\cdot[m_{N_1,N_2}(\xi_1,\xi_2)\xi^{\perp_2}],\label{bA}
\end{align}
we are first led to estimate the following term, which arises when the divergence misses both copies of $f$:
\begin{equation}\label{anr1}
\sum_{N_1,N_2} \biggl\| \int_{t_0}^t e^{isH}  b_{N_1,N_2}^\angle[ P_{N_1}v(s),P_{N_2}\bar{v}(s)]\,ds \biggr\|_{L_x^2}.
\end{equation}

The terms in which the divergence hits a copy of $f$ take the form
\[
\int_{t_0}^t\int e^{is\Phi} m_{N_1,N_2}(\xi_1,\xi_2)\bigl\{[\xi^{\perp_2}\cdot \nabla_{\xi_2}\wh{f}(\xi_1)]\whb{f}(\xi_2)+ \wh{f}(\xi_1)[\xi^{\perp_2}\cdot\nabla_{\xi_2}\whb{f}(\xi_2)]\bigr\}\,d\xi_2\,ds.
\]
For these terms, we first rely on the vector identity
\[
(x\times y)\cdot (z\times w) = (x\cdot z)(y\cdot w) - (y\cdot z)(x\cdot w)
\]
to write
\begin{align*}
\xi^{\perp_2}\cdot\nabla_{\xi_2}\whb{f}(\xi_2) & = \tfrac{1}{|\xi_2|^2} \bigl(\xi_2\cdot\xi_2\bigr) \bigl(\xi^{\perp_2}\cdot\nabla_{\xi_2}\whb{f}(\xi_2)\bigr) \\
& =\tfrac{1}{|\xi_2|^2}\bigl[ (\xi_2\times\xi^{\perp_2})\cdot(\xi_2\times\nabla_{\xi_2}\whb{f}(\xi_2)) + (\xi_2\cdot \xi^{\perp_2})(\xi_2\cdot\nabla_{\xi_2}\whb{f}(\xi_2)\bigr] \\
& = \tfrac{\xi_2\times\xi^{\perp_2}}{|\xi_2|^2}\cdot \bigl(\xi_2\times\nabla_{\xi_2}\whb{f}(\xi_2)\bigr).
\end{align*}
Exploiting $\xi_j\times\xi^{\perp_j}=\xi_j\times\xi$ and the radiality of $H(\xi)$, this further reduces to
\[
\xi^{\perp_2}\cdot\nabla_{\xi_2}\whb{f}(\xi_2) = e^{-isH(\xi_2)}\tfrac{\xi_2\times\xi}{|\xi_2|^2}\cdot (\xi_2\times\nabla_{\xi_2})\whb{v}(\xi_2).
\]
We next write
\[
\xi^{\perp_2}\cdot\nabla_{\xi_2}\bigl(\wh{f}(\xi_1)\bigr) = -\xi^{\perp_1}\cdot(\nabla_{\xi_1}\wh{f})(\xi_1) + \bigl(\xi\cdot\tunit{\xi_1}\tunit{\xi_1}-\xi\cdot\tunit{\xi_2}\tunit{\xi_2}\bigr)\nabla \wh{f}(\xi_1).
\]

Thus, defining
\begin{align}
& b^\angle_{j, N_1,N_2}(\xi_1,\xi_2) = m_{N_1,N_2}(\xi_1,\xi_2) \tfrac{\xi_j\times\xi}{|\xi_j|^2},\label{bfj} \\
& \tilde b^\angle_{N_1,N_2}(\xi_1,\xi_2) = m_{N_1,N_2}(\xi_1,\xi_2)\bigl[\xi\cdot\tunit{\xi_1}\tunit{\xi_1}-\xi\cdot\tunit{\xi_2}\tunit{\xi_2}\bigr],\label{tbA}
\end{align}
we are led to estimate the following terms:
\begin{align}
&\sum_{N_1,N_2}\biggl\| \int_{t_0}^t e^{isH} b_{1,N_1,N_2}^\angle[P_{N_1}(x\times\nabla)v(s),P_{N_2}\bar{v}(s)]\,ds\biggr\|_{L_x^2},\label{anr2} \\
&\sum_{N_1,N_2}\biggl\| \int_{t_0}^t e^{isH} b_{2, N_1,N_2}^\angle[P_{N_1}v(s), P_{N_2}(x\times\nabla)\bar{v}(s)]\,ds\biggr\|_{L_x^2},\label{anr3}\\
&\sum_{N_1,N_2} \biggl\| \int_{t_0}^t e^{isH} \tilde b^\angle_{N_1,N_2}[P_{N_1}Jv(s), P_{N_2}\bar v(s)]\,ds\biggr\|_{L_x^2}.\label{anr4}
\end{align}


\section{Non-resonant decompositions I: Preliminaries and $\bar v^2$ terms}\label{sec:dn}

The next three sections are devoted to the decompositions of frequency space into non-resonant regions for the phases corresponding to each quadratic nonlinear term.

The following lemma illustrates how we estimate many of the multipliers appearing below.

\begin{lemma}[Typical estimate] Consider a multiplier of the form
\[
b(\xi_1,\xi_2) = \tilde{\psi}(\tfrac{\xi_1}{N_1})\tilde{\psi}(\tfrac{\xi_2}{N_2}) m(\xi_1,\xi_2).
\]
Suppose $N_2\leq N_1$, and that for fixed $\xi$ the multiplier $b(\xi-\xi_2,\xi_2)$ vanishes except for those $\xi_2$ on a set of volume $\lesssim N_2^3\theta^2$.  Suppose further that
\[
\bigl| \partial_{\xi_2}^\alpha m \bigr|\lesssim A\cdot \ell^{|\alpha|},\quad|\alpha|\leq 2.
\]
Then
\[
\op{b}\lesssim A\cdot N_2^{\frac32}\theta \cdot \max\bigl\{ \tfrac{1}{N_2},\ell\bigr\}^{\frac32},
\]
where we recall the notation $\op{\cdot}$ from \eqref{bnorm}.
\end{lemma}
\begin{proof} Note that for $j\in\{1,2\}$, we have
\begin{equation}\label{cutoff-loss}
\bigl| \partial_{\xi_j}^\alpha\bigl[ \tilde{\psi}(\tfrac{\xi_1}{N_1})\tilde{\psi}(\tfrac{\xi_2}{N_2})\bigr]\bigr| \lesssim \bigl(\tfrac{1}{N_1\wedge N_2}\bigr)^{|\alpha|}.
\end{equation}
By the Leibniz rule, \eqref{cutoff-loss}, and H\"older's inequality, we  arrive at the estimate
\[
\opt{b} \lesssim A\cdot N_2^{\frac32}\theta \cdot \max\bigl\{ \tfrac{1}{N_2}, \ell\bigr\}^{\frac32},
\]
which gives the result. \end{proof}

In general, we will use cutoffs to decompose frequency space into regions of time, space, or angular non-resonance, as quantified by the bounds we can prove for $\Phi$ and its derivatives. We will also need to keep track of the bounds satisfied by the derivatives of the cutoff functions.

We will consistently use the following notation.

\begin{definition} For $x,y\in\R^3\backslash\{0\}$ we let $\angle(x,y)\in[0,\pi)$ be the angle defined via $\cos\angle (x,y) = \tfrac{x}{|x|}\cdot \tfrac{y}{|y|}$. We define
\[
\theta_{02}=\angle(\xi,\xi_2), \quad \theta_{01}=\angle(\xi,\xi_1), \quad \theta_{12}=\angle(\xi_1,\xi_2),
\]
where $\xi=\xi_1+\xi_2$.  We also denote $\theta'=\pi-\theta$.  Note that $\theta_{02}+\theta_{01}+\theta_{12}'=\pi$.
\end{definition}

For reference, we collect some bounds on the derivatives of $H$, all of which follow from explicit computation.

\begin{lemma}[Derivative bounds for $H$]\label{H-derivatives} Write $H(\xi)=h(|\xi|),$ where
\[
h(r) = r\langle r\rangle = r(2+r^2)^{1/2}.
\]
Then
\[
h'(r) \sim \langle r\rangle,\quad h''(r)\sim \tfrac{r}{\langle r\rangle}, \quad h'''(r) \sim \tfrac{1}{\langle r\rangle^5}, \quad h^{(4)}(r) \sim \tfrac{r}{\langle r\rangle^7}.
\]
Consequently,
\[
|\partial_\xi^\alpha H(\xi) | \lesssim \langle \xi\rangle (\tfrac{1}{|\xi|})^{|\alpha|-1},\quad  |\alpha|\leq 4.
\]
\end{lemma}

The next lemma will be of frequent use in the following sections.

\begin{lemma}\label{DH-diff} For any $x,y\in\R^3$,
\[
|\nabla H(x)\pm \nabla H(y)| \lesssim \bigl| |x|-|y| \bigr|\bigl(\tfrac{|x|}{\langle x\rangle} \vee \tfrac{|y|}{\langle y\rangle}\bigr) + \bigl(\langle x\rangle\wedge \langle y\rangle\bigr)\sin(\tfrac12\angle(x,\mp y)).
\]
By the triangle inequality, one can replace $\bigl| |x|-|y|\bigr|$ by $|x-y|$.
\begin{proof} Without loss of generality, assume $|y|\leq |x|$.  Write
\[
\nabla H(x)\pm \nabla H(y) = \tunit{x}\bigl[h'(|x|)-h'(|y|)\bigr] + h'(|y|)\bigl(\tunit{x}\pm \tunit{y}\bigr).
\]
By the fundamental theorem of calculus and Lemma~\ref{H-derivatives},
\[
\bigl| h'(|x|)-h'(|y|)\bigr| \lesssim \bigl||x|-|y|\bigr|\bigl(\tfrac{|x|}{\langle x\rangle}\vee \tfrac{|y|}{\langle y\rangle}\bigr).
\]
As $h'(|y|)\sim\langle y\rangle$ and
\[
\bigl| \tunit{x}\pm\tunit{y} \bigr| = 2\sin(\tfrac12\angle(x,\mp y)),
\]
the result follows. \end{proof}
\end{lemma}

\subsection{Non-resonant decomposition for $\bar v^2$ terms}\label{dec1}

The `decomposition' for $\bar v^2$ terms is particularly simple, due to the fact that there is always time non-resonance.

\begin{proposition}[Non-resonant decomposition for $\bar v^2$ terms]\label{prop:dec1} Fix $N_1,N_2$.  Define
\begin{align*}
\Phi &= H(\xi)+H(\xi_1)+H(\xi_2), \\
b^T(\xi_1,\xi_2) &=\tilde\psi(\tfrac{\xi_1}{N_1}) \tilde\psi(\tfrac{\xi_2}{N_2})A(\xi_1,\xi_2)U(\xi)\frac{\nabla_\xi\Phi}{\Phi}.
\end{align*}
Then
\begin{equation}\label{bT-bound-vbar2}
\op{b^T} \lesssim \frac{1}{N_1\vee N_2}.
\end{equation}
\end{proposition}

\begin{remark}
The multiplier $b^T$ does depend on $N_1$ and $N_2$.  This will be the case with all the multipliers in the following sections.  We have elected to omit any explicit reference to this dependence in our notation to reduce clutter.
\end{remark}

\begin{proof}  Using Lemma~\ref{H-derivatives}, one can check that
\begin{align*}
|\Phi| &\gtrsim N_1\langle N_1\rangle \vee N_2 \langle N_2\rangle, \\
|\partial_{\xi_j}^\alpha [\nabla_\xi\Phi +\nabla_{\xi_j}\Phi]| &\lesssim \bigl(\langle N_1 \rangle \vee \langle N_2\rangle\bigr) \bigl(\tfrac{1}{N_1\wedge N_2}\bigr)^{|\alpha|}, \quad |\alpha|\leq 2.
\end{align*}
Thus,
\[
\bigl| \partial_{\xi_j}^\alpha \tfrac{\nabla_\xi\Phi}{\Phi}\bigr| \lesssim \tfrac{1}{N_1\vee N_2}\bigl(\tfrac{1}{N_1\wedge N_2}\bigr)^{|\alpha|},\quad j\in\{1,2\},\quad |\alpha|\leq 2.
\]
Using this together with \eqref{cutoff-loss} and \eqref{Ajbds}, we deduce \eqref{bT-bound-vbar2}.
\end{proof}


\section{Non-resonant decompositions II: $v^2$ terms}\label{dec2}
In this section we carry out the decomposition into non-resonant regions for the $v^2$ terms in \eqref{bs-quad}.  This corresponds to the phase $\Phi = H(\xi)-H(\xi_1)-H(\xi_2)$.  Recall that we always work at fixed frequencies $|\xi_j|\sim N_j$.  By symmetry, it suffices to consider the case
\begin{equation}\label{vvN2N1}
N_2\leq N_1.
\end{equation}

\begin{proposition}[Non-resonant decomposition for $v^2$ terms]\label{prop:dec2}\leavevmode\ Given $N_2\leq N_1$, there exists a decomposition $1 = \sum_{j=1}^4 \rho_j$ such that the following holds:
the multipliers
\begin{align*}
&b^T_j = \tilde{\psi}(\tfrac{\xi_1}{N_1})\tilde{\psi}(\tfrac{\xi_2}{N_2})\rho_j(\xi_1,\xi_2)A(\xi_1,\xi_2)U(\xi)\tfrac{\nabla_\xi \Phi}{\Phi},\quad j\in\{1,2,3\} \\
&b^X =  \tilde{\psi}(\tfrac{\xi_1}{N_1})\tilde{\psi}(\tfrac{\xi_2}{N_2})\rho_4(\xi_1,\xi_2)A(\xi_1,\xi_2)U(\xi)\tfrac{\nabla_\xi\Phi \nabla_{\xi_2}\Phi}{|\nabla_{\xi_2}\Phi|^2}, \\
&\tilde b^X = \nabla_{\xi_2}\cdot b^X
\end{align*}
satisfy
\[
\begin{aligned}
&\op{b_j^T} \lesssim \begin{cases} \tfrac{1}{N_1^{3/2}} & \text{if }\, N_1\leq 1 \\ \tfrac{1}{N_1} & \text{if }\, N_1>1 \end{cases}\quad j\in\{1,2,3\}, \\
&\op{b^X} \lesssim \begin{cases} \bigl(\tfrac{\langle N_1\rangle}{N_1}\bigr)^{\frac12+} \tfrac{N_2}{\langle N_2\rangle} &\text{always,} \\
\tfrac{N_2}{N_1} &\text{if }\, N_1\gtrsim 1 \text{ and }N_1\gg N_2,
\end{cases} \\
&\op{\tilde b^X} \lesssim  \begin{cases} \bigl(\tfrac{\langle N_1\rangle}{N_1}\bigr)^{\frac32} \tfrac{1}{\langle N_2\rangle} &\text{always,}\\
\tfrac{1}{N_1}  & \text{if }\,N_1\gtrsim 1 \text{ and }N_1\gg N_2.
\end{cases}
\end{aligned}
\]
\end{proposition}

\subsection{Region 1. Time non-resonance} Define
\[
\rho_1(\xi_1,\xi_2) = \chi_1^T(\xi_1,\xi_2)=\varphi(\tfrac{8\xi}{N_1})\tilde\psi(\tfrac{\xi_2}{N_1})\psi(\tfrac{\xi_1}{N_1}).
\]
This is a Coifman--Meyer multiplier that restricts to the region
\begin{equation}\label{vvR1}
|\xi|\leq |\xi_1|\sim |\xi_2|.
\end{equation}
In particular,
\begin{equation}\label{vvchi1Tbd}
|\partial_{\xi_2}^\alpha \rho_1 | \lesssim \bigl(\tfrac{1}{N_1\vee N_2}\bigr)^{|\alpha|},\quad |\alpha|\leq 3.
\end{equation}

As in \eqref{bT}, we define the multiplier
\begin{equation}\label{vvb1T}
b_1^T(\xi_1,\xi_2) = \rho_1(\xi_1,\xi_2)\tilde\psi(\tfrac{\xi_1}{N_1})\tilde\psi(\tfrac{\xi_2}{N_2})U(\xi)A(\xi_1,\xi_2)\frac{\nabla_\xi \Phi}{\Phi}.
\end{equation}
\begin{lemma}[Region 1 bounds]\label{L:vvR1} The following bound holds:
\begin{equation}\label{vvR1-bound}
\op{b^T_1} \lesssim \tfrac{1}{N_1}.
\end{equation}
\end{lemma}

\begin{proof}
Using \eqref{vvR1} and Lemma~\ref{H-derivatives}, one can check
\begin{align*}
|\Phi|&\gtrsim N_1\langle N_1\rangle, \\
\bigl| \partial_{\xi_2}^\alpha[ \nabla_\xi\Phi+\nabla_{\xi_2}\Phi]\bigr| & \lesssim \langle N_1\rangle\bigl(\tfrac{1}{N_1}\bigr)^{|\alpha|},\quad |\alpha|\leq 2.
\end{align*}
Thus,
\[
\bigl| \partial_{\xi_2}^\alpha \tfrac{\nabla_\xi\Phi}{\Phi}\bigr| \lesssim \tfrac{1}{N_1}\bigl(\tfrac{1}{N_1}\bigr)^{|\alpha|},\quad |\alpha|\leq 2.
\]
Using this together with the cutoff bounds (\eqref{cutoff-loss}, \eqref{vvchi1Tbd}) and \eqref{Ajbds}, we deduce
\[
\| b_1^T \|_{L_\xi^\infty \dot H_{\xi_2}^1}^{\frac12}\|b_1^T\|_{L_\xi^\infty \dot H_{\xi_2}^2}^{\frac12} \lesssim \tfrac{1}{N_1},
\]
which implies \eqref{vvR1-bound}. \end{proof}

\begin{remark} On the complement of Region 1 (i.e. on the support of $1-\rho_1$), we have
\begin{equation}\label{vvR1c}
|\xi|\sim|\xi_1|.
\end{equation}
\end{remark}

\subsection{Region 2. Time non-resonance}  Let $\phi_k:\mathbb{S}^2\to\R$ be a partition of unity adapted to a maximal $10^{-6}$-separated set $\{\omega_k\}$ on $\mathbb{S}^2$. Define
\[
\mathcal{R}_2 := \{(k,\ell):\angle(\omega_k,\omega_\ell) \geq \tfrac{2\pi}{3} + 4\cdot 10^{-6}\},
\]
and let
\[
\chi_2^T(\xi_1,\xi_2) = \sum_{(k,\ell)\in\mathcal{R}_2} \phi_k(\tunit{\xi})\phi_\ell(\tunit{\xi_1}).
\]
Note that
\begin{equation}\label{vvchi2Tbd}
\bigl| \partial_{\xi_2}^\alpha \chi_2^T \bigr| \lesssim \bigl(\tfrac{1}{N_1}\bigr)^{|\alpha|},\quad |\alpha|\leq 3.
\end{equation}

We define
\[
\rho_2(\xi_1,\xi_2)=[1-\chi_1^T(\xi_1,\xi_2)]\chi_2^T(\xi_1,\xi_2)
\]
and
\begin{equation}\label{vvb2T}
b_2^T(\xi_1,\xi_2) = \rho_2(\xi_1,\xi_2)\tilde\psi(\tfrac{\xi_1}{N_1}) \tilde\psi(\tfrac{\xi_2}{N_2})A(\xi_1,\xi_2)U(\xi)\frac{\nabla_\xi\Phi}{\Phi}.
\end{equation}
\begin{lemma}[Region 2 bounds]\label{L:vvR2} The following bound holds:
\begin{equation}\label{vvR2-bound}
\op{b_2^T} \lesssim \tfrac{1}{N_1}.
\end{equation}
\end{lemma}
\begin{proof} First note that on the support of $\chi_2^T$ we have $\theta_{01}\geq \tfrac{2\pi}{3}$.  Thus,
\[
|\xi|^2+|\xi_1|^2 - |\xi_2|^2 = 2|\xi_1| |\xi| \cos(\theta_{01}) < 0.
\]
This implies $|\xi_2|\geq |\xi|$, which in turn implies
\[
|\Phi| \gtrsim N_1\langle N_1\rangle.
\]

Using this together with Lemma~\ref{H-derivatives} (and recalling $N_2\leq N_1$), one also finds
\[
|\partial_{\xi_2}^\alpha \nabla_\xi\Phi| \lesssim \langle N_1\rangle\bigl(\tfrac{1}{N_1})^{|\alpha|}\qtq{for} |\alpha|\leq 2,\quad
|\partial_{\xi_2}^\alpha \Phi| \lesssim
 \begin{cases} \langle N_1\rangle & |\alpha| = 1, \\
\tfrac{\langle N_2\rangle}{N_2} & |\alpha|=2.
\end{cases}
\]

It follows that
\[
\bigl| \partial_{\xi_2}^\alpha \tfrac{\nabla_\xi\Phi}{\Phi}\bigr| \lesssim \begin{cases}
\tfrac{1}{N_1}\cdot\bigl(\tfrac{1}{N_1}\bigr)^{|\alpha|} & |\alpha| \leq 1, \\
\tfrac{1}{N_1}\cdot \tfrac{1}{N_1 N_2} & |\alpha|=2.\end{cases}
\]

Using this together with the cutoff bounds (\eqref{cutoff-loss}, \eqref{vvchi1Tbd}, \eqref{vvchi2Tbd}) and \eqref{Ajbds}, we deduce
\[
\opt{b_2^T}\lesssim \tfrac{1}{N_1},
\]
which gives \eqref{vvR2-bound}.\end{proof}

\begin{remark} On the complement of Region 2 (i.e. on the support of $1-\chi_2^T$), we have
\begin{equation}\label{12vs1}
\theta_{01}\leq \tfrac{2\pi}{3} + 8\cdot 10^{-6},\qtq{whence}
\sin(\theta_{01})\sim \sin(\tfrac12\theta_{01}).
\end{equation}
\end{remark}

\subsection{Region 3. Time non-resonance}  Let $\phi_k:\mathbb{S}^2\to\R$ be a partition of unity adapted to a maximal $10^{-6}\frac{N_1}{\langle N_1\rangle}$-separated set $\{\omega_k\}$ on $\mathbb{S}^2$.  Define
\begin{align*}
\mathcal{R}_3 & := \{(k,\ell): \angle( \omega_k,  \omega_\ell) \leq 10^{-4}\tfrac{N_1}{\langle N_1\rangle} - 4\cdot 10^{-6}\tfrac{N_1}{\langle N_1\rangle}\} \\
& \quad\quad \cup \{(k,\ell):\angle(\omega_k,\omega_\ell)\geq \pi - 10^{-4}\tfrac{N_1}{\langle N_1\rangle} + 4\cdot 10^{-6}\tfrac{N_1}{\langle N_1\rangle}\}.
\end{align*}
We let
\[
\chi_3^T(\xi_1,\xi_2) = \sum_{(k,\ell)\in\mathcal{R}_3} \phi_k(\tunit{\xi})\phi_\ell(\tunit{\xi_2}).
\]

One can check directly that
\begin{equation}\label{vvchi3Tbd}
\bigl|\partial_{\xi_2}^\alpha \chi_3^T\bigr| \lesssim \bigl(\tfrac{\langle N_1\rangle}{N_1N_2}\bigr)^{|\alpha|},\quad |\alpha|\leq 3.
\end{equation}

We define
\[
\rho_3(\xi_1,\xi_2)=\prod_{j=1}^2 \bigl[1-\chi_j^T(\xi_1,\xi_2)\bigr]\chi_3^T(\xi_1,\xi_2)
\]
and let
\[
b_3^T(\xi_1,\xi_2) = \rho_3(\xi_1,\xi_2)A(\xi_1,\xi_2)U(\xi)\tilde\psi(\tfrac{\xi_1}{N_1})\tilde\psi(\tfrac{\xi_2}{N_2})\frac{\nabla_\xi\Phi}{\Phi}.
\]
\begin{lemma}[Region 3 bounds]\label{L:vvR3} The following bound holds:
\begin{equation}\label{vvR3-bound}
\op{b_3^T} \lesssim \begin{cases} \tfrac{1}{N_1^{3/2}} & N_1 \leq 1, \\ \tfrac{1}{N_1} & N_1>1. \end{cases}
\end{equation}
\end{lemma}

\begin{proof} We begin by collecting some facts related to the support of $b_3^T$.

First note that by the definition of $\chi_3^T$, we have
\begin{equation}\label{chi3T-angle}
\theta_{02}\leq 10^{-4}\tfrac{N_1}{\langle N_1\rangle}\qtq{or}\theta_{02}\geq \pi-10^{-4}\tfrac{N_1}{\langle N_1\rangle},
\end{equation}
so that
\[
\sin(\theta_{02})\leq 10^{-4} \tfrac{N_1}{\langle N_1\rangle}.
\]
Thus, using the law of sines, \eqref{12vs1}, and \eqref{vvR1c}, we deduce
\begin{equation}\label{vvR3angles}
\begin{aligned}
& \sin(\theta_{12})= \tfrac{|\xi|}{|\xi_1|}\sin(\theta_{02}) \lesssim 10^{-4}\tfrac{N_1}{\langle N_1\rangle},  \\
& \sin(\tfrac12\theta_{01})\sim \tfrac{N_2}{N_1}\sin(\theta_{02}) \lesssim 10^{-4}\tfrac{N_2}{\langle N_1\rangle}\ll 1.
\end{aligned}
\end{equation}
In particular, $\theta_{01}\lesssim 10^{-4}\tfrac{N_2}{\langle N_1\rangle}\ll 1$, so that
\begin{equation}\label{02vs12}
\theta_{02}\sim\theta_{12}.
\end{equation}

Next, note that for fixed $\xi$, the multiplier $\chi_3^T(\xi,\xi_2) \tilde\psi(\frac{\xi_1}{N_1})\tilde\psi(\frac{\xi_2}{N_2})$ restricts $\xi_2$ to a set of volume
\begin{equation}\label{vvR3-volume}
N_2^3\tfrac{N_1^2}{\langle N_1\rangle^2}.
\end{equation}

We now claim the lower bound
\begin{equation}\label{vvR3-philb}
|\Phi| \gtrsim \tfrac{N_1^2 N_2}{\langle N_1\rangle}.
\end{equation}
To justify this, we treat the two alternatives in \eqref{chi3T-angle} separately, beginning with the first.  Writing
\[
\Phi = \bigl[h(|\xi_1|+|\xi_2|) - h(|\xi_1|) - h(|\xi_2|)\bigr] + \bigl[h(|\xi|) - h(|\xi_1|+|\xi_2|)\bigr],
\]
a direct computation gives
\[
h(|\xi_1|+|\xi_2|) - h(|\xi_1|) - h(|\xi_2|) = \tfrac{|\xi_1|\,|\xi_2|(|\xi_2|+2|\xi_1|)}{\langle \xi_1\rangle+ \langle |\xi_1|+|\xi_2|\rangle}
	+ \tfrac{|\xi_1|\,|\xi_2|(|\xi_1|+2|\xi_2|)}{\langle \xi_2\rangle + \langle |\xi_1| + |\xi_2|\rangle}\gtrsim\tfrac{N_1^2 N_2}{\langle N_1\rangle}.
\]
On the other hand, recalling \eqref{vvR1c}, we have
\[
| h(|\xi|) - h(|\xi_1|+|\xi_2|) | = \biggl| \int_{|\xi_1|+|\xi_2|}^{|\xi|} h'(r)\,dr \biggr| \lesssim \langle N_1\rangle \bigl| |\xi|-|\xi_1|-|\xi_2|\bigr|.
\]
As direct computation also gives
\[
\bigl| |\xi|-|\xi_1|-|\xi_2| \bigr| = 4\tfrac{|\xi_1|\,|\xi_2|}{|\xi|+|\xi_1|+|\xi_2|}\sin^2(\tfrac12\theta_{12}),
\]
we can use \eqref{vvR3angles} and continue from above to estimate
\[
|h(|\xi|) - h(|\xi_1|+|\xi_2|)| \lesssim  \langle N_1\rangle N_2 \sin^2(\tfrac12\theta_{12}) \ll \tfrac{N_1^2 N_2}{\langle N_1\rangle}.
\]
Thus \eqref{vvR3-philb} holds when $\theta_{02}\leq 10^{-4} \frac{N_1}{\langle N_1\rangle}$.

We now turn to the justification of \eqref{vvR3-philb} under the second alternative in \eqref{chi3T-angle}.  In fact, we will show something stronger, namely,
\begin{equation}\label{vvR3-philb2}
|\Phi| \gtrsim N_2 \qtq{when} \theta_{02}\geq \pi - 10^{-4} N_1/\langle N_1\rangle.
\end{equation}
Indeed, under this condition, $\cos(\theta_{02})<0$.  Thus
\[
|\xi|^2+|\xi_2|^2 - |\xi_1|^2 = 2|\xi|\,|\xi_2|\cos(\theta_{02})<0.
\]
This implies $|\xi_1|\geq |\xi|$ and therefore $\Phi\leq - H(\xi_2)$ and so \eqref{vvR3-philb2} follows.

We next claim the upper bound
\begin{equation}\label{vvR3-phixiub}
|\nabla_\xi\Phi| \lesssim N_2.
\end{equation}
Indeed, this follows from Lemma~\ref{DH-diff} and \eqref{vvR3angles}:
\[
|\nabla_\xi\Phi|  = |\nabla H(\xi)-\nabla H(\xi_1)| \lesssim N_2\tfrac{N_1}{\langle N_1\rangle} + \langle N_1\rangle \sin(\tfrac12\theta_{01}) \lesssim N_2.
\]

To proceed, we break into three cases.

\textbf{Case 1.} We first suppose
\begin{equation}
N_1 > 1. \label{vvR3-case1}
\end{equation}
For which the bounds established in \eqref{vvR3-philb} and \eqref{vvR3-phixiub} will be sufficient.  Indeed from these and Lemma~\ref{H-derivatives}, we have
\begin{align*}
|\Phi|& \gtrsim N_1 N_2, \\
|\partial_{\xi_2}^\alpha \nabla_\xi\Phi| &\lesssim
\begin{cases} N_2 & |\alpha| = 0, \\ \langle N_1\rangle \bigl(\tfrac{1}{N_1}\bigr)^{|\alpha|}& |\alpha|\in\{1,2\},\end{cases} \\
|\partial_{\xi_2}^\alpha \Phi| & \lesssim \begin{cases} N_1 & |\alpha|=1, \\ \tfrac{\langle N_2\rangle}{N_2} & |\alpha|=2. \end{cases}
\end{align*}

Thus,
\[
\bigl| \partial_{\xi_2}^\alpha \tfrac{\nabla_\xi\Phi}{\Phi}\bigr| \lesssim \tfrac{1}{N_1}\bigl(\tfrac{1}{N_2}\bigr)^{|\alpha|},\quad |\alpha|\leq 2.
\]
Combining this with the cutoff bounds (\eqref{cutoff-loss}, \eqref{vvchi1Tbd}, \eqref{vvchi2Tbd}, \eqref{vvchi3Tbd}), \eqref{Ajbds}, and the volume bound \eqref{vvR3-volume}, we deduce
\[
\opt{b_3^T} \lesssim \tfrac{1}{N_1},
\]
which is acceptable.

\textbf{Case 2.}  Next, we suppose
\begin{equation}
N_1\leq 1 \qtq{and} \theta_{02} \leq 10^{-4} \tfrac{N_1}{\langle N_1\rangle}. \label{vvR3-case2}
\end{equation} In this case, we will again use the bounds in $\eqref{vvR3-philb}$ and \eqref{vvR3-phixiub}; however, we will need the following additional estimate:
\begin{equation}\label{vvR3-phixi2ub}
|\nabla_{\xi_2}\Phi| \lesssim N_1.
\end{equation}
To prove this, we first note that by \eqref{02vs12} and \eqref{vvR3-case2}, we have
\[
\theta_{12} \lesssim 10^{-4} N_1\ll 1,\qtq{so that}\sin(\tfrac12\theta_{12}) \lesssim N_1.
\]
Thus, we may use Lemma~\ref{DH-diff}, \eqref{vvN2N1}, \eqref{vvR1c}, and \eqref{vvR3-case2} to bound
\[
|\nabla_{\xi_2}\Phi| = |\nabla H(\xi_1)-\nabla H(\xi_2)| \lesssim |\xi|\tfrac{N_1}{\langle N_1\rangle}+ \langle N_2\rangle N_1 \lesssim N_1^2 + N_1 \lesssim N_1.
\]

Noting that Lemma~\ref{H-derivatives} also gives
\[
|\partial_{\xi_2}^\alpha \nabla_\xi\Phi| \lesssim \bigl(\tfrac{1}{N_1}\bigr)^{|\alpha|}\qtq{for} |\alpha|\leq 2\quad \qtq{and} \quad|\partial_{\xi_2}^{\alpha} \Phi| \lesssim \tfrac{1}{N_2}\qtq{for}|\alpha|=2,
\]
we use \eqref{vvR3-philb}, \eqref{vvR3-phixiub}, and \eqref{vvR3-phixi2ub} to deduce
\[
\bigl| \partial_{\xi_2}^\alpha \tfrac{\nabla_\xi\Phi}{\Phi}\bigr| \lesssim \tfrac{1}{N_1^2}\bigl(\tfrac{1}{N_1N_2}\bigr)^{|\alpha|},\quad |\alpha|\leq 2.
\]
Combining this with the cutoff bounds (\eqref{cutoff-loss}, \eqref{vvchi1Tbd}, \eqref{vvchi2Tbd}, \eqref{vvchi3Tbd}), \eqref{Ajbds} and the volume bound \eqref{vvR3-volume}, and recalling that $|U(\xi)| \lesssim N_1$ in this regime, we deduce
\[
\opt{b_3^T} \lesssim \tfrac{1}{N_1^{3/2}},
\]
which is acceptable.

\textbf{Case 3.}  Finally, we suppose
\begin{equation}
N_1\leq 1 \qtq{and} \theta_{02} \geq \pi - 10^{-4}\tfrac{N_1}{\langle N_1\rangle}.\label{vvR3-case3}
\end{equation}

Noting that Lemma~\ref{H-derivatives} implies
\[
|\partial_{\xi_2}^\alpha\nabla_\xi\Phi| \lesssim (\tfrac{1}{N_1})^{|\alpha|},\quad |\partial_{\xi_2}^\alpha\Phi| \lesssim (\tfrac{1}{N_2})^{|\alpha|-1},\quad |\alpha|\in\{1,2\},
\]
we use \eqref{vvR3-philb2} (which holds under the assumption \eqref{vvR3-case3}) and \eqref{vvR3-phixiub} to deduce:
\[
\bigl| \partial_{\xi_2}^\alpha \tfrac{\nabla_\xi\Phi}{\Phi}\bigr| \lesssim \bigl(\tfrac{1}{N_1N_2}\bigr)^{|\alpha|},\quad |\alpha|\leq 2.
\]
Combining this with the cutoff bounds \eqref{cutoff-loss} (\eqref{vvchi1Tbd}, \eqref{vvchi2Tbd}, \eqref{vvchi3Tbd}), \eqref{Ajbds}, and the volume bound \eqref{vvR3-volume}, and recalling that $|U(\xi)| \lesssim N_1$ in this regime, we deduce
\[
\opt{b_3^T} \lesssim N_1^{1/2},
\]
which is acceptable.

This completes the proof of Lemma~\ref{L:vvR3}. \end{proof}

\begin{remark} On the complement of Region 3 (i.e. on the support of $1-\chi_3^T$), we have
\begin{equation}\label{vvR3c}
10^{-4} \tfrac{N_1}{\langle N_1\rangle} - 8\cdot 10^{-6}\tfrac{N_1}{\langle N_1\rangle} \leq \theta_{02} \leq \pi - 10^{-4}\tfrac{N_1}{\langle N_1\rangle} + 8\cdot 10^{-6} \tfrac{N_1}{\langle N_1\rangle}.
\end{equation}
Thus
\begin{equation}\label{vvR3c02}
\sin(\theta_{02})\gtrsim \tfrac{N_1}{\langle N_1\rangle}.
\end{equation}
By the law of sines and \eqref{vvR1c}, the same is true of $\sin(\theta_{12})$.  In particular,
\begin{equation}\label{vvR3c2}
\sin(\tfrac12\theta_{12})\gtrsim \tfrac{N_1}{\langle N_1\rangle}.
\end{equation}
\end{remark}

\subsection{Region 4. Space non-resonance} We finally define
\[
\rho_4(\xi_1,\xi_2)=\prod_{j=1}^3 [1-\chi_j^T(\xi_1,\xi_2)]
\]
and let
\begin{align*}
&b^X(\xi_1,\xi_2)=\rho_4(\xi_1,\xi_2)\tilde\psi(\tfrac{\xi_1}{N_1})\tilde\psi(\tfrac{\xi_2}{N_2})A(\xi_1,\xi_2) U(\xi) \tfrac{\nabla_\xi\Phi \nabla_{\xi_2}\Phi}{|\nabla_{\xi_2}\Phi|^2},\\
&\tilde b^X(\xi_1,\xi_2) = \nabla_{\xi_2} \cdot b^X(\xi_1,\xi_2).
\end{align*}

We first claim that
\begin{equation}\label{vvR4xivsxi2}
\tfrac{|\nabla_\xi\Phi|}{|\nabla_{\xi_2}\Phi|} \lesssim \begin{cases} \tfrac{\langle N_1 \rangle N_2}{N_1\langle N_2\rangle} &\text{always,} \\
\tfrac{N_2}{N_1}& \text{if }\, N_1\gtrsim 1 \text{ and } N_1\gg N_2.
\end{cases}
\end{equation}

We begin by using Lemma~\ref{DH-diff}, \eqref{12vs1}, and the law of sines to estimate
\begin{align}
|\nabla_\xi\Phi| & \lesssim |\xi_2|\tfrac{|\xi_1|}{\langle \xi_1\rangle} + \langle \xi_1\rangle \sin(\tfrac12\theta_{01})\nonumber \\
& \lesssim \tfrac{N_1 N_2}{\langle N_1\rangle} + \tfrac{\langle N_1\rangle N_2}{N_1}\sin(\theta_{12}) \nonumber \\
& \lesssim \begin{cases} \tfrac{N_1 N_2}{\langle N_1\rangle} + \tfrac{\langle N_1\rangle N_2}{N_1}\sin(\tfrac12\theta_{12}) &\text{always,} \\ N_2 &  \text{if }\, N_1\gtrsim 1 \text{ and } N_1\gg N_2.\end{cases}\label{vvR4phixiub}
\end{align}

Next, using \eqref{vvR3c2},
\begin{align}
|\nabla_{\xi_2}\Phi| & \gtrsim \bigl|\tunit{\xi_1} - \tunit{\xi_2}\bigr|\min\{h'(|\xi_1|),h'(|\xi_2|)\}  \nonumber\\
&\gtrsim \langle N_2\rangle \sin(\tfrac12\theta_{12}) \gtrsim \tfrac{N_1\langle N_2\rangle}{\langle N_1\rangle}.\label{vvR4phixi2lb}
\end{align}
Using \eqref{vvR4phixi2lb} together with the first estimate in \eqref{vvR4phixiub}, we deduce the first bound in \eqref{vvR4xivsxi2}.  On the other hand,
\begin{align}
|\nabla_{\xi_2}\Phi|& \geq |\nabla H(\xi_1)| - |\nabla H(\xi_2)| \gtrsim N_1 \quad \text{if }\, N_1\gtrsim 1 \text{ and } N_1\gg N_2.  \label{vvR4-case2-lb}
\end{align}
This bound, together with the second estimate in \eqref{vvR4phixiub}, gives the second bound in \eqref{vvR4xivsxi2}.

Note that using \eqref{vvR4phixi2lb} together with \eqref{vvR1c} and the law of sines, we also get
\begin{equation}\label{vvR4phixi2lb2}
|\nabla_{\xi_2}\Phi| \gtrsim \langle N_2\rangle \sin(\theta_{02}).
\end{equation}

We are now in a position to prove bounds for the multipliers $b^X$ and $\tilde b^X$.
\begin{lemma}[Region 4 bounds]\label{L:vvR4} The following bounds hold:
\begin{align}\label{vvR4-bd1}
&\op{b^X} \lesssim \begin{cases} \bigl(\tfrac{\langle N_1\rangle}{N_1}\bigr)^{\frac12+} \tfrac{N_2}{\langle N_2\rangle} &\text{always,} \\
\tfrac{N_2}{N_1} & \text{if }\, N_1\gtrsim 1 \text{ and } N_1\gg N_2,
\end{cases} \\
\label{vvR4-bd2}
&\op{\tilde b^X}  \lesssim  \begin{cases} \bigl(\tfrac{\langle N_1\rangle}{N_1}\bigr)^{\frac32} \tfrac{1}{\langle N_2\rangle} &\text{always,} \\
\tfrac{1}{N_1}  &  \text{if }\, N_1\gtrsim 1 \text{ and } N_1\gg N_2.
\end{cases}
\end{align}
\end{lemma}

\begin{proof}To begin, recall that $\tilde b^X = \nabla_{\xi_2}\cdot b^X.$  For the terms in which the divergence misses the factor $\tfrac{\nabla_\xi\Phi\nabla_{\xi_2}\Phi}{|\nabla_{\xi_2}\Phi|^2}$, we claim that we get an upper bound of $\tfrac{1}{N_2}\op{b^X}$, which (using \eqref{vvR4-bd1}) one can check is acceptable.  Indeed, using the cutoff bounds \eqref{cutoff-loss}, \eqref{vvchi1Tbd}, \eqref{vvchi2Tbd}, together with \eqref{Ajbds}, we see that if the divergence hits the product
\[
\tilde\psi(\tfrac{\xi_1}{N_1})\tilde\psi(\tfrac{\xi_2}{N_2})(1-\chi_1^T)(1-\chi_2^T)A(\xi_1,\xi_2),
\]
then we will have an additional factor of $\tfrac{1}{N_2}$, but otherwise we can argue exactly the same as for $b^X$.  If the divergence hits $(1-\chi_3^T)$, then we get an additional factor of $\tfrac{\langle N_1\rangle}{N_1N_2}$ (cf. \eqref{vvchi3Tbd}).  However, in this case, we can also use the better volume bound \eqref{vvR3-volume}.  Thus, once again we face an additional factor of $\tfrac{1}{N_2}$, which is acceptable.

Thus, to estimate $\tilde b^X$, it suffices to treat $b^X$ and the multiplier
\[
\tilde b^X_* = \prod_{j=1}^3(1-\chi_j^T)\cdot \tilde{\psi}(\tfrac{\xi_1}{N_1})\tilde{\psi}(\tfrac{\xi_2}{N_2}) A(\xi_1, \xi_2)U(\xi)\nabla_{\xi_2}\cdot\tfrac{\nabla_\xi\Phi\nabla_{\xi_2}\Phi}{|\nabla_{\xi_2}\Phi|^2},
\]

We will first prove the general bounds for $b^X$ and $\tilde b^X_*$, and then give the improvements when $N_1\gtrsim 1$ and $N_1\gg N_2$.

Using Lemma~\ref{H-derivatives}, we find
\[
|\partial_{\xi_2}^\alpha \nabla_\xi\Phi| \lesssim \langle N_1\rangle \bigl(\tfrac{1}{N_1}\bigr)^{|\alpha|}, \quad |\partial_{\xi_2}^\alpha\nabla_{\xi_2}\Phi | \lesssim \langle N_2\rangle \bigl(\tfrac{1}{N_2}\bigr)^{|\alpha|},\quad 1\leq |\alpha|\leq 3.
\]
Using \eqref{vvR4xivsxi2} and \eqref{vvR4phixi2lb2} as well (along with \eqref{vvR3c02}), we find:
\begin{equation}\label{vvR4-multub1}
\bigl| \partial_{\xi_2}^\alpha \tfrac{\nabla_\xi\Phi\nabla_{\xi_2}\Phi}{|\nabla_{\xi_2}\Phi|^2}\bigr| \lesssim \tfrac{\langle N_1\rangle N_2}{N_1\langle N_2\rangle}\bigl(\tfrac{1}{N_2\sin(\theta_{02})}\bigr)^{|\alpha|}\lesssim\tfrac{\langle N_1\rangle N_2}{N_1\langle N_2\rangle}\bigl(\tfrac{\langle N_1\rangle}{N_1 N_2}\bigr)^{|\alpha|} ,\quad |\alpha|\leq 3.
\end{equation}

We now claim that
\begin{equation}\label{vvR4-opt1}
\opt{b^X} \lesssim \bigl(\tfrac{\langle N_1\rangle}{N_1}\bigr)^{\frac12+} \tfrac{N_2}{\langle N_2\rangle},
\end{equation}
which implies the first bound in \eqref{vvR4-bd1}.

\begin{proof}[Proof of \eqref{vvR4-opt1}] If no derivatives hit the product
\begin{equation}\label{vvR4-bad-product}
(1-\chi_3^T)\tfrac{\nabla_\xi\Phi\nabla_{\xi_2}\Phi}{|\nabla_{\xi_2}\Phi|^2},
\end{equation}
then we use the cutoff bounds (\eqref{cutoff-loss}, \eqref{vvchi1Tbd}, \eqref{vvchi2Tbd}), together with \eqref{vvR4-multub1} (with $|\alpha|=0$),  \eqref{Ajbds}, and the upper bound $|U(\xi)|\lesssim\tfrac{N_1}{\langle N_1\rangle}$ to estimate the contribution to the $L_\xi^\infty \dot H_{\xi_2}^s$-norms by
\[
\tfrac{N_2}{\langle N_2\rangle}N_2^{\frac32-s},
\]
for $s\in\{1,2\}$, which is acceptable.

If any derivative hits $1-\chi_3^T$, then we are in a position to use the volume bound \eqref{vvR3-volume}.  In particular, in the case that all derivatives land on $1-\chi_3^T$, we may use \eqref{vvchi3Tbd}, \eqref{vvR3-volume}, and \eqref{vvR4-multub1} (with $|\alpha|=0$) to estimate the contribution to the $L_\xi^\infty \dot H_{\xi_2}^s$-norms by
\[
\bigl(\tfrac{\langle N_1\rangle}{N_1}\bigr)^{s-1} \tfrac{N_2}{\langle N_2\rangle} N_2^{\frac32-s}
\]
for $s\in\{1,2\}$, which is acceptable.  By using the second bound in \eqref{vvR4-multub1} (with $|\alpha|=1$) and the volume bound \eqref{vvR3-volume}, we get the same estimate for $s=2$ when one derivative lands on each factor in \eqref{vvR4-bad-product}, which is again an acceptable contribution.

It remains to consider the case when all derivatives land on $\frac{\nabla_\xi\Phi\nabla_{\xi_2}\Phi}{|\nabla_{\xi_2}\Phi|^2}$ for $s\in\{1,2\}$.  For this, we use the first bound in \eqref{vvR4-multub1}, and for fixed $\xi$ we use spherical coordinates in the $\xi_2$ variable (with $\xi$ as the north pole) to compute the $L_{\xi_2}^2$-norm.  In particular, recalling \eqref{vvR3c}, we may estimate the contribution to the $L_\xi^\infty \dot H_{\xi_2}^s$-norms by
\begin{equation}\label{vvR4-L2compute}
\tfrac{N_2}{\langle N_2\rangle}N_2^{\frac32-s}\biggl(\int_{\frac{N_1}{\langle N_1\rangle}}^{\pi - \frac{N_1}{\langle N_1\rangle}}\frac{d\varphi}{(\sin{\varphi})^{2s-1}}\biggr)^{1/2} \lesssim
\begin{cases}
\tfrac{N_2^{\frac52-s}}{\langle N_2\rangle} \bigl(\tfrac{\langle N_1\rangle}{N_1}\bigr)^{0+} & s=1, \\
\tfrac{N_2^{\frac52-s}}{\langle N_2\rangle} \bigl(\tfrac{\langle N_1\rangle}{N_1}\bigr)^{s-1} & s\in\{2,3\}
\end{cases}
\end{equation}
which is acceptable.  This completes the proof of \eqref{vvR4-opt1}. (Note that we do not need the case $s=3$ here; however, we will use it below.)   \end{proof}

We next claim
\begin{equation}\label{vvR4-opt2}
\opt{b^X_*} \lesssim \bigl(\tfrac{\langle N_1\rangle}{N_1}\bigr)^{\frac32} \tfrac{1}{\langle N_2\rangle},
\end{equation}
which implies the first bound in \eqref{vvR4-bd2}.
\begin{proof}[Proof of \eqref{vvR4-opt2}] We argue similarly to the case of \eqref{vvR4-opt1}.  If no derivatives hit
\begin{equation}\label{vvR4-bad-product2}
(1-\chi_3^T)\nabla_{\xi_2}\cdot\tfrac{\nabla_\xi\Phi\nabla_{\xi_2}\Phi}{|\nabla_{\xi_2}\Phi|^2},
\end{equation}
then we use \eqref{cutoff-loss}, \eqref{vvchi1Tbd}, \eqref{vvchi2Tbd},  \eqref{Ajbds}, the second bound in \eqref{vvR4-multub1} (with $|\alpha|=1$), and the upper bound $|U(\xi)|\lesssim\tfrac{N_1}{\langle N_1\rangle}$ to estimate the contribution to the $L_\xi^\infty \dot H_{\xi_2}^s$-norms by
\[
\tfrac{\langle N_1\rangle}{N_1} N_2^{\frac32-s} \tfrac{1}{\langle N_2\rangle}.
\]
for $s\in\{1,2\}$, which is acceptable.

If any derivative hits $1-\chi_3^T$, then we are in a position to use the volume bound \eqref{vvR3-volume}.  In particular, in the case that all derivatives land on $1-\chi_3^T$, we may use \eqref{vvchi3Tbd}, \eqref{vvR3-volume}, and the second bound in \eqref{vvR4-multub1} (with $|\alpha|=1$) to estimate the contribution to the $L_\xi^\infty \dot H_{\xi_2}^s$-norms by
\[
\bigl(\tfrac{\langle N_1\rangle}{N_1}\bigr)^{s}  N_2^{\frac32-s}\tfrac{1}{\langle N_2\rangle}
\]
for $s\in\{1,2\}$, which is also acceptable.  By using the second bound in \eqref{vvR4-multub1} (with $|\alpha|=2$) and the volume bound \eqref{vvR3-volume}, we get the same estimate for $s=2$ when one derivative lands on each factor in \eqref{vvR4-bad-product2}, which is again an acceptable contribution.

It remains to consider the case when all derivatives land on $\nabla_{\xi_2}\cdot\frac{\nabla_\xi\Phi\nabla_{\xi_2}\Phi}{|\nabla_{\xi_2}\Phi|^2}$ for $s\in\{1,2\}$.  For this, we estimate as before; that is we use the first bound in \eqref{vvR4-multub1}, and for fixed $\xi$ we use spherical coordinates in the $\xi_2$ variable (with $\xi$ as the north pole) to compute the $L_{\xi_2}^2$-norm. By  \eqref{vvR4-L2compute}, we estimate the contribution to the $L_\xi^\infty \dot H_{\xi_2}^s$-norms by
\[
\tfrac{N_2^{\frac52-(s+1)}}{\langle N_2\rangle}\bigl(\tfrac{\langle N_1\rangle}{N_1}\bigr)^s\sim \bigl(\tfrac{\langle N_1\rangle}{N_1}\bigr)^{s}  N_2^{\frac32-s}\tfrac{1}{\langle N_2\rangle}
\]
for $s\in\{1,2\}$, which is acceptable.\end{proof}

So far, we have established the first estimates in \eqref{vvR4-bd1} and \eqref{vvR4-bd2}.  To complete the proof of Lemma~\ref{L:vvR4}, we need to consider the case
\begin{equation}\label{vvR4-bd1-case2}
N_1\gtrsim 1 \qtq{and} N_1\gg N_2.
\end{equation}

Using Lemma~\ref{H-derivatives} and \eqref{vvR4-bd1-case2}, we find
\[
|\partial_{\xi_2}^\alpha \nabla_\xi\Phi| \lesssim (\tfrac{1}{N_1})^{|\alpha|-1}, \quad |\partial_{\xi_2}^\alpha \nabla_{\xi_2}\Phi| \lesssim \langle N_2\rangle (\tfrac{1}{N_2})^{|\alpha|},\quad 1\leq |\alpha|\leq 3.
\]
Thus, recalling \eqref{vvR4xivsxi2} and \eqref{vvR4-case2-lb}, we find:
\[
\bigl|\partial_{\xi_2}^\alpha \tfrac{\nabla_\xi\Phi\nabla_{\xi_2}\Phi}{|\nabla_{\xi_2}\Phi|^2}\bigr|\lesssim \tfrac{N_2}{N_1}\bigl(\tfrac{1}{N_2}\bigr)^{|\alpha|},\quad |\alpha|\leq 3.
\]
Using this together with \eqref{cutoff-loss}, \eqref{vvchi1Tbd}, \eqref{vvchi2Tbd}, \eqref{vvchi3Tbd}, and  \eqref{Ajbds}, it is not hard to verify that
\[
\opt{b^X} \lesssim \tfrac{N_2}{N_1},
\]
which gives the second bound in \eqref{vvR4-bd1}.  Similarly,
\[
\opt{\tilde b^X_*} \lesssim \tfrac{1}{N_1},
\]
giving the second bound in \eqref{vvR4-bd2}.  This completes the proof of Lemma~\ref{L:vvR4}. \end{proof}

\begin{proof}[Proof of Proposition~\ref{prop:dec2}]  The proof of Proposition~\ref{prop:dec2} follows immediately from Lemmas~\ref{L:vvR1}, \ref{L:vvR2}, \ref{L:vvR3}, and \ref{L:vvR4}.  One may also check that the cutoffs $\rho_j$ sum to $1$.  Note that we always choose the worst multiplier bound in the statement of Proposition~\ref{prop:dec2}.  For example, the bounds for $\op{b_1^T}$ and $\op{b_2^T}$ when $N_1\leq 1$ are better than those stated; however, the stated bound is attained by $b_3^T$.
\end{proof}


\section{Non-resonant decompositions III: $|v|^2$ terms}\label{dec3}
In this section we carry out the decomposition into non-resonant regions for the $|v|^2$ terms in \eqref{bs-quad}.  By symmetry, it suffices to consider the phase
\[
\Phi = H(\xi)+H(\xi_2) - H(\xi_1).
\]
As before, we will work at fixed frequencies $|\xi_j|\sim N_j$.  To simplify notation, we will suppress the $N_1, N_2$ subscripts on our multipliers.

The result of this section is the following.
\begin{proposition}[Non-resonant decomposition for $|v|^2$ terms]\label{prop:dec3}  Given $N_1$ and $N_2$, there exists a decomposition $1 = \sum_{j=1}^8 \rho_j$ on the support of $\tilde\psi(\tfrac{\xi_1}{N_1})\tilde\psi(\tfrac{\xi_2}{N_2})$ such that the following holds:
\begin{itemize}
\item For $j\in\{1,2,3,6\}$, define
\[
b_j^T(\xi_1,\xi_2) = \rho_j(\xi_1,\xi_2)\tilde\psi(\tfrac{\xi_1}{N_1})\tilde\psi(\tfrac{\xi_2}{N_2})A(\xi_1,\xi_2)U(\xi)\tfrac{\nabla_\xi\Phi}{\Phi}.
\]
Then
\[
\begin{array}{llll}
\op{b_1^T} \lesssim \tfrac{1}{N_2}, & \text{with} & N_1\lesssim N_2 & \text{in }\supp b_1^T, \\ \\
\op{b_j^T} \lesssim 1, & \text{with} & N_1\sim N_2 & \text{in }\supp b_j^T,\quad j\in\{2,3\}, \\ \\
\op{b_6^T} \lesssim \tfrac{1}{N_1 N_2^{1/2}}, & \text{with} & N_2\lesssim N_1 \lesssim 1 & \text{in }\supp b_6^T.
\end{array}
\]
\item For $j\in\{5,7,8\}$, define
\begin{align*}
& b_j^X(\xi_1,\xi_2)=\rho_j(\xi_1,\xi_2)\tilde\psi(\tfrac{\xi_1}{N_1})\tilde\psi(\tfrac{\xi_2}{N_2})A(\xi_1,\xi_2)U(\xi)\tfrac{\nabla_\xi\Phi\nabla_{\xi_2}\Phi}{|\nabla_{\xi_2}\Phi|^2}, \\
& \tilde b_j^X(\xi_1,\xi_2)=\nabla_{\xi_2}\cdot b_j^X.
\end{align*}
Then
\[
\begin{array}{llll}
\op{b_5^X} \lesssim \tfrac{N_2}{N_1}, & \text{with} & N_1\gtrsim 1\text{ and } N_1\gg N_2 & \text{in }\supp b_5^X, \\ \\
\op{b_j^X} \lesssim \tfrac{1}{N_1^{1/2}}, & \text{with} & N_2\lesssim N_1 \lesssim 1 & \text{in }\supp b_j^X,\quad j\in\{7,8\}
\end{array}
\]
and
\[
\op{\tilde b_j^X} \lesssim \tfrac{1}{N_2}\op{b_j^X},\quad j\in\{5,7,8\},
\]
with $\tilde b_j^X$ having the same support properties as $ b_j^X$.
\item Finally, define
\begin{align*}
& m(\xi_1,\xi_2)=\rho_4(\xi_1,\xi_2)\tilde\psi(\tfrac{\xi_1}{N_1})\tilde\psi(\tfrac{\xi_2}{N_2})A(\xi_1,\xi_2)U(\xi)\tfrac{\nabla_\xi\Phi}{\xi^{\perp_2}\cdot \nabla H(\xi_1)}, \\
& b^\angle(\xi_1,\xi_2)= \nabla_{\xi_2}\cdot[m(\xi_1,\xi_2)\xi^{\perp_2}], \\
&  b_k^\angle(\xi_1,\xi_2) = m(\xi_1,\xi_2) \tfrac{\xi_k\times\xi^{\perp_k}}{|\xi_k|^2}, \quad k\in\{1,2\}, \\
& \tilde b^\angle(\xi_1,\xi_2) = m(\xi_1,\xi_2)[\xi\cdot\tunit{\xi_1}\tunit{\xi_1}-\xi\cdot\tunit{\xi_2}\tunit{\xi_2}].
\end{align*}
Then
\[
\op{b^\angle}+\op{b_k^\angle}+\op{\tilde b^\angle} \lesssim 1, \quad k\in\{1,2\},
\]
with $1\lesssim N_1 \sim N_2$ in the support of these multipliers.
\end{itemize}
\end{proposition}

To begin, we record a lemma that allows us to exploit some cancellation in derivatives of $\Phi$.
\begin{lemma}\label{D2H-diff} Let $\Phi = H(\xi)+H(\xi_2)-H(\xi_1)$. For $1\leq |\alpha|\leq 4$,
\[
\bigl|\partial_{\xi_2}^\alpha \Phi \bigr| \lesssim \bigl( \tfrac{\langle \xi_1\rangle}{|\xi_1|^{|\alpha|}} + \tfrac{\langle \xi_2\rangle}{|\xi_2|^{|\alpha|}}\bigr)|\xi|
\]
\end{lemma}

\begin{proof} Using $H(\xi)=H(-\xi)$ and $\xi_1=\xi-\xi_2$, we deduce
\[
\partial_{\xi_2}^\alpha \Phi = [\partial_{\xi}^\alpha H](\xi_2) - [\partial_\xi^\alpha H](-\xi_1).
\]
Thus, the result is a consequence of the fundamental theorem of calculus and Lemma~\ref{H-derivatives}.\end{proof}

\subsection{Region 1. Time non-resonance}  Define
\[
\rho_1(\xi_1,\xi_2)=\chi_1^T(\xi_1,\xi_2) \equiv 1 \qtq{if} N_1\leq \tfrac1{64} N_2
\]
and vanishing otherwise.  We then define
\[
b_1^T(\xi_1,\xi_2)= \rho_1(\xi_1,\xi_2)\tilde\psi(\tfrac{\xi_1}{N_1})\tilde\psi(\tfrac{\xi_2}{N_2})A(\xi_1,\xi_2)U(\xi)\tfrac{\nabla_\xi\Phi}{\Phi}.
\]
Where this is non-zero, we have
\begin{equation}\label{R1}
|\xi_1|\leq |\xi|\sim N_2.
\end{equation}

\begin{lemma}[Region 1 bounds]\label{L:R1} The following bound holds:
\[
\op{b_1^T} \lesssim \tfrac{1}{ N_2}.
\]
Moreover, $N_1\lesssim N_2$ in the support of $b_1^T$.
\end{lemma}

\begin{proof} Using \eqref{R1} and Lemma~\ref{H-derivatives}, we deduce for $j\in\{1,2\}$:
\begin{align*}
&|\Phi|\gtrsim N_2\langle N_2\rangle,\quad |\nabla_\xi\Phi|\lesssim\langle N_2\rangle,\quad |\nabla_{\xi_j}\Phi|\lesssim \langle N_2\rangle, \\
& \bigl|\partial_{\xi_j}^\alpha[\nabla_\xi\Phi + \nabla_{\xi_j}\Phi]\bigr| \lesssim \langle N_1\rangle (\tfrac{1}{N_1}\bigr)^{|\alpha|},\quad |\alpha|\in\{1,2\}.
\end{align*}
Thus,
\[
\bigl|\partial_{\xi_j}^\alpha\tfrac{\nabla_\xi\Phi}{\Phi} \bigr| \lesssim \tfrac{1}{N_2}\bigl(\tfrac{1}{N_1}\bigr)^{|\alpha|},\quad |\alpha|\leq 2.
\]
Recalling also \eqref{cutoff-loss}, we deduce
\[
\opo{b_1^T} \lesssim \tfrac{1}{ N_2},
\]
which gives the lemma.
\end{proof}

\begin{remark} In the remaining cutoffs, a factor $1-\chi^T_1$ will enforce the constraint
\begin{equation}\label{R1c}
N_2 \leq 32 N_1 \qtq{and so also} |\xi_2|\lesssim |\xi_1|.
\end{equation}
\end{remark}

\subsection{Region 2. Time non-resonance} Let $\phi_k:\mathbb{S}^2\to\R$ be a partition of unity adapted to a maximal $10^{-6}$-separated set $\{\omega_k\}$ on $\mathbb{S}^2$.  Define
\[
\mathcal{R}_2:= \{(k,\ell):\angle(\omega_k,\omega_\ell)\geq \tfrac{2\pi}{3} + 4\cdot 10^{-6}\},
\]
and let
\[
\chi_2^T(\xi_1,\xi_2) = \sum_{(k,\ell)\in\mathcal{R}_2} \phi_k(\tunit{\xi})\phi_{\ell}(\tunit{\xi_1}).
\]
One can check that
\begin{equation}\label{chi2Tbd}
|\partial_{\xi_j}^\alpha \chi_2^T | \lesssim \bigl(\tfrac{1}{N_1}\bigr)^{|\alpha|},\quad|\alpha|\leq 3.
\end{equation}
We define
\[
\rho_2(\xi_1,\xi_2) =  [1-\chi_1^T(\xi_1,\xi_2)]\chi_2^T(\xi_1,\xi_2)
\]
and
\[
b_2^T(\xi_1,\xi_2) =\rho_2(\xi_1,\xi_2)\tilde\psi(\tfrac{\xi_1}{N_1})\tilde\psi(\tfrac{\xi_2}{N_2})A(\xi_1,\xi_2)U(\xi)\frac{\nabla_\xi\Phi}{\Phi}.
\]

\begin{lemma}[Region 2 bounds]\label{L:R2} The following bound holds:
\[
\op{b_2^T} \lesssim 1.
\]
Moreover, $N_1\sim N_2$ in the support of $b_2^T$.
\end{lemma}

\begin{proof}
On the support of $\chi_2^T$ we have $\theta_{01}\geq\tfrac{2\pi}{3}$.  Thus $\cos\theta_{01}\leq-\frac12$ and so
\[
|\xi_2|^2 - |\xi_1|^2 = |\xi|^2 -2|\xi_1|\,|\xi|\cos(\theta_{01}) \geq |\xi_1|\,|\xi|.
\]
In particular, using \eqref{R1c}, we deduce that
\begin{equation}\label{R2sim}
|\xi_1|\sim|\xi_2|,
\end{equation}
which justifies the last assertion in the lemma, as well as
\begin{equation}\label{R2-1}
|\xi_2| - |\xi_1| \geq \tfrac{|\xi|\,|\xi_1|}{|\xi_2|+|\xi_1|} \gtrsim |\xi|.
\end{equation}
These relations in turn yield
\begin{equation}\label{R2-philb}
|\Phi|\geq H(\xi_2)-H(\xi_1) = \int_{|\xi_1|}^{|\xi_2|}h'(r)\,dr \gtrsim \langle N_1\rangle\bigl| |\xi_2|-|\xi_1|\bigr| \gtrsim \langle N_1\rangle|\xi|.
\end{equation}

We note that by Lemma~\ref{D2H-diff}, we have
\begin{equation}\label{R2-ub}
|\partial_{\xi_2}^\alpha \Phi\bigr| \lesssim |\xi|\langle N_1\rangle\bigl(\tfrac{1}{N_1}\bigr)^{|\alpha|},\quad |\alpha|\in\{1,2\},
\end{equation}
while by Lemma~\ref{H-derivatives},
\[
|\partial_{\xi_2}^\alpha \nabla_\xi\Phi| \lesssim \langle N_1\rangle\bigl(\tfrac{1}{N_1}\bigr)^{|\alpha|},\quad |\alpha|\leq 2.
\]
Using this together with \eqref{R2-philb}, we find:
\[
|\partial_{\xi_j}^{\alpha}\tfrac{\nabla_\xi\Phi}{\Phi}\bigr| \lesssim \tfrac{1}{|\xi|}\bigl(\tfrac{1}{N_1}\bigr)^{|\alpha|},\quad |\alpha|\leq 2.
\]
Recalling the cutoff bounds \eqref{cutoff-loss} and \eqref{chi2Tbd} (and relying on the factor $U(\xi)$), we deduce
\[
\opt{b_2^T} \lesssim 1,
\]
which completes the proof of Lemma~\ref{L:R2}. \end{proof}

\begin{remark} On the complement of Region 2, we have
\begin{equation}\label{R2c}
\theta_{01}\leq \tfrac{2\pi}{3}+8\cdot 10^{-6},\qtq{whence} \sin(\theta_{01})\sim\sin(\tfrac12\theta_{01}).
\end{equation}
\end{remark}

\subsection{Region 3. Time non-resonance}\label{SS:9.3} Let $\phi_k$ and $\omega_k$ be as in Region 2.  Define
\[
\mathcal{R}_3=\{(k,\ell):\angle(\omega_k,\omega_\ell)\leq \tfrac{\pi}{3}+4\cdot 10^{-6}\} \cup \{(k,\ell):\angle(\omega_k,\omega_\ell)\geq \tfrac{2\pi}{3}-4\cdot 10^{-6}\}
\]
and let
\begin{equation}\label{R3-cutoff}
\chi_3^T(\xi_1,\xi_2) = \sum_{(k,\ell)\in\mathcal{R}_3} \phi_k(\tunit{\xi})\phi_\ell(\tunit{\xi_2}).
\end{equation}
One verifies that
\begin{equation}\label{chi3Tbd}
|\partial_{\xi_2}^\alpha \chi_3^T| \lesssim \bigl(\tfrac{1}{|\xi_2|}\bigr)^{|\alpha|}.
\end{equation}

We define
\begin{equation}\label{R3-more-cutoffs}
 \rho_3(\xi_1,\xi_2) =\prod_{j=1}^2 \bigl[1-\chi_j^T(\xi_1,\xi_2)\bigr]\cdot \chi_3^T(\xi_1,\xi_2)
    \qtq{if} C \leq N_1 \leq 64 N_2
\end{equation}
and vanishing otherwise (see also \eqref{R1c}).  Note that $C>0$ will be chosen below.

We define
\[
b_3^T(\xi_1,\xi_2) = \rho_3(\xi_1,\xi_2)\tilde\psi(\tfrac{\xi_1}{N_1})\tilde\psi(\tfrac{\xi_2}{N_2})A(\xi_1,\xi_2)U(\xi)\frac{\nabla_\xi\Phi}{\Phi}.
\]
\begin{lemma}[Region 3 bounds]\label{L:R3}
The following bound holds,
\[
\op{b_3^T} \lesssim 1.
\]
Moreover, $1\lesssim N_1\sim N_2$ on the support of $b_3^T$.
\end{lemma}

\begin{proof} From \eqref{R1c} and definition of $\rho_3$, we have that
\begin{equation}\label{R3-condition}
|\xi_1|\sim |\xi_2|\gtrsim 1.
\end{equation}
in the support of $b_3^T$.  This justifies the last assertion in the lemma.

Turning to bounding $b_3^T$, we first claim that
\begin{equation}\label{R3-philb}
|\Phi|\gtrsim |\xi| N_2.
\end{equation}

To see this, we introduce the function
\[
G(\xi)=H(\xi)-(|\xi|^2+1) = \frac{-1}{H(\xi)+|\xi|^2+1} \in [-1,0).
\]
A direct computation gives
\[
\Phi = H(\xi)+H(\xi_2)-H(\xi_1) =2\xi_2\cdot\xi + G(\xi_2)-G(\xi_1)+ G(\xi)-G(0).
\]
Note that
\[
G(\eta)-G(\mu) = \frac{-[H(\mu)-H(\eta)+|\mu|^2-|\eta|^2]}{[H(\eta)+|\eta|^2+1][H(\mu)+|\mu|^2+1]},
\]
so that
\[
|G(\eta)-G(\mu)| \lesssim \tfrac{1+|\eta|+|\mu|}{\langle\eta\rangle^2\langle\mu\rangle^2}\cdot|\eta-\mu|\lesssim |\eta-\mu|.
\]
Noting also that $|\cos(\theta_{02})|\geq \tfrac12 - 10^{-3}$ in the support of $\chi_3^T$, we have
\[
|\Phi| \geq 2|\xi|\,|\xi_2|\,|\cos(\theta_{02})| - C|\xi|  \gtrsim |\xi|(|\xi_2|-C) \gtrsim|\xi|\cdot|\xi_2|,
\]
provided we choose the constant appropriately in \eqref{R3-more-cutoffs}.

We next claim
\begin{equation}\label{R3-phiub}
|\partial_{\xi_2}^\alpha \nabla_\xi\Phi| \lesssim \bigl(\tfrac{1}{N_2}\bigr)^{|\alpha|-1},\quad |\alpha|\leq 2.
\end{equation}

By Lemma~\ref{H-derivatives} and \eqref{R2sim},
\[
|\nabla_\xi\Phi| = |\nabla H(\xi)-\nabla H(\xi_1)| \lesssim \langle\xi\rangle+\langle\xi_1\rangle \lesssim  N_2.
\]
which gives the case $|\alpha| = 0$ of \eqref{R3-phiub}. The cases $|\alpha|\in\{1,2\}$ follow from Lemma~\ref{H-derivatives}, recalling that \eqref{R3-condition} holds.

Finally, using Lemma~\ref{D2H-diff} and \eqref{R3-condition}, we have
\begin{equation}\label{R3-phiub2}
|\partial_{\xi_2}^\alpha \Phi| \lesssim |\xi|(\tfrac{1}{N_1})^{|\alpha|-1},\quad |\alpha|\in\{1,2\}.
\end{equation}

Using \eqref{R3-philb}, \eqref{R3-phiub}, and \eqref{R3-phiub2}, we deduce:
\[
\bigl| \partial_{\xi_2}^\alpha \tfrac{\nabla_\xi\Phi}{\Phi}\bigr| \lesssim \tfrac{1}{|\xi|}\bigl(\tfrac{1}{N_2}\bigr)^{|\alpha|},\quad|\alpha|\leq 2.
\]
Using this together with the cutoff bounds (\eqref{cutoff-loss}, \eqref{chi2Tbd}, \eqref{chi3Tbd}), \eqref{Ajbds}, and relying on the factor $U(\xi)$, we deduce
\[
\opt{b_3^T} \lesssim 1,
\]
which completes the proof of the Lemma~\ref{L:R3}.
\end{proof}

\begin{remark} On the support of $1-\chi_3^T$, we have
\begin{equation}\label{R3c}
\theta_{02}\in[\tfrac{\pi}{3},\tfrac{2\pi}{3}],\qtq{so that} |\cos(\theta_{02})| \leq \tfrac12.
\end{equation}
\end{remark}

\subsection{Region 4. Angular non-resonance}\label{SS:9.4}  Define
\[
\rho_4(\xi_1,\xi_2) = \prod_{j=1}^3 \bigl[1-\chi_j^T(\xi_1,\xi_2)\bigr] \qtq{if} C\leq N_1 \leq 64 N_2
\]
and vanishing otherwise (see also \eqref{R1c}).  Note that
\begin{equation}\label{chi4bd}
\bigl|\partial_{\xi_j}^\alpha \rho_4\bigr| \lesssim \bigl(\tfrac{1}{N_2}\bigr)^{|\alpha|}.
\end{equation}
On the support of $\rho_4$, we will use the integration by parts described in Section~\ref{S:anr}.

Define
\[
m(\xi_1,\xi_2)=\rho_4(\xi_1,\xi_2)\tilde\psi(\tfrac{\xi_1}{N_1})\tilde\psi(\tfrac{\xi_2}{N_2})A(\xi_1,\xi_2)U(\xi)\tfrac{\nabla_\xi\Phi}{\xi^{\perp_2}\cdot \nabla H(\xi_1)},
\]
where we recall the notation $\xi^{\perp_j}=\xi-\xi\cdot\tunit{\xi_j}\tunit{\xi_j}$.

Recalling the discussion in Section~\ref{S:anr}, we need to prove multiplier bounds for
\begin{align*}
& b^\angle(\xi_1,\xi_2)= \nabla_{\xi_2}\cdot[m(\xi_1,\xi_2)\xi^{\perp_2}], \\
&  b_j^\angle(\xi_1,\xi_2) = m(\xi_1,\xi_2) \tfrac{\xi_j\times\xi}{|\xi_j|^2}, \\
& \tilde b^\angle(\xi_1,\xi_2) = m(\xi_1,\xi_2)[\xi\cdot\tunit{\xi_1}\tunit{\xi_1}-\xi\cdot\tunit{\xi_2}\tunit{\xi_2}].
\end{align*}

\begin{lemma}[Region 4 bounds]\label{L:R4} The following bounds hold:
\begin{equation}
\op{b^{\angle}}+\op{b_j^{\angle}}+\op{\tilde b^{\angle}} \lesssim 1. \label{L:R41}
\end{equation}
Moreover, $1\lesssim N_1\sim N_2$ in the support of the multipliers above.
\end{lemma}

\begin{proof} First note that in the support of $\rho_4$, we have
\begin{equation}\label{R4-condition}
|\xi_1|\sim |\xi_2|\gtrsim 1.
\end{equation}

Using \eqref{R3c} and \eqref{R4-condition} one readily checks
\begin{equation}\label{R4lb1}
|\xi^{\perp_2}\cdot \nabla H(\xi_1)| =  \bigl| \tfrac{|\xi|^2|\xi_2|^2-(\xi\cdot\xi_2)^2}{|\xi_2|^2}\cdot\tfrac{h'(|\xi_1|)}{|\xi_1|}\bigr| \gtrsim \tfrac{|\xi|^2|\xi_2|^2\langle \xi_1\rangle}{|\xi_1| |\xi_2|^2}\gtrsim |\xi|^2.
\end{equation}
For higher derivatives, one can check that
\[
|\partial_{\xi_2}^\alpha\bigl[ \xi^{\perp_2}\cdot \nabla H(\xi_1)\bigr] |\lesssim |\xi|^2\bigl(\tfrac{1}{N_1}\bigr)^{|\alpha|},\quad |\alpha|\leq 3.
\]

From Lemma~\ref{H-derivatives}, we also have
\[
|\partial_{\xi_2}^\alpha \nabla_\xi\Phi| \lesssim N_1\bigl(\tfrac{1}{N_1}\bigr)^{|\alpha|},\quad |\alpha|\leq 3.
\]

Using the last three bounds, one deduces
\begin{equation}\label{R4-symbol1}
\bigl| \partial_{\xi_2}^\alpha \bigl(\tfrac{\nabla_\xi\Phi}{\xi^{\perp_2}\cdot\nabla H(\xi_1)}\bigr)\bigr| \lesssim \tfrac{N_1}{|\xi|^2}\bigl(\tfrac{1}{N_1}\bigr)^{|\alpha|},\quad |\alpha|\leq 3.
\end{equation}

Next, note that
\[
|\partial_{\xi_2}^\alpha \xi^{\perp_2}\,| \lesssim |\xi|\bigl(\tfrac{1}{N_1}\bigr)^{|\alpha|},\quad |\alpha|\leq 3.
\]
Using this together with the cutoff bounds (\eqref{cutoff-loss}, \eqref{chi2Tbd}, \eqref{chi3Tbd}, \eqref{chi4bd}),  \eqref{Ajbds}, and \eqref{R4-symbol1} (and recalling the presence of $U(\xi)$), we deduce
\[
\opt{b^{\angle}} \lesssim 1,
\]
giving the first estimate in \eqref{L:R41}.

Next, we recall \eqref{R4-condition} and note that
\[
\bigl|\partial_{\xi_2}^{\alpha} \tfrac{\xi_j\times \xi}{|\xi_j|^2}\bigr| \lesssim \tfrac{|\xi|}{N_1}\bigl(\tfrac{1}{N_1}\bigr)^{|\alpha|},\quad|\alpha|\leq 2.
\]
Combining this with the bounds above (i.e. all of the cutoff bounds and the estimate \eqref{R4-symbol1}), one can deduce
\[
\opt{b_j^{\angle}} \lesssim 1,
\]
giving the second estimate in \eqref{L:R41}.

Finally, we compute directly
\[
\xi\cdot\tunit{\xi_1}\tunit{\xi_1}-\xi\cdot\tunit{\xi_2}\tunit{\xi_2} = \tfrac{(\xi\cdot\xi_1)[\xi\cdot(\xi_2-\xi_1)]\xi_1+|\xi_1|^2[|\xi|^2\xi_1-(\xi\cdot\xi_2)\xi] }{|\xi_1|^2|\xi_2|^2}.
\]
Thus one can check
\[
\bigl| \partial_{\xi_j}^\alpha\bigl(\xi\cdot\tunit{\xi_1}\tunit{\xi_1} - \xi\cdot\tunit{\xi_2}\tunit{\xi_2}\bigr)\bigr| \lesssim \tfrac{|\xi|^2}{N_1}\bigl(\tfrac{1}{N_1}\bigr)^{|\alpha|},\quad|\alpha|\leq 2,
\]
and we can again estimate as above to deduce
\[
\opt{\tilde b^{\angle}}\lesssim 1,
\]
giving the final estimate in \eqref{L:R41}.  Note that in contrast to the first two estimates in \eqref{L:R41}, we use the fact that $U(\xi)\leq 1$.
\end{proof}
\subsection{Region 5. Space non-resonance} Let
\[
\rho_5(\xi_1,\xi_2) = \prod_{j=1}^2 \bigl[1-\chi_j^T(\xi_1,\xi_2)\bigr] \qtq{if} N_1 \geq C \qtq{and} N_2 < \tfrac1{64} N_1
\]
and vanishing otherwise.  One can readily check that
\begin{equation}\label{chi5bd}
|\partial_{\xi_j}^\alpha \rho_5|\lesssim \bigl(\tfrac{1}{N_2}\bigr)^{|\alpha|},\quad |\alpha|\leq 3
\end{equation}
and also that on the support of $\rho_5$,
\begin{equation}\label{R5-condition}
16 |\xi_2| < |\xi_1| \sim |\xi|\qtq{and} |\xi_1| \gtrsim 1
\end{equation}

We define the multipliers
\begin{align*}
& b_5^X(\xi_1,\xi_2) = \rho_5(\xi_1,\xi_2)\tilde\psi(\tfrac{\xi_1}{N_1})\tilde\psi(\tfrac{\xi_2}{N_2})A(\xi_1,\xi_2)U(\xi)\tfrac{\nabla_{\xi}\Phi\nabla_{\xi_2}\Phi}{|\nabla_{\xi_2}\Phi|^2}, \\
& \tilde b_5^X(\xi_1,\xi_2)=\nabla_{\xi_2}\cdot b_5^X(\xi_1,\xi_2).
\end{align*}

\begin{lemma}[Region 5 bounds]\label{L:R5}  The following bounds hold:
\[
\op{b_5^X} \lesssim \tfrac{N_2}{N_1}, \quad \op{\tilde b_5^X} \lesssim \tfrac{1}{N_1},
\]
with $N_1\gtrsim 1$ and $ N_1\gg N_2$ on the support of these multipliers.
\end{lemma}

\begin{proof} By Lemma~\ref{H-derivatives}, \eqref{R5-condition}, and the fundamental theorem of calculus,
\begin{equation}\label{R5-ub}
|\nabla_{\xi}\Phi| = |\nabla H(\xi)-\nabla H(\xi_1)| \lesssim N_2.
\end{equation}
Similarly,
\begin{equation}\label{R5-lb}
|\nabla_{\xi_2}\Phi| = \bigl| h'(|\xi_1|)\tunit{\xi_1} + h'(|\xi_2|)\tunit{\xi_2} \bigr| \gtrsim \langle N_1\rangle - \langle N_2\rangle \gtrsim N_1.
\end{equation}

For higher derivatives, we use Lemma~\ref{H-derivatives} and \eqref{R5-condition} to estimate
\[
|\partial_{\xi_2}^\alpha \nabla_{\xi}\Phi| \lesssim N_1\bigl(\tfrac{1}{N_1}\bigr)^{|\alpha|},\quad
|\partial_{\xi_2}^\alpha \nabla_{\xi_2}\Phi| \lesssim \langle N_2\rangle\bigl(\tfrac{1}{N_2}\bigr)^{|\alpha|},\quad|\alpha|\leq 3.
\]
Using this together with \eqref{R5-ub} and \eqref{R5-lb}, we deduce
\begin{equation}\label{R5-symbol}
\bigl| \partial_{\xi_2}^\alpha \tfrac{\nabla_{\xi}\Phi\nabla_{\xi_2}\Phi}{|\nabla_{\xi_2}\Phi|^2} \bigr| \lesssim \tfrac{N_2}{N_1}\bigl(\tfrac{1}{N_2}\bigr)^{|\alpha|},\quad |\alpha|\leq 3.
\end{equation}
This, together with the cutoff bounds (\eqref{cutoff-loss}, \eqref{chi2Tbd}, \eqref{chi3Tbd}, \eqref{chi4bd}, and \eqref{chi5bd}) and \eqref{Ajbds}, implies
\[
\opt{b_5^X}\lesssim \tfrac{N_2}{N_1},
\]
giving the first bound.  Distributing the derivative in $\tilde b_5^X$ and using the same bounds above yields
\[
\opt{\tilde b_5^X} \lesssim \tfrac{1}{N_1},
\]
giving the second bound.  This completes the proof of Lemma~\ref{L:R5}. \end{proof}

\begin{remark}  Combining the cutoff functions defined thus far, we have
$$
\sum_{j=1}^5 \rho_j(\xi_1,\xi_2) = \begin{cases} \chi_2(\xi_1,\xi_2) &\text{if $\frac1{32}N_2\leq N_1 <C$} \\ 1 & \text{otherwise.} \end{cases}
$$
\end{remark}

\subsection{Region 6. Time non-resonance}
Let $\phi_k$ be a partition of unity adapted to a maximal $10^{-6}N_1$-separated set $\{\omega_k\}$ on $\mathbb{S}^2$. Define
\[
\mathcal{R}_6 = \{(k,\ell):\angle(\omega_k,\omega_\ell)\geq \tfrac\pi2 \vee (\pi - 2\cdot 10^{-4}N_1 + 4\cdot 10^{-6}N_1)\},
\]
and let
\[
\chi_6^T(\xi_1,\xi_2) = \sum_{(k,\ell)\in\mathcal{R}_6} \phi_k(\tunit{\xi})\phi_\ell(\tunit{\xi_2}).
\]
Note that
\[
\bigl| \partial_{\xi_2}^\alpha \chi_6^T \bigr| \lesssim \bigl(\tfrac{1}{N_1N_2}\bigr)^{|\alpha|},\quad |\alpha|\leq3.
\]

We define
\[
\rho_6(\xi_1,\xi_2) = \bigl[1-\chi_2^T(\xi_1,\xi_2)\bigr]\chi_6^T(\xi_1,\xi_2) \qtq{if} \tfrac1{32}N_2\leq N_1 <C
\]
and vanishing otherwise.  Note that
\begin{equation}\label{chi6Tbd}
\bigl| \partial_{\xi_2}^\alpha \rho_6\bigr| \lesssim \bigl(\tfrac{1}{N_1N_2}\bigr)^{|\alpha|},\quad |\alpha|\leq 3.
\end{equation}
We define the multiplier
\[
b_6^T(\xi_1,\xi_2) = \rho_6(\xi_1,\xi_2)\tilde\psi(\tfrac{\xi_1}{N_1})\tilde\psi(\tfrac{\xi_2}{N_2})A(\xi_1,\xi_2)U(\xi)\frac{\nabla_\xi\Phi}{\Phi}.
\]
\begin{lemma}[Region 6 bounds]\label{L:R6}
The following bound holds:
\[
\op{b_6^T} \lesssim \tfrac{1}{N_1 N_2^{1/2}}.
\]
Moreover, $N_2\lesssim N_1\lesssim 1$ on the support of $b_6^T$.
\end{lemma}

\begin{proof} First note that for fixed $\xi$, the multiplier $\chi_6^T(\xi,\xi_2) \tilde\psi(\frac{\xi_1}{N_1})\tilde\psi(\frac{\xi_2}{N_2})$ restricts $\xi_2$ to a set of volume
\begin{equation}\label{R6-volume}
N_2^3 N_1^2.
\end{equation}
By construction, $\theta_{02}\geq \pi - 2\cdot 10^{-4}N_1$ on the support of $ \chi_6^T$; thus
\begin{equation}\label{R6-angle}
\sin(\tfrac12\theta_{02}') \leq 10^{-4} N_1.
\end{equation}

We first claim the lower bound
\begin{equation}\label{R6-lb}
|\Phi| \gtrsim |\xi| N_1 N_2.
\end{equation}
We begin by writing
\[
\Phi = h(|\xi|)+h(|\xi_2|)-h(|\xi|+|\xi_2|)+ h(|\xi|+|\xi_2|)-h(|\xi_1|).
\]
By direct computation, one verifies
\[
\bigl| h(|\xi|)+h(|\xi_2|)-h(|\xi|+|\xi_2|) \bigr| = \bigl| \tfrac{|\xi|\,|\xi_2|(|\xi_2|+2|\xi|)}{\langle \xi\rangle + \langle|\xi|+|\xi_2|\rangle} + \tfrac{|\xi|\,|\xi_2|(|\xi|+2|\xi_2|)}{\langle \xi_2\rangle + \langle |\xi|+|\xi_2|\rangle}\bigr| \gtrsim |\xi| N_1 N_2.
\]

For the remaining piece, we use the fundamental theorem of calculus, Lemma~\ref{H-derivatives}, and \eqref{R6-angle} to get
\[
\bigl| h(|\xi|+|\xi_2|) - h(|\xi_1|)\bigr| \lesssim \bigl| |\xi|+|\xi_2| - |\xi_1| \bigr| \lesssim \tfrac{|\xi|\,|\xi_2|}{|\xi|+|\xi_2|+|\xi_1|}\sin^2(\tfrac12\theta_{02}') \ll |\xi| N_1 N_2.
\]
Thus \eqref{R6-lb} holds.

Next, we claim
\begin{equation}\label{R6-ub}
|\nabla_\xi\Phi| \lesssim N_2.
\end{equation}
Indeed, using Lemma~\ref{DH-diff}, we first get
\[
|\nabla_\xi\Phi| = |\nabla H(\xi)-\nabla H(\xi_1)| \lesssim N_1 N_2 + \sin(\tfrac12\theta_{01}).
\]
By \eqref{R2c}, the law of sines, and \eqref{R6-angle},
\[
\sin(\tfrac12\theta_{01}) \sim \tfrac{N_2}{N_1} \sin(\theta_{02}) \sim \tfrac{N_2}{N_1}\sin(\theta_{02}') \lesssim \tfrac{N_2}{N_1}\sin(\tfrac12\theta_{02}')\lesssim N_2.
\]
Thus (recalling $N_1\lesssim 1$), \eqref{R6-ub} holds.

For higher derivatives, Lemma~\ref{H-derivatives} gives
\begin{equation}\label{R6-ubb}
|\partial_{\xi_2}^\alpha \nabla_\xi\Phi|\lesssim (\tfrac{1}{N_1})^{|\alpha|},\quad |\alpha|\in\{1,2\}.
\end{equation}

Next, we will show
\begin{equation}\label{R6-ub2}
|\partial_{\xi_2}^\alpha \Phi | \lesssim
\begin{cases}
|\xi| & |\alpha| =1, \\
|\xi|\tfrac{1}{N_2^2} & |\alpha| = 2.
\end{cases}
\end{equation}
In fact, for $|\alpha|=2$ this follows directly from Lemma~\ref{D2H-diff}. For $|\alpha|=1$, we first have by Lemma~\ref{DH-diff},
\[
|\nabla_{\xi_2}\Phi| = |\nabla H(\xi_2)+\nabla H(\xi_1)| \lesssim |\xi| + \sin(\tfrac12\theta_{12}').
\]
As $\theta_{01}+\theta_{12}'=\theta_{02}'\leq \frac\pi2$, the law of sines and \eqref{R6-angle} yield
\[
\sin(\tfrac12\theta_{12}') \sim \sin(\theta_{12}') \sim \tfrac{|\xi|}{N_1}\sin(\theta_{02}') \lesssim |\xi|,
\]
as needed.

Using \eqref{R6-lb}, \eqref{R6-ub}, \eqref{R6-ubb}, and \eqref{R6-ub2}, we deduce
\[
\bigl|\partial_{\xi_2}^\alpha\tfrac{\nabla_\xi\Phi}{\Phi}\bigr| \lesssim \tfrac{1}{|\xi| N_1}\cdot\begin{cases} \bigl(\tfrac{1}{N_1 N_2}\bigr)^{|\alpha|}& |\alpha|\in\{0,1\}, \\
\tfrac{1}{N_1 N_2^3} & |\alpha|=2.\end{cases}
\]
Combining this with the cutoff bounds (\eqref{cutoff-loss}, \eqref{chi2Tbd}, \eqref{chi3Tbd}, \eqref{chi4bd}, \eqref{chi5bd}, and \eqref{chi6Tbd}), \eqref{Ajbds}, and the volume bound \eqref{R6-volume}, we deduce
\[
\opt{b_6^T} \lesssim \tfrac{1}{N_1 N_2^{1/2}},
\]
giving the desired bound. This completes the proof of Lemma~\ref{L:R6}. \end{proof}

\begin{remark} On the support of $1-\chi_6^T$, one has
\begin{equation}\label{R6c-angle}
\theta_{02}\leq \pi - 2\cdot 10^{-4}N_1+ 8\cdot 10^{-6}N_1,
\end{equation}
so that
\begin{equation}\label{R6c}
\sin(\tfrac12\theta_{02}') \gtrsim N_1.
\end{equation}
\end{remark}

Before proceeding to Regions 7 and 8, let us pause to derive a lower bound that we will be using in both: Assume first that  $|\xi_2|\leq |\xi_1|$.  Then
\[
\nabla_{\xi_2}\Phi = \bigl(\tunit{\xi_1}+\tunit{\xi_2})h'(|\xi_2|) + \tunit{\xi_1}\int_{|\xi_2|}^{|\xi_1|}h''(r)\,dr.
\]
As $\bigl(\tunit{\xi_1}+\tunit{\xi_2}\bigr)\cdot\tunit{\xi_1} = 1 + \cos(\theta_{12})\geq 0$, it follows that
\begin{equation}\label{R7-lb1}
|\nabla_{\xi_2}\Phi| \gtrsim \max\{ \sin(\tfrac12\theta_{12}'), N_2\bigl| |\xi_1|-|\xi_2| \bigr|\}.
\end{equation}
Repeating the preceding argument with roles reversed, shows that \eqref{R7-lb1} actually holds whenever $N_2\lesssim N_1\lesssim 1$.

\subsection{Region 7. Space non-resonance} Let $\phi_k$ be a partition of unity adapted to a maximal $10^{-6}N_1$-separated set $\{\omega_k\}\subset\mathbb{S}^2$. Define
\[
\mathcal{R}_7 = \{(k,\ell):\angle(\omega_k,\omega_\ell) \leq \tfrac\pi6\wedge(2\cdot 10^{-4}N_1 - 4\cdot 10^{-6}N_1)\}
\]
and set
\[
\chi_7^X(\xi_1,\xi_2)=\sum_{(\omega_k,\omega_\ell)\in\mathcal{R}_7} \phi_k(\tunit{\xi})\phi_\ell(\tunit{\xi_2}).
\]

We now define
\[
\rho_7(\xi_1,\xi_2) =\prod_{j\in\{2,6\}}\bigl[1-\chi_j^T(\xi_1,\xi_2)\bigr]\chi_7^X(\xi_1,\xi_2) \qtq{if} \tfrac1{32}N_2\leq N_1 <C
\]
and vanishing otherwise. Note that
\begin{equation}\label{chi7Xbd}
\bigl| \partial_{\xi_2}^\alpha \chi_7^X\bigr| \lesssim \bigl(\tfrac{1}{N_1N_2}\bigr)^{|\alpha|}.
\end{equation}
We also define
\begin{align*}
&b_7^X(\xi_1,\xi_2) = \rho_7(\xi_1,\xi_2)\tilde\psi(\tfrac{\xi_1}{N_1})\tilde\psi(\tfrac{\xi_2}{N_2})A(\xi_1,\xi_2)U(\xi) \tfrac{\nabla_\xi\Phi \nabla_{\xi_2}\Phi}{|\nabla_{\xi_2}\Phi|^2}, \\
&\tilde b_7^X(\xi_1,\xi_2)=\nabla_{\xi_2}\cdot b_7^X(\xi_1,\xi_2).
\end{align*}

\begin{lemma}[Region 7 bounds]\label{L:R7} The following bounds hold:
\[
\op{b_7^X} \lesssim\tfrac{N_2}{N_1^{1/2}}, \quad \op{\tilde b_7^X} \lesssim\tfrac{1}{N_1^{3/2}}.
\]
Moreover, $N_2\lesssim N_1\lesssim 1$ on the support of these multipliers.
\end{lemma}

\begin{proof}  Note that the cutoff $ \chi_7^X$ guarantees
\begin{equation}\label{R7-angle}
\theta_{02}\leq 2\cdot 10^{-4} N_1,\qtq{so that}\sin(\tfrac12\theta_{02})\lesssim N_1.
\end{equation}
Furthermore, for fixed $\xi$, this restricts $\xi_2$ to a set of volume
\begin{equation}\label{R7-volume}
N_2^3 N_1^2.
\end{equation}

We first establish
\begin{equation}\label{R7-ub1}
\bigl|\partial_{\xi_2}^\alpha \nabla_\xi\Phi\bigr| \lesssim
\begin{cases} N_2 & |\alpha| = 0, \\ \bigl(\tfrac{1}{N_1}\bigr)^{|\alpha|} & 1\leq |\alpha|\leq 3.\end{cases}
\end{equation}
Indeed, using Lemma~\ref{DH-diff} with \eqref{R2c}, the law of sines, and \eqref{R7-angle}, we find
\[
|\nabla_\xi\Phi| \lesssim N_1 N_2 + \sin(\tfrac12\theta_{01}) \lesssim N_1 N_2 + \tfrac{N_2}{N_1}\sin(\theta_{02}) \lesssim N_2,
\]
which gives the case $|\alpha|=0$.  The other cases follow from Lemma~\ref{H-derivatives}, recalling $N_1\lesssim 1$.

We next claim
\begin{equation}\label{R7-lb}
|\nabla_{\xi_2}\Phi|\gtrsim 1.
\end{equation}
Recall from \eqref{R2c} that $\theta_{01}\leq \tfrac{2\pi}{3}+8\cdot 10^{-6}$; furthermore, by the law of sines, and \eqref{R7-angle}, we have
\[
\sin(\theta_{01}) \sim \tfrac{N_2}{N_1}\sin(\theta_{02}) \ll N_2.
\]
We conclude $\theta_{01}\ll N_2$. Recalling \eqref{R7-angle} and the fact that $\theta_{01}+\theta_{02}+\theta_{12}' = \pi$, we deduce that
\[
\theta_{12}'\geq \pi-cN_1
\]
for some $0<c\ll 1$. Thus $\sin(\tfrac12\theta_{12}')\gtrsim 1$, and so \eqref{R7-lb1} gives \eqref{R7-lb}.

We next note that Lemma~\ref{H-derivatives} immediately gives
\begin{equation}\label{R7-ub2}
\bigl| \partial_{\xi_2}^{\alpha}\nabla_{\xi_2}\Phi\bigr| \lesssim \bigl(\tfrac{1}{N_2}\bigr)^{|\alpha|},\quad |\alpha|\leq 3.
\end{equation}

Using \eqref{R7-ub1}, \eqref{R7-lb}, and \eqref{R7-ub2}, one verifies
\begin{equation}\label{R7-symbol}
\bigl|\partial_{\xi_2}^{\alpha}\tfrac{\nabla_\xi\Phi \nabla_{\xi_2}\Phi}{|\nabla_{\xi_2}\Phi|^2} \bigr|\lesssim
\begin{cases} N_2 & |\alpha|=0, \\
N_2\cdot \tfrac{1}{N_1}\bigl(\tfrac{1}{N_2}\bigr)^{|\alpha|},& 1\leq |\alpha|\leq 3.
\end{cases}
\end{equation}

We are now in a position to estimate the multipliers.  Comparing \eqref{R7-symbol} and \eqref{Ajbds} with all of the cutoff bounds (\eqref{cutoff-loss}, \eqref{chi2Tbd}, \eqref{chi6Tbd}, and \eqref{chi7Xbd}), and exploiting the volume bound \eqref{R7-volume}, one finds that
\[
\opt{b_7^X} \lesssim \tfrac{N_2}{N_1^{1/2}}.
\]
This gives the first multiplier bound.

We turn to estimating $\tilde b_7^X$.  If the divergence lands on any term other than
\[
m(\xi_1,\xi_2):=[1-\chi_6^T]\chi_7^X\tfrac{\nabla_\xi\Phi\nabla_{\xi_2}\Phi}{|\nabla_{\xi_2}\Phi|^2},
\]
we will face an additional $\tfrac{1}{N_2}$, but otherwise we can estimate exactly as above.  Thus, we get an acceptable contribution for these terms (cf. $N_1\lesssim 1$), and we are left to estimate
\[
b^* := \tilde\psi(\tfrac{\xi_1}{N_1})\tilde\psi(\tfrac{\xi_2}{N_2})A(\xi_1,\xi_2)U(\xi)[1-\chi_1^T][1-\chi_2^T]\nabla_{\xi_2}\cdot m(\xi_1,\xi_2).
\]
Recalling the cutoff bounds \eqref{chi6Tbd} and \eqref{chi7Xbd}, we see that we will lose an additional $\tfrac{1}{N_1 N_2}$ when estimating $b^*$. In particular, using \eqref{R7-symbol}, the cutoff bounds, and the volume bound \eqref{R7-volume}, we deduce
\[
\opt{b^*} \lesssim N_1^{-3/2},
\]
giving the second multiplier bound. This completes the proof of Lemma~\ref{L:R7}.\end{proof}

\begin{remark} On the support of $(1-\chi_6^T)(1-\chi_7^X)$, we have
\begin{equation}\label{R7c}
\sin(\theta_{02}) \gtrsim N_1.
\end{equation}
\end{remark}

\subsection{Region 8. Space non-resonance} We define
\[
\rho_8(\xi_1,\xi_2) =\prod_{j\in\{2,6\}} \bigl[1-\chi_j^T\bigr]\bigl[1-\chi_7^X(\xi_1,\xi_2)\bigr] \qtq{if} \tfrac1{32}N_2\leq N_1 <C
\]
and vanishing otherwise.  Corresponding to this we define multipliers
\begin{align*}
&b_8^X(\xi_1,\xi_2) = \rho_8(\xi_1,\xi_2)\tilde\psi(\tfrac{\xi_1}{N_1})\tilde\psi(\tfrac{\xi_2}{N_2})A(\xi_1,\xi_2)U(\xi) \tfrac{\nabla_\xi\Phi \nabla_{\xi_2}\Phi}{|\nabla_{\xi_2}\Phi|^2}, \\
&\tilde b_8^X(\xi_1,\xi_2)=\nabla_{\xi_2}\cdot b_8^X(\xi_1,\xi_2).
\end{align*}

\begin{lemma}[Region 8 bounds]\label{L:R8} The following bounds hold:
\begin{align*}
\op{b_8^X} \lesssim \tfrac{1}{N_1^{1/2}}, \quad \op{\tilde b_8^X} \lesssim \tfrac{1}{N_1^{1/2} N_2},
\end{align*}
with $N_2\lesssim N_1\lesssim 1$ on the support of these multipliers.
\end{lemma}

\begin{proof} To begin, recall that $\tilde b_8^X = \nabla_{\xi_2}\cdot b_8^X$. For the terms in which the divergence misses the factor $\tfrac{\nabla_{\xi}\Phi\nabla_{\xi_2}\Phi}{|\nabla_{\xi_2}\Phi|^2}$, we claim that we get an upper bound of $\tfrac{1}{N_2}\op{b_8^X}$, which is an acceptable bound.  Indeed, using the cutoff bounds \eqref{cutoff-loss}, \eqref{chi2Tbd}, as well as \eqref{Ajbds}, we first see that if the divergence hits the product
\[
\tilde\psi(\tfrac{\xi_1}{N_1})\tilde\psi(\tfrac{\xi_2}{N_2})(1-\chi_1^T)(1-\chi_2^T) A(\xi_1,\xi_2),
\]
then we get an additional factor of $\tfrac{1}{N_2}$, but we will otherwise be able to estimate exactly as we do for $b_8^X$ below.  If the divergence hits the product
\[
(1-\tilde \chi_6^T)(1-\tilde \chi_7^X),
\]
then we will get an additional factor of $\tfrac{1}{N_1 N_2}$ (cf. \eqref{chi6Tbd} and \eqref{chi7Xbd}); however, in this case, we will always be able to use the improved volume bounds in \eqref{R6-volume} or \eqref{R7-volume}, which gains an additional $N_1$ in the arguments presented below.  In particular, we again face only an additional factor of $\tfrac{1}{N_2}$.

Thus, to treat $\tilde b_8^X$, it suffices to get suitable estimates for $b_8^X$ and the multiplier
\begin{equation}\label{b8star}
b_8^* :=\rho_8(\xi_1,\xi_2) \tilde\psi(\tfrac{\xi_1}{N_1})\tilde\psi(\tfrac{\xi_2}{N_2}) A(\xi_1,\xi_2) U(\xi)\nabla_{\xi_2}\cdot\bigl(\tfrac{\nabla_\xi\Phi\nabla_{\xi_2}\Phi}{|\nabla_{\xi_2}\Phi|^2}\bigr).
\end{equation}

We begin by applying Lemma~\ref{DH-diff} and using \eqref{R2c} and the law of sines to deduce
\begin{equation}\label{R8-ub1}
|\nabla_\xi\Phi| \lesssim N_1 N_2 + \sin(\tfrac12\theta_{01}) \lesssim N_1 N_2 + \sin(\theta_{01})\lesssim N_1 N_2 + \tfrac{N_2}{|\xi|}\sin(\tfrac12\theta_{12}').
\end{equation}

\textbf{Case 1.} Recalling that $N_2\lesssim N_1\lesssim 1$, we first consider the case
\begin{equation}\label{R8-case1}
N_2 \ll N_1 \sim |\xi|.
\end{equation}

Note that in this case, $\bigl| |\xi_1|-|\xi_2|\bigr| \gtrsim N_1$, so that \eqref{R8-ub1} and \eqref{R7-lb1} give
\begin{equation}\label{R8-case1-ub}
|\nabla_{\xi}\Phi| \lesssim (1+\tfrac{N_2}{|\xi|})|\nabla_{\xi_2}\Phi| \lesssim |\nabla_{\xi_2}\Phi|.
\end{equation}

Thus by \eqref{R7-lb1}, the law of sines, \eqref{R8-case1}, and \eqref{R7c}, we have
\[
|\nabla_{\xi_2}\Phi| \gtrsim \sin(\theta_{12}') \sim \tfrac{|\xi|}{N_1}\sin(\theta_{02}) \gtrsim \sin(\theta_{02})\gtrsim N_1.
\]
For higher derivatives, we have by Lemma~\ref{H-derivatives},
\[
\bigl| \partial_{\xi_2}^\alpha \nabla_{\xi}\Phi\bigr| \lesssim \bigl(\tfrac{1}{N_1}\bigr)^{|\alpha|},\quad
\bigl| \partial_{\xi_2}^\alpha \nabla_{\xi_2}\Phi\bigr| \lesssim \bigl(\tfrac{1}{N_2}\bigr)^{|\alpha|},\quad 1\leq |\alpha|\leq 3.
\]
Using these bounds, one verifies
\begin{equation}\label{R8-case1-symbol}
\bigl| \partial_{\xi_2}^\alpha \tfrac{\nabla_\xi\Phi\nabla_{\xi_2}\Phi}{|\nabla_{\xi_2}\Phi|^2} \bigr| \lesssim
\bigl(\tfrac{1}{N_2\sin(\theta_{02})}\bigr)^{|\alpha|} \lesssim \bigl(\tfrac{1}{N_1 N_2}\bigr)^{|\alpha|},\quad |\alpha|\leq 3.
\end{equation}
Combining the second estimate in \eqref{R8-case1-symbol} with the cutoff bounds (\eqref{cutoff-loss}, \eqref{chi2Tbd}, \eqref{chi6Tbd}, \eqref{chi7Xbd}), \eqref{Ajbds}, and noting that $U(\xi)\sim N_1$ in this regime, we deduce
\[
\opt{b_8^X} \lesssim \tfrac{1}{N_1^{1/2}},
\]
which is acceptable.  (The worst terms arise when all of the derivatives hit $\tfrac{\nabla_\xi\Phi\nabla_{\xi_2}\Phi}{|\nabla_{\xi_2}\Phi|^2}$.)

We turn to estimating $b_8^*$ in Case 1 (cf. \eqref{b8star}).  We wish to show
\[
\opt{b_8^*} \lesssim \tfrac{1}{N_1^{1/2} N_2},
\]
For this, it will suffice to prove
\begin{equation}\label{R8-ets}
\|b_8^*\|_{L_\xi^\infty \dot H_{\xi_2}^s} \lesssim N_1^{1-s} N_2^{\frac12-s},\quad s\in\{1,2\}.
\end{equation}

First, if no derivatives hit
\[
(1-\chi_6^T)(1-\chi_7^X)\nabla_{\xi_2}\cdot\bigl(\tfrac{\nabla_\xi\Phi\nabla_{\xi_2}\Phi}{|\nabla_{\xi_2}\Phi|^2}\bigr),
\]
then we use the cutoff bounds (\eqref{cutoff-loss}, and \eqref{chi2Tbd}), \eqref{Ajbds}, and the second estimate in \eqref{R8-case1-symbol} (with $|\alpha|=1$) to estimate the contribution to the $L_\xi^\infty \dot H_{\xi_2}^s$-norms by
\[
N_2^{\frac12-s}
\]
for $s\in\{1,2\}$, which is acceptable.

Next, if any derivative hits $(1-\chi_6^T)(1-\chi_7^X)$, then we may use the volume bound in \eqref{R6-volume} or \eqref{R7-volume}.  In particular, when all derivatives land on these cutoffs, we may use \eqref{chi6Tbd}, \eqref{chi7Xbd}, and the second estimate in \eqref{R8-case1-symbol} (with $|\alpha|=1$) to estimate the contribution to the $L_\xi^\infty \dot H_{\xi_2}^s$-norms by
\[
N_1^{1-s} N_2^{\frac12-s}
\]
for $s\in\{1,2\}$, which is acceptable.

It remains to consider the following two situations:
\begin{itemize}
\item[(i)] $s=2$, with one derivative hitting $\nabla_{\xi_2}\cdot\tfrac{\nabla_{\xi}\Phi\nabla_{\xi_2}\Phi}{|\nabla_{\xi_2}\Phi|^2}$ and the other hitting the cutoffs $(1-\chi_6^T)(1-\chi_7^X)$,
\item[(ii)] $s\in\{1,2\}$ and all derivatives land on $\nabla_{\xi_2}\cdot\tfrac{\nabla_\xi\Phi\nabla_{\xi_2}\Phi}{|\nabla_{\xi_2}\Phi|^2}$.
\end{itemize}

For these terms, for fixed $\xi$ we use spherical coordinates in $\xi_2$ (with $\xi$ as the north pole) to compute the $L_{\xi_2}^2$-norms.  We also recall \eqref{R7c}.

Using \eqref{chi6Tbd}, \eqref{chi7Xbd}, and the first bound in \eqref{R8-case1-symbol} (with $|\alpha|=2$), we estimate the contribution of terms in (i) to the $L_\xi^\infty \dot H_{\xi_2}^2$-norm by
\[
\tfrac{N_1\cdot N_2^{\frac32}}{N_1 N_2^3}\biggl(\int_{N_1}^{\pi-N_1}\frac{d\varphi}{(\sin\varphi)^3}\biggr)^{\frac12} \lesssim N_1^{-1} N_2^{-\frac32},
\]
which is acceptable in light of \eqref{R8-ets}.

Similarly, we estimate the contribution of the terms in (ii) to the $L_\xi^\infty \dot H_{\xi_2}^s$-norms by
\[
\tfrac{N_1 N_2^{3/2}}{N_2^{s+1}} \biggl(\int_{N_1}^{\pi-N_1} \frac{d\varphi}{(\sin\varphi)^{2s+1}}\biggr)^{1/2} \lesssim N_1^{1-s} N_2^{\frac12-s}\qtq{for} s\in\{1,2\},
\]
which is acceptable.  This completes the proof of \eqref{R8-ets} and hence of Lemma~\ref{L:R8} in Case 1.

\textbf{Case 2.} Suppose now that
\begin{equation}\label{R8-case2}
N_1 \sim N_2\lesssim 1.
\end{equation}

By \eqref{R7-lb1}, the law of sines, and \eqref{R7c}, we have
\begin{equation}\label{R8-lb3}
|\nabla_{\xi_2}\Phi|\gtrsim \sin(\tfrac12\theta_{12}')\gtrsim\sin(\theta_{12}')\sim\tfrac{|\xi|}{N_1}\sin(\theta_{02}) \gtrsim |\xi|.
\end{equation}
Combining this with the upper bound \eqref{R8-ub1}, we find
\begin{equation}\label{R8-case2-ub1}
|\nabla_\xi \Phi| \lesssim\tfrac{N_1}{|\xi|}|\nabla_{\xi_2}\Phi|.
\end{equation}

For higher derivatives of $\nabla_\xi\Phi$, we appeal to Lemma~\ref{H-derivatives} to find
\begin{equation}\label{R8-ub3}
\bigl|\partial_{\xi_2}^\alpha \nabla_\xi\Phi\bigr| \lesssim \bigl(\tfrac{1}{N_1}\bigr)^{|\alpha|},\quad 1\leq|\alpha|\leq 3.
\end{equation}
For higher derivatives of $\nabla_{\xi_2}\Phi$, Lemma~\ref{D2H-diff} and \eqref{R8-case2} give
\begin{equation}\label{R8-ub4}
\bigl| \partial_{\xi_2}^\alpha\Phi\bigr| \lesssim |\xi|\bigl(\tfrac{1}{N_1}\big)^{|\alpha|}, \quad 1\leq |\alpha|\leq 4.
\end{equation}

Combining \eqref{R8-case2}, \eqref{R8-case2-ub1}, \eqref{R8-lb3}, \eqref{R8-ub3}, \eqref{R8-ub4} and recalling the bounds \eqref{R7c}, we deduce
\begin{equation}\label{R8-symbol}
\bigl|\partial_{\xi_2}^{\alpha}\tfrac{\nabla_{\xi}\Phi\nabla_{\xi_2}\Phi}{|\nabla_{\xi_2}\Phi|^2}\bigr| \lesssim \tfrac{N_1}{|\xi|}\bigl(\tfrac{1}{N_1\sin\theta_{02}}\bigr)^{|\alpha|} \lesssim \tfrac{N_1}{|\xi|}\bigl(\tfrac{1}{N_1^2}\bigr)^{|\alpha|},\quad |\alpha|\leq 3.
\end{equation}

We are now in a position to prove the multiplier bounds, beginning with $b_8^X$.  We use the second estimate in \eqref{R8-symbol}, the cutoff bounds (\eqref{cutoff-loss}, \eqref{chi2Tbd}, \eqref{chi6Tbd}, \eqref{chi7Xbd}), \eqref{Ajbds}, the fact that $U(\xi)\sim|\xi|$ in this regime, and \eqref{R8-case2} to arrive at the estimate
\[
\opt{b_8^X} \lesssim \tfrac{1}{N_1^{1/2}},
\]
which is acceptable.

We turn to estimating $b_8^*$ in Case 2 (cf. \eqref{b8star}).  In this case, we need to show
\[
\opt{b_8^*} \lesssim \tfrac{1}{N_1^{3/2}},
\]
for which it will suffice to prove
\[
\|b_8^*\|_{L_\xi^\infty \dot H_{\xi_2}^s} \lesssim N_1^{\frac32-2s},\quad s\in\{1,2\}.
\]

 First, if no derivatives hit
\[
(1-\chi_6^T)(1-\chi_7^X)\nabla_{\xi_2}\cdot\bigl(\tfrac{\nabla_\xi\Phi\nabla_{\xi_2}\Phi}{|\nabla_{\xi_2}\Phi|^2}\bigr),
\]
then we use the cutoff bounds (\eqref{cutoff-loss}, and \eqref{chi2Tbd}), \eqref{Ajbds}, and the second estimate in \eqref{R8-symbol} (with $|\alpha|=1$) to estimate the contribution to the $L_\xi^\infty \dot H_{\xi_2}^s$-norms by
\[
N_1^{\frac12-s}
\]
for $s\in\{1,2\}$, which is acceptable.

If any derivative hits $(1-\chi_6^T)(1-\chi_7^X)$, then we may use the volume bound in \eqref{R6-volume} or \eqref{R7-volume}.  In particular, when all derivatives land on these cutoffs, we may use \eqref{chi6Tbd}, \eqref{chi7Xbd}, and the second estimate in \eqref{R8-symbol} (with $|\alpha|=1$) to estimate the contribution to the $L_\xi^\infty \dot H_{\xi_2}^s$-norms by
\[
N_1^{\frac32-2s}
\]
for $s\in\{1,2\}$, which is acceptable.

It remains to consider the following two situations:
\begin{itemize}
\item[(i)] $s=2$, with one derivative hitting $\nabla_{\xi_2}\cdot\tfrac{\nabla_{\xi}\Phi\nabla_{\xi_2}\Phi}{|\nabla_{\xi_2}\Phi|^2}$ and the other hitting the cutoffs $(1-\chi_6^T)(1-\chi_7^X)$,
\item[(ii)] $s\in\{1,2\}$ and all derivatives land on $\nabla_{\xi_2}\cdot\tfrac{\nabla_\xi\Phi\nabla_{\xi_2}\Phi}{|\nabla_{\xi_2}\Phi|^2}$.
\end{itemize}

For these terms, for fixed $\xi$ we again use spherical coordinates in $\xi_2$ (with $\xi$ as the north pole) to compute the $L_{\xi_2}^2$-norms.  Recall also \eqref{R7c}.

Using \eqref{chi6Tbd}, \eqref{chi7Xbd}, and the first bound in \eqref{R8-symbol} (with $|\alpha|=2$), we estimate the contribution of terms in (i) to the $L_\xi^\infty \dot H_{\xi_2}^2$-norm by
\[
\tfrac{N_1^{\frac52}}{N_1^4}\biggl(\int_{N_1}^{\pi-N_1}\frac{d\varphi}{(\sin\varphi)^3}\biggr)^{\frac12} \lesssim N_1^{-\frac52},
\]
which is acceptable.

Similarly, we estimate the contribution of the terms in (ii) to the $L_\xi^\infty \dot H_{\xi_2}^s$-norms by
\[
\tfrac{N_1^{\frac52}}{N_1^{s+1}} \biggl(\int_{N_1}^{\pi-N_1} \frac{d\varphi}{(\sin\varphi)^{2s+1}}\biggr)^{1/2} \lesssim N_1^{\frac32-2s}\qtq{for} s\in\{1,2\},
\]
which is acceptable. This completes the proof of Lemma~\ref{L:R8} in Case 2.
\end{proof}

\begin{proof}[Proof of Proposition~\ref{prop:dec3}]  Proposition~\ref{prop:dec3} follows from Lemmas~\ref{L:R1}, \ref{L:R2}, \ref{L:R3}, \ref{L:R4}, \ref{L:R5}, \ref{L:R6}, \ref{L:R7}, and \ref{L:R8}.  One may also check that the multipliers $\rho_j$ sum to $1$ on the support of $\tilde\psi(\tfrac{\xi_1}{N_1})\tilde\psi(\tfrac{\xi_2}{N_2})$.  Note that in the statement of Proposition~\ref{prop:dec3}, we always choose the worst bounds; for example, the bounds obtained for $b_8^X$ in Lemma~\ref{L:R8} are strictly worse than those for $b_7^X$ in Lemma~\ref{L:R7}.
\end{proof}

\section{Estimation of non-resonant terms}\label{S:est}

In this section, we complete the proof of Proposition~\ref{prop:bs-quad} by estimating the contribution of the time, space, and angular non-resonant regions of frequency space to \eqref{bs-quad}.

Recall that in Sections~\ref{S:tnr}--\ref{S:anr}, we collected the terms that we need to estimate in order to deal with each type of non-resonant region, namely \eqref{tnr} and \eqref{tnr2} for time non-resonant regions, \eqref{snr1} and \eqref{snr2} for space non-resonant regions, and \eqref{anr1}, \eqref{anr2}--\eqref{anr4} for angular non-resonant regions.  These involve certain bilinear operators, namely, $b^T$ for time non-resonant terms (cf. \eqref{bT}), $b^X$ and $\tilde b^X$ for space non-resonant terms (cf. \eqref{bX}, \eqref{tbX}), and $b^\angle$, $b^\angle_j$, $\tilde b^\angle$ for angular non-resonant terms (cf. \eqref{bA}, \eqref{bfj}, \eqref{tbA}).  In Sections~\ref{sec:dn}--\ref{dec3}, we decomposed frequency space into non-resonant regions for each type of quadratic nonlinearity and proved bounds for the resulting bilinear operators (cf. Propositions~\ref{prop:dec1}, \ref{prop:dec2}, and \ref{prop:dec3}).

In this section, we will show that the bounds established in Propositions~\ref{prop:dec1}, \ref{prop:dec2}, and \ref{prop:dec3} suffice to prove the desired estimate in Proposition~\ref{prop:bs-quad}.

Our main tools will be the bilinear estimate Proposition~\ref{prop:bilinear} and the decay estimates in Lemma~\ref{lem:decay}.  For convenience, we record some particular consequences of Lemma~\ref{lem:decay} here.  We let $t_0$ and $I$ be as in Proposition~\ref{prop:bs-quad}, that is, $I=[t_0,1]$ if $t_0<1$ and $I=[t_0,2t_0]$ otherwise.  Then for $2\leq r\leq 6$, we have
\begin{equation}\label{F2}
\|v_N\|_{L_t^\infty L_x^r(I\times\R^3)} \lesssim N^{\frac12-\frac3r}\langle t_0\rangle^{\frac3r-\frac32}\|v\|_{Z(I)}.
\end{equation}
Indeed, \eqref{F2} follows from interpolation between $L_x^2$ and $L_x^6$, Lemma~\ref{lem:decay}, and  the $L_t^\infty \dot H_x^1$-control of $v$.  We apply \eqref{F2} most often with $r=6-$, i.e.
\[
\|v_N\|_{L_t^\infty L_x^{6-}(I\times\R^3)} \lesssim N^{0-}\langle t_0\rangle^{-1+}\|v\|_{Z(I)}.
\]
We also use the following consequence of Bernstein and Lemma~\ref{lem:decay}:
\[
\|v_N\|_{L_t^\infty L_x^6(I\times\R^3)} \lesssim N^{0+}\langle t_0\rangle^{-1+}\|v\|_{Z(I)}.
\]

\begin{remark} As in Proposition~\ref{prop:bs-quad} we work on dyadic intervals, when we apply Proposition~\ref{prop:bilinear} with a dual Strichartz norm $L_t^\alpha L_x^\beta$, we always have the option to bound the contribution of all time integrals by $\langle t_0\rangle^{1/\alpha}$.
\end{remark}

\subsection{Estimation of time non-resonant terms}\label{est-tnr}  To estimate the contribution of time non-resonant regions to \eqref{bs-quad}, we need to bound terms of the form \eqref{tnr} and \eqref{tnr2}, which involve the multiplier $b^T$ (cf. \eqref{bT}).  To estimate these terms, we will use the bilinear estimate Proposition~\ref{prop:bilinear}.

By virtue of symmetries inherent in the terms that we need to estimate and in light of the multiplier bounds in Propositions~\ref{prop:dec1}, \ref{prop:dec2}, and \ref{prop:dec3}, we see that we may reduce to the case $N_2\leq N_1$; moreover, we see that it is natural to cover this with three regimes: (i) $N_2\lesssim N_1\lesssim 1$, (ii) $1\lesssim N_1 \sim N_2$, and (iii) $1\vee N_2 \lesssim N_1$.  Of course, the second regime is contained in the third, but we have to deal with a different multiplier bound in this particular case (cf. Regions 2 and 3 in Proposition~\ref{prop:dec3}).  We record here the bounds we must contend with in each regime:
\begin{equation}\label{tnr-worst}
\op{b^T} \lesssim \begin{cases} N_1^{-1} N_2^{-\frac12} & N_2\lesssim N_1 \lesssim 1, \\
1 & 1\lesssim N_1\sim N_2, \\
N_1^{-1} & 1\vee N_2\lesssim N_1.
\end{cases}
\end{equation}

\textbf{1.} We begin by estimating terms of the form \eqref{tnr}.  We can treat the cases $1\lesssim N_1 \sim N_2$ and $1\vee N_2 \lesssim 1$ together, using the worse multiplier bound of $1$ in \eqref{tnr-worst}.  Using the dual Strichartz norm $L_t^1 L_x^2$ with Proposition~\ref{prop:bilinear}, we estimate these terms via Bernstein, Lemma~\ref{lem:decay}, and \eqref{F2}:
\begin{align*}
\sum_{1\vee N_2 \leq N_1} &\langle t_0\rangle\|v_{N_1}\|_{L_t^\infty L_x^{3}} \|v_{N_2}\|_{L_t^\infty L_x^{6}} \\
& \lesssim \sum_{1\vee N_2\lesssim N_1} \langle t_0\rangle^{-\frac12+}N_2^{0+}N_1^{-\frac12} \|v\|_{Z(I)}^2 \lesssim \langle t_0\rangle^{-\frac12+} \|v\|_{Z(I)}^2,
\end{align*}
which is acceptable.

\textbf{2.} We estimate terms of the form \eqref{tnr} in the regime $N_2\lesssim N_1 \lesssim 1$ using Proposition~\ref{prop:bilinear} with the dual Strichartz norm $L_t^1 L_x^2$ and Lemma~\ref{lem:decay}:
\begin{align*}
\sum_{N_2\lesssim N_1 \lesssim 1}& \langle t_0\rangle N_1^{-1} N_2^{-\frac12} \|v_{N_1}\|_{L_t^\infty L_x^3} \|v_{N_2}\|_{L_t^\infty L_x^6} \\
& \lesssim \sum_{N_2\lesssim N_1 \lesssim 1} \langle t_0\rangle N_2^{\frac12} \|U^{-1} v\|_{L_t^\infty L_x^3} \|U^{-1} v\|_{L_t^\infty L_x^6} \lesssim \langle t_0\rangle^{-\frac16}\|v\|_{Z(I)}^2,
\end{align*}
which is acceptable.

\textbf{3.} We next consider terms of the form \eqref{tnr2} in the regime $1\lesssim N_1 \sim N_2$.  Note that since $N_1\sim N_2$, it suffices to treat terms in which $P_{N_2}$ falls on the nonlinearity.

Using Proposition~\ref{prop:bilinear} with the dual Strichartz norm $L_t^2 L_x^{6/5}$, \eqref{tnr-worst}, \eqref{F2}, and Lemma~\ref{lem:decay}, we may bound the contribution of the quadratic part of $N_v(u)$ by
\begin{align*}
\sum_{1 \lesssim N_1 \sim N_2}& \langle t_0\rangle^{\frac32} \|v_{N_1}\|_{L_t^\infty L_x^{6-}} \|u^2\|_{L_t^{\infty} L_x^{\frac32+}} \\
& \lesssim \sum_{1\lesssim N_1 \sim N_2} \langle t_0\rangle^{\frac12+} N_1^{0-} \|u\|_{L_t^\infty L_x^{3+}}^2 \|v\|_{Z(I)} \lesssim \langle t_0\rangle^{-\frac{5}{18}+} \|v\|_{Z(I)}^3.
\end{align*}
We estimate the contribution of cubic and higher terms via \eqref{tnr-worst}, Bernstein, Lemma~\ref{lem:decay}, and \eqref{dual-est}: for $k\in\{3,4,5\}$,
\begin{align*}
\sum_{1\lesssim N_1 \sim N_2} & \langle t_0\rangle \|v_{N_1}\|_{L_t^\infty L_x^{\infty}} \|P_{N_2}(u^k)\|_{L_t^2 L_x^{\frac65}} \\
& \lesssim \sum_{1\lesssim N_1 \sim N_2} \langle t_0\rangle N_1^{\frac12}N_2^{-1} \|v\|_{L_t^\infty L_x^6} \|\nabla(u^k)\|_{L_t^2 L_x^{\frac65}} \lesssim \langle t_0\rangle^{-\frac{19}{18}} \|v\|_{Z(I)}^{k+1},
\end{align*}
which is acceptable.

\textbf{4.} We next estimate terms of the form \eqref{tnr2} in the regime $1\vee N_2 \lesssim N_1$.

We first consider the case when $P_{N_2}$ falls on the nonlinearity.  To estimate the contribution of the quadratic terms, we use Proposition~\ref{prop:bilinear} with the dual Strichartz norm $L_t^1 L_x^2$, \eqref{tnr-worst}, Bernstein, and Lemma~\ref{lem:decay}:
\begin{align*}
\sum_{1\vee N_2 \lesssim N_1} & \langle t_0\rangle^2 N_1^{-1}\|v_{N_1}\|_{L_t^\infty L_x^{6-}} \|P_{N_2}(u^2)\|_{L_t^\infty L_x^{3+}} \\
& \lesssim\sum_{1\vee N_2 \lesssim N_1} \langle t_0\rangle^{1+} N_1^{-1} N_2^{0+} \|u\|_{L_t^\infty L_x^6}^2 \|v\|_{Z(I)}\lesssim \langle t_0\rangle^{-\frac{5}{9}+} \|v\|_{Z(I)}^3,
\end{align*}
which is acceptable.  We next consider the contribution of the cubic and higher terms.  We use Proposition~\ref{prop:bilinear} with the dual Strichartz norm $L_t^2 L_x^{6/5}$, \eqref{tnr-worst}, Bernstein, Lemma~\ref{lem:decay}, and \eqref{dual-est}: for $k\in\{3,4,5\}$, we bound these terms by
\begin{align*}
\sum_{1\vee N_2 \lesssim N_1} &\langle t_0\rangle N_1^{-1} \|v_{N_1}\|_{L_t^\infty L_x^{\infty-}} \|P_{N_2}(u^k)\|_{L_t^2 L_x^{\frac65+}} \\
& \lesssim \sum_{1\vee N_2 \lesssim N_1} \langle t_0\rangle N_1^{-\frac12-}N_2^{0+} \|v\|_{L_t^\infty L_x^6} \|u^k\|_{L_t^2 L_x^{\frac65}}
\lesssim  \langle t_0\rangle^{-\frac{19}{18}} \|v\|_{Z(I)}^{k+1},
\end{align*}
which is acceptable.

We turn to the case when $P_{N_1}$ lands on the nonlinearity.  For the quadratic terms, we use Proposition~\ref{prop:bilinear} with the dual Strichartz norm $L_t^1 L_x^2$.  We get the contribution
\begin{align*}
 \sum_{1\vee N_2 \lesssim N_1}&\langle t_0\rangle^2 N_1^{-1} \|v_{N_2}\|_{L_t^\infty L_x^{6+}}\|u^2\|_{L_t^\infty L_x^{3-}} \\
 & \lesssim  \sum_{1\vee N_2 \lesssim N_1}\langle t_0\rangle^2 N_1^{-1} N_2^{0+}\|v\|_{L_t^\infty L_x^6} \|u\|_{L_t^\infty L_x^{6-}}^2\lesssim \langle t_0\rangle^{-\frac59+}\|v\|_{Z(I)}^3,
\end{align*}
which is acceptable.  For the cubic and higher terms, we instead use $L_t^2 L_x^{6/5}$ and rely on \eqref{dual-est} and Bernstein: for $k\in\{3,4,5\}$ we get the contribution
\begin{align*}
\sum_{1\vee N_2 \lesssim N_1}&\langle t_0\rangle N_1^{-1}\|P_{N_1}(u^k)\|_{L_t^2 L_x^{\frac65}} \|v_{N_2}\|_{L_t^\infty L_x^{\infty}} \\
& \lesssim \sum_{1\vee N_2 \lesssim N_1}\langle t_0\rangle N_1^{-1}N_2^{\frac12} \|u^k\|_{L_t^2 L_x^{\frac65}} \|v\|_{L_t^\infty L_x^6} \lesssim
\langle t_0\rangle^{-\frac{19}{18}} \|v\|_{Z(I)}^{k+1},
\end{align*}
which is acceptable.

\textbf{5.} Next, we estimate terms of the form \eqref{tnr2} in the regime $N_2 \lesssim N_1\lesssim 1$.  To begin, we consider only the contribution of the cubic and higher terms.

We first consider the case when $P_{N_2}$ falls on the nonlinearity.  We use Proposition~\ref{prop:bilinear} with the dual Strichartz norm $L_t^2 L_x^{6/5}$, \eqref{dual-est}, Bernstein, and Lemma~\ref{lem:decay} to obtain for $k\in\{3,4,5\}$ the contribution
\begin{align*}
\sum_{N_2\lesssim N_1 \lesssim 1}& \langle t_0\rangle N_1^{-1} N_2^{-\frac12} \|P_{N_1}v\|_{L_t^\infty L_x^{6-}} \|P_{N_2}(u^k)\|_{L_t^2 L_x^{\frac32+}} \\
& \lesssim \sum_{N_2\lesssim N_1 \lesssim 1}  \langle t_0\rangle N_2^{0+} \|U^{-1}v\|_{L_t^\infty L_x^{6-}} \|u^k\|_{L_t^2 L_x^{\frac65}}  \lesssim \langle t_0\rangle^{-\frac{5}{6}+} \|v\|_{Z(I)}^{k+1},
\end{align*}
which is acceptable.

Next, we consider the case when $P_{N_1}$ falls on the nonlinearity. We again use Proposition~\ref{prop:bilinear} and dual Strichartz norm $L_t^2 L_x^{6/5}$, \eqref{dual-est}, Bernstein, and Lemma~\ref{lem:decay} to obtain for $k\in\{3,4,5\}$ the contribution
\begin{align*}
\sum_{N_2\lesssim N_1 \lesssim 1}& \langle t_0\rangle N_1^{-1} N_2^{-\frac12}\|v_{N_2}\|_{L_t^\infty L_x^{6-}} \|P_{N_1}(u^k)\|_{L_t^2 L_x^{\frac32+}} \\
& \lesssim \sum_{N_2\lesssim N_1 \lesssim 1} \langle t_0\rangle N_2^{\frac12} N_1^{-\frac12+} \|U^{-1}v\|_{L_t^\infty L_x^{6-}} \|u^k\|_{L_t^2 L_x^{\frac65}} \lesssim \langle t_0\rangle^{-\frac56+}\|v\|_{Z(I)}^{k+1},
\end{align*}
which is acceptable.

\textbf{6.}  Finally, we treat the contribution of the quadratic terms in the nonlinearity to \eqref{tnr2} when $N_2\lesssim N_1\lesssim 1$.  We will use Proposition~\ref{prop:bilinear} with the dual Strichartz norm $L_t^1 L_x^2$.

Recalling that $u_1=v_1$ and $u_2=U^{-1}v_2$, we note that the quadratic terms in \eqref{E:cqv} are of the form
\[
\text{\O}(U(u^2) + uv).
\]

If $P_{N_2}$ lands on the $U(u^2)$ term, we estimate via Lemma~\ref{lem:decay}:
\begin{align*}
\sum_{N_2\lesssim N_1\lesssim 1} & \langle t_0\rangle^2 N_2^{-\frac12}N_1^{-1} \|v_{N_1}\|_{L_t^\infty L_x^6} \|UP_{N_2}(u^2)\|_{L_t^\infty L_x^3} \\
& \lesssim \sum_{N_2\lesssim N_1\lesssim 1}\langle t_0\rangle^2 N_2^{\frac12}\|U^{-1}v\|_{L_t^\infty L_x^6} \|u\|_{L_t^\infty L_x^6}^2 \lesssim \langle t_0\rangle^{-\frac13}\|v\|_{Z(I)}^3,
\end{align*}
which is acceptable.  If $P_{N_1}$ lands on the $U(u^2)$ term, we get the contribution
\begin{align*}
\sum_{N_2\lesssim N_1\lesssim 1}& \langle t_0\rangle^2 N_2^{-\frac12} N_1^{-1} \|v_{N_2}\|_{L_t^\infty L_x^6} \|UP_{N_1}(u^2)\|_{L_t^\infty L_x^3} \\
& \lesssim \sum_{N_2\lesssim N_1 \lesssim 1} \langle t_0\rangle^2 N_2^{\frac12} \|U^{-1}v\|_{L_t^\infty L_x^6} \|u\|_{L_t^\infty L_x^6}^2 \lesssim \langle t_0\rangle^{-\frac13} \|v\|_{Z(I)}^3,
\end{align*}
which is acceptable.

If $P_{N_2}$ lands on the $uv$ term, then we estimate by Bernstein and Lemma~\ref{lem:decay}:
\begin{align*}
\sum_{N_2\lesssim N_1\lesssim 1}&\langle t_0\rangle^2 N_1^{-1} N_2^{-\frac12} \|v_{N_1}\|_{L_t^\infty L_x^6} \|P_{N_2}(uv)\|_{L_t^\infty L_x^3} \\
& \lesssim \sum_{N_2\lesssim N_1\lesssim 1}\langle t_0\rangle^2 N_2^{0+}\|U^{-1} v\|_{L_t^\infty L_x^6} \|uv\|_{L_t^\infty L_x^{2-}} \\
& \lesssim \sum_{N_2\lesssim N_1\lesssim 1}\langle t_0\rangle^2 N_2^{0+}\|U^{-1} v\|_{L_t^\infty L_x^6} \|u\|_{L_t^\infty L_x^{3-}}\|v\|_{L_t^\infty L_x^6} \lesssim \langle t_0\rangle^{-\frac16+}\|v\|_{Z(I)}^3,
\end{align*}
which is acceptable.

Finally, if $P_{N_1}$ lands on the $uv$ term, we get the contribution
\begin{align*}
\sum_{N_2\lesssim N_1\lesssim 1}&\langle t_0\rangle^2 N_1^{-1} N_2^{-\frac12} \|v_{N_2}\|_{L_t^\infty L_x^6}\|P_{N_1}(uv)\|_{L_t^\infty L_x^3} \\
& \lesssim \sum_{N_2\lesssim N_1\lesssim 1}\langle t_0\rangle^2 N_2^{\frac12} N_1^{-\frac12+}\|U^{-1}v\|_{L_t^\infty L_x^6} \|uv\|_{L_t^\infty L_x^{2-}}  \lesssim \langle t_0\rangle^{-\frac16+}\|v\|_{Z(I)}^3,
\end{align*}
which is acceptable.  This completes the estimation of the contribution of the time non-resonant regions.
\subsection{Estimation of space non-resonant terms}\label{est-snr}
To estimate the contribution of the space non-resonant regions to \eqref{bs-quad}, we need to bound terms of the form \eqref{snr1}, involving the $b^X$ multiplier, and \eqref{snr2}, involving the $\tilde b^X$ multiplier. Again, we will rely on Proposition~\ref{prop:bilinear}.

Note that we only had space non-resonant regions for the $v^2$ and $|v|^2$ nonlinearities; the multiplier bounds we established appear in Propositions~\ref{prop:dec2} and \ref{prop:dec3}.  In all cases of space non-resonance, we had $N_2\lesssim N_1$. Examining the multiplier bounds, it is natural to split into two cases: $N_2\lesssim N_1 \lesssim 1$ and $1\vee N_2 \lesssim N_1$. We record here the worst bounds for each multiplier in these two regimes:
\begin{equation}\label{snr-worst}
\begin{aligned}
&\op{b^X} \lesssim \begin{cases} N_1^{-\frac12} & N_2 \lesssim N_1 \lesssim 1, \\
N_2 N_1^{-1} & 1\vee N_2 \lesssim N_1.\end{cases}, \\
&\op{\tilde b^X} \lesssim \begin{cases} N_2^{-1} N_1^{-\frac12} & N_2\lesssim N_1 \lesssim 1, \\
N_1^{-1} & 1\vee N_2\lesssim N_1.\end{cases}
\end{aligned}
\end{equation}

We first consider terms of the form \eqref{snr1} in the case $N_2 \lesssim N_1\lesssim 1$.  We will use Proposition~\ref{prop:bilinear} with the dual Strichartz norm $L_t^{4/3} L_x^{3/2}$.  We estimate such terms via \eqref{snr-worst}, Bernstein, and Lemma~\ref{lem:decay}:
\begin{align*}
\sum_{N_2\lesssim N_1\lesssim 1}&\langle t_0\rangle^{\frac34} N_1^{-\frac12}\bigl\{ \|P_{N_1}Jv\|_{L_t^\infty L_x^2} \|v_{N_2}\|_{L_t^\infty L_x^6} + \|P_{N_2}Jv\|_{L_t^\infty L_x^{2+}}\|v_{N_1}\|_{L_t^\infty L_x^{6-}}\bigr\} \\
& \lesssim \sum_{N_2\lesssim N_1\lesssim 1}\langle t_0\rangle^{\frac34} N_1^{-\frac12}\bigl\{ N_2\langle t_0\rangle^{-\frac79} + N_2^{0+}N_1^{1-}\langle t_0\rangle^{-\frac79+}\bigr\} \|v\|_{Z(I)}^2 \\
& \lesssim \langle t_0\rangle^{-\frac{1}{36}+}\|v\|_{Z(I)}^2,
\end{align*}
which is acceptable.

We next consider terms of the form \eqref{snr1} in the case $1\vee N_2\lesssim N_1$.  We will use Proposition~\ref{prop:bilinear} with the dual Strichartz norm $L_t^{\frac43+}L_x^{\frac32-}$.  We estimate such terms via \eqref{snr-worst} and \eqref{F2}:
\begin{align*}
\sum_{N_2\vee 1 \lesssim N_1}&\langle t_0\rangle^{\frac34-}\tfrac{N_2}{N_1}\bigl\{\|P_{N_1}Jv\|_{L_t^\infty L_x^2} \|v_{N_2}\|_{L_t^\infty L_x^{6-}} + \|P_{N_2} Jv\|_{L_t^\infty L_x^2}\|v_{N_1}\|_{L_t^\infty L_x^{6-}}\bigr\} \\
& \lesssim \sum_{N_2\vee 1\lesssim N_1} \langle t_0\rangle^{-\frac14+}\bigl\{N_2^{1-}N_1^{-1} + N_2 N_1^{-1-}\bigr\}\|v\|_{Z(I)}^2 \lesssim \langle t_0\rangle^{-\frac14+}\|v\|_{Z(I)}^2,
\end{align*}
which is acceptable.

We turn to terms of the form \eqref{snr2} in the case $N_2\lesssim N_1\lesssim 1$. We will use Proposition~\ref{prop:bilinear} with the dual Strichartz norm $L_t^1 L_x^2$. We estimate such terms via \eqref{snr-worst},
using \eqref{FL6} for $v_{N_2}$ and Lemma~\ref{lem:decay}(iii) for $v_{N_1}$:
\begin{align*}
\sum_{N_2\lesssim N_1\lesssim 1}& \langle t_0\rangle N_2^{-1} N_1^{-\frac12}\|v_{N_1}\|_{L_t^\infty L_x^3} \|v_{N_2}\|_{L_t^\infty L_x^6} \\
&\lesssim \sum_{N_2\lesssim N_1\lesssim 1} \langle t_0\rangle^{-\frac{1}{18}} N_2^{\frac13} N_1^{\frac12} \|v\|_{Z(I)}^2\lesssim \langle t_0\rangle^{-\frac{1}{18}}\|v\|_{Z(I)}^2,
\end{align*}
which is acceptable.

Finally, we consider terms of the form \eqref{snr2} in the case $1\vee N_2 \lesssim N_1$.  We will use Proposition~\ref{prop:bilinear} with the dual Strichartz norm $L_t^1L_x^2$.  We estimate such terms via \eqref{snr-worst}, Lemma~\ref{lem:decay}, and Bernstein:
\begin{align*}
\sum_{1\vee N_2 \lesssim N_1}&  \langle t_0\rangle N_1^{-1} \|v_{N_1}\|_{L_t^\infty L_x^6} \|v_{N_2}\|_{L_t^\infty L_x^3} \\
& \lesssim \sum_{1\vee N_2 \lesssim N_1} N_1^{-1} N_2^{0+} \|v_{N_2}\|_{L_t^\infty L_x^{3-}}\|v\|_{Z(I)} \lesssim \langle t_0\rangle^{-\frac12+}\|v\|_{Z(I)}^2,
\end{align*}
which is acceptable.  This completes the estimation of the contribution of space non-resonant regions to \eqref{bs-quad}.

\subsection{Estimation of angular non-resonant terms}\label{est-anr}

To estimate the contribution of angular non-resonant regions to \eqref{bs-quad}, we need to bound the terms \eqref{anr1}, \eqref{anr2}, \eqref{anr3}, and \eqref{anr4}.

There is only one angular non-resonant region in our decompositions.  The relevant multiplier bounds appear in Proposition~\ref{prop:dec3}; in particular, the bounds are all $\lesssim 1$.  Note also that in the angular non-resonant region, we have the condition $1\lesssim N_1 \sim N_2$.

To estimate the term \eqref{anr1}, we use Proposition~\ref{prop:bilinear} with the dual Strichartz norm $L_t^1 L_x^2$ and \eqref{F2}:
\begin{align*}
\sum_{1\lesssim N_1\sim N_2}&\langle t_0\rangle \|v_{N_1}\|_{L_t^\infty L_x^4} \|v_{N_2}\|_{L_t^\infty L_x^4} \\
& \lesssim \sum_{1\lesssim N_1 \sim N_2} \langle t_0\rangle^{-\frac12} N_1^{-\frac14}N_2^{-\frac14}\|v\|_{Z(I)}^2 \lesssim \langle t_0\rangle^{-\frac12}\|v\|_{Z(I)}^2.
\end{align*}
which is acceptable.

We next consider \eqref{anr2} and \eqref{anr3}.  As $N_1\sim N_2$, it suffices to treat \eqref{anr2}.  Using Proposition~\ref{prop:bilinear} with the dual Strichartz norm $L_t^{\frac43+}L_x^{\frac32-}$ and \eqref{F2}, we get the contribution
\begin{align*}
\sum_{1\lesssim N_1\sim N_2} & \langle t_0\rangle^{\frac34-}\|P_{N_1}(x\times\nabla)v\|_{L_t^\infty L_x^2} \|v_{N_2}\|_{L_t^\infty L_x^{6-}} \\
&\lesssim \sum_{1\lesssim N_1 \sim N_2} \langle t_0\rangle^{-\frac14+}N_2^{0-}\|v\|_{Z(I)}^2 \lesssim \langle t_0\rangle^{-\frac14+}\|v\|_{Z(I)}^2,
\end{align*}
which is acceptable.

Finally, we estimate \eqref{anr4}.  Arguing as above, we estimate this term by
\[
\sum_{1\leq N_1\sim N_2}\langle t_0\rangle^{\frac34-}\|P_{N_1}Jv\|_{L_t^\infty L_x^2} \|P_{N_2}v\|_{L_t^\infty L_x^{6-}} \lesssim \langle t_0\rangle^{-\frac14+}\|v\|_{Z(I)}^2,
\]
which is acceptable.  This completes the estimation of the contribution of the angular non-resonant regions to \eqref{bs-quad}, and hence the proof of Proposition~\ref{prop:bs-quad}.

\section{Proof of the main result}\label{S:proof}

In this section, we finally prove the main result, Theorem~\ref{T:main}.  As discussed in the Introduction, \cite{KMV} already guarantees that we have a unique global solution $v$ to \eqref{E:cqv} for initial data $v_0$ as in Theorem~\ref{T:main}.  Thus, it remains to show that the solution scatters.  We will prove scattering forward in time.

Collecting the results of Proposition~\ref{prop:str}, Proposition~\ref{prop:weighted1}, Proposition~\ref{prop:weighted2}, and Proposition~\ref{prop:bs-quad}, we arrive at the following \emph{a priori} estimate:
\[
\|v\|_{Z(I)} \lesssim \|v(t_0)\|_{X} + \langle t_0\rangle^{-\eps} \sum_{k=2}^6 \|v\|_{Z(I)}^k
\]
for some $\eps>0$, where $t_0=\inf I$.  Using this, one can show that if $\|v_0\|_{X}$ is sufficiently small, then the solution obeys
\begin{equation}\label{Zbd}
\|v\|_{Z([0,\infty)]} \lesssim 1.
\end{equation}

To prove scattering, it suffices to show that $\{e^{itH}v(t)\}$ is Cauchy in $H_x^1$, $\{(x\times\nabla)e^{itH}v(t)\}$ is Cauchy in $L_x^2$, and $\{xe^{itH}v(t)\}$ is Cauchy in $L_x^2$ as $t\to\infty$.

We first use the Duhamel formula \eqref{duhamel1}, Proposition~\ref{prop:str}, and \eqref{Zbd} to estimate
\begin{align*}
\| e^{itH}v(t) - e^{isH}v(s) \|_{H_x^1} &=\biggl\| \int_s^t e^{i\tau H}N_v(u(\tau))\,d\tau\biggr\|_{H_x^1}\\
& \lesssim s^{-\eps}\sum_{k=2}^5\|v\|_{Z([0,\infty))}^k \lesssim s^{-\eps}
\end{align*}
for $t>s>1$.  Thus $\{e^{itH}v(t)\}$ is Cauchy in $H_x^1$ as $t\to\infty$.  Arguing similarly (using Proposition~\ref{prop:str} and recalling \eqref{def:S}) also shows that $\{(x\times\nabla)e^{itH}v(t)\}$ is Cauchy in $L_x^2$ as $t\to\infty$. 

To show that $\{xe^{itH}v(t)\}$ is Cauchy in $L_x^2$, we use the normal form of the equation; that is, we use the Duhamel formula with \eqref{nf-duh}.  Then, for $t>s>1$, we use Proposition~\ref{prop:weighted1}, Proposition~\ref{prop:weighted2}, Proposition~\ref{prop:bs-quad}, and \eqref{Zbd} to estimate
\begin{align*}
\|&x[e^{itH}v(t)-e^{isH}v(s)]\|_{L_x^2} \\
& \lesssim \sum_{k=2}^6 \biggl\| x\int_s^t e^{i\tau H}\N_k(u(\tau))\,d\tau\biggr\|_{L_x^2} \\
& \quad + \|xe^{isH}B[v_1(s),v_1(s)]\|_{L_x^2}+\|xe^{isH}B[U^{-1}v_2(s),U^{-1}v_2(s)]\|_{L_x^2} \\
& \quad + \|xe^{itH}B[v_1(t),v_1(t)]\|_{L_x^2}+\|xe^{itH}B[U^{-1}v_2(t),U^{-1}v_2(t)]\|_{L_x^2} \\
& \lesssim s^{-\eps}\sum_{k=2}^6 \|v\|_{Z([0,\infty)}^k \lesssim s^{-\eps}.
\end{align*}
Thus $\{xe^{itH}v(t)\}$ is Cauchy in $L_x^2$ as $t\to\infty$.

Finally, note that the argument above gives a rate of convergence of $t^{-\eps}$ to the scattering state.  This completes the proof of Theorem~\ref{T:main}.

\end{document}